\documentclass[11pt]{article}
\usepackage{fancyhdr,amsmath, graphicx, xcolor, color, amsfonts, layout, mathrsfs, amssymb, amsthm, wrapfig, manfnt,mathtools}
\usepackage[english]{babel}
\usepackage{wasysym}
\usepackage[hyphens]{url}
\usepackage{breakurl}
\usepackage[pagebackref=true,colorlinks,allcolors=blue,breaklinks]{hyperref}
\usepackage{scalefnt}
\usepackage[normalem]{ulem}

\usepackage{pgf}
\usepackage{tikz}
\usepackage{tikz-qtree}
\usepackage{forest}
\usetikzlibrary{shadows,trees,arrows}
\usetikzlibrary{shapes,positioning,decorations.pathreplacing} 

\usepackage{natbib}
\usepackage{verbatim}
\usepackage{array}

\usepackage{mathrsfs}

\newtheorem*{theo*}{Theorem}
\newtheorem{defi}{Definition}[section]
\newtheorem{rem}[defi]{Remark}
\newtheorem{lem}[defi]{Lemma}
\newtheorem{theo}[defi]{Theorem}
\newtheorem{cor}[defi]{Corollary}
\newtheorem{pro}[defi]{Proposition}
\newtheorem{exa}[defi]{Example}
\newtheorem{conj}[defi]{Conjecture}

\newtheoremstyle{remarques}{\medskipamount}{\medskipamount}{}
                        {0pt}{\bfseries}{.}{ }{}
\theoremstyle{remarques}

\newcommand{\cercle}[1]{\textcircled{\raisebox{-0.8pt}{\small#1}}}
\DeclareMathOperator{\Espe}{\bf E}
\renewcommand{\geq}{\geqslant}
\DeclareMathOperator{\indic}{{\bf 1}}
\renewcommand{\le}{\leqslant}
\renewcommand{\leq}{\leqslant}
\DeclareMathOperator{\pref}{cont}
\DeclareMathOperator{\Proba}{\bf P}
\newcommand{\tooSous}[1]{\smash{\mathop{\longrightarrow}\limits _{#1}}}

\newcounter{compteurExo}[section]

\DeclareMathOperator{\casc}{casc}

\newcommand{\set}[1]{\left\{ #1\right\}}

\newcommand{\g}[1]{\mathbb #1}

\newcommand{\rond}[1]{{\mathscr #1}}

\newcommand{\Fdebut}{\color{brown}}
\newcommand{\Ffin}{\color{black}}

\newcommand{\Ndebut}{\color{violet}}
\newcommand{\Nfin}{\color{black}}

\newcommand\1{\leavevmode\hbox{\rm \small1\kern-0.35em\normalsize1}}
\newcommand\ind[1]{\1_{\{#1\}}}

\usepackage[draft]{fixme}
\fxusetheme{color}
\FXRegisterAuthor{a}{aa}{A}
\FXRegisterAuthor{b}{ab}{B}
\FXRegisterAuthor{pe}{ape}{P}
\FXRegisterAuthor{f}{af}{F}
\FXRegisterAuthor{n}{an}{N}

\mathtoolsset{showonlyrefs,showmanualtags}
\begin{document}

 \title{Variable Length Memory Chains:\\characterization of stationary probability measures}
 \author{Peggy C\'enac\thanks{Universit\'e de Bourgogne,
Institut de Math\'ematiques de Bourgogne,
IMB UMR 5584 CNRS,
9 rue Alain Savary - BP 47870, 21078 DIJON CEDEX, France.
}, Brigitte Chauvin\thanks{Universit\'e Paris-Saclay, UVSQ, CNRS UMR 8100, Laboratoire de Math\'ematiques de Versailles,
78000 Versailles, France
}, Camille No\^us\thanks{Cogitamus Laboratory},
Fr\'ed\'eric Paccaut\thanks{Laboratoire Ami\'enois de Math\'ematique Fondamentale et Appliqu\'ee,
CNRS UMR 7352,
Universit\'e  de Picardie Jules Verne,
33 rue Saint-Leu, 80039 Amiens, France.
} \ 
and Nicolas Pouyanne\footnotemark[2]}
\maketitle
\begin{abstract}
Variable Length Memory Chains (VLMC), which are generalizations of finite order Markov chains, turn out to be an essential tool to modelize random sequences in many domains, as well as an interesting object in contemporary probability theory.
The question of the existence of stationary probability measures leads us to introduce a key combinatorial structure for words produced by a VLMC:
the \emph{Longest Internal Suffix}.
This notion allows us to state a necessary and sufficient condition for a general VLMC to admit a unique invariant probability measure.

This condition turns out to get a much simpler form for a subclass of VLMC: the \emph{stable} VLMC.
This natural subclass, unlike the general case, enjoys a renewal property.
Namely, a stable VLMC induces a semi-Markov chain on an at most countable state space.
Unfortunately, this discrete time renewal process does not contain the whole information of the VLMC, preventing the study of a stable VLMC to be reduced to the study of its induced semi-Markov chain.
For a subclass of stable VLMC, the convergence in distribution of a VLMC towards its stationary probability measure is established.

Finally, finite state space semi-Markov chains turn out to be very special stable VLMC, shedding some new light on their limit distributions.
\end{abstract}


\tableofcontents

\section{Introduction}
\label{sec:intro}

In a Variable Length Memory Chain (VLMC), unlike fixed order Markov chains, the probability to predict the next symbol depends on a possibly unbounded part of the past, the length of which depends on the past itself. 
These relevant parts of pasts are called \emph{contexts}. 
They are stored in a \emph{context tree}. 
With each context is associated a probability distribution prescribing the conditional probability of the next symbol, given this context. 

\vskip 5pt
In this paper we obtain some necessary and sufficient conditions to ensure existence and uniqueness of a stationary probability measure for a general VLMC.

\vskip 5pt
Pending a complete presentation in Section~\ref{sec:def}, let us now introduce a few objects, notably the combinatorial notion of alpha-LIS (LIS for \emph{Longest Internal Suffix}), on which our main result is based.
Let $\rond A$ be a finite set, called the alphabet.
A so-called context tree is a saturated tree $\rond T$ on this alphabet, \emph{i.e.} a tree such that each node has $0$ or $\#\rond A$ children.
The leaves and the infinite branches of $\rond T$ are called contexts.
The set of contexts, supposed to be at most countable, is denoted by $\rond C$.

\vskip 5pt
To each context $c\in\rond C$ is attached a probability distribution $q_{c}$ on $\rond A$.
Endowed with this probabilistic structure, such a tree is named a probabilised context tree.
Let $\rond R$ be the set of right-infinite words on the alphabet $\rond A$.
The related VLMC is defined as the $\rond R$-valued Markov chain $\left( U_n\right) _{n\geq 0}$ whose transitions are given by
\begin{equation}
\forall n\geq 0,~\forall\alpha\in\rond A,~\Proba\left( U_{n+1}=\alpha U_n | U_n\right)
=q_{\pref\left(U_n\right)}\left(\alpha\right),
\end{equation}
where $\pref(u)\in\rond C$ is defined as the only prefix of the right-infinite word $u$  appearing as a context.
See Figure \ref{fig:ex0} for an example of context tree.

\vskip 5pt
If $\pi$ is a probability measure on $\rond R$, asking $\pi$ to be stationary for such a Markov chain $\left( U_n\right) _n$ amounts to saying that, for any finite word $w$ which writes $w=\alpha v$ where $\alpha\in\rond A$ and where $v$ is a non-internal finite word of the context tree,
\begin{equation}
\label{formule:casc0}
\pi\left(w\rond R\right) =q_{\pref (v)}\left(\alpha\right)\pi\left( v\rond R\right).
\end{equation}
In this equality, $w\rond R$ denotes the set of all right-infinite words that begin by $w$.
This formula applies again for $\pi\left( v\rond R\right)$, and so on, and so forth, until... it is not possible anymore, which means that the suffix of ${w}$ is of the form $\alpha s$ where $\alpha\in\rond A$ and $s$ is an internal word of the context tree. This leads to pointing out the following unique decomposition of any finite word $w$:
\[
w=\beta _1\beta _2\dots\beta _{p_w}\alpha _ws_w,
\]
where
 
$\bullet$ $p_w$ is a nonnegative integer and $\beta_i\in \rond A$, for all $i = 1, \dots , p_w$, 

$\bullet$ $s_w$ is the longest internal strict suffix of $w$,

$\bullet$ $\alpha _w\in\rond A$.

\vskip 5pt
In this decomposition, $s_w$ is called the \emph{LIS} of $w$ and $\alpha_w s_w$ the \emph{alpha-LIS} of $w$. 
Consequently, for any stationary measure $\pi$ and for any finite non-empty word $w$, write $ w = v\alpha_w s_w$ where $v$ is a finite word and $\alpha_w s_w$ is the alpha-LIS of $ w$ so that iterating Formula \eqref{formule:casc0} gives
\begin{equation}
\label{formule:casc1}
\pi\left(w\rond R\right)  = \casc (w) \pi\left(\alpha_w s_w\rond R\right),
\end{equation}
where $\casc (w)$, 
the \emph{cascade} of $w$, is defined as
\[
\casc (w)=\prod _{1\leq k\leq p_w}q_{\pref\left(\beta _{k+1}\dots\beta _{p_w}\alpha _ws_w\right)}(\beta _k).
\]

Elementary arguments on measures show thus that any stationary probability measure on $\rond R$ is determined by its value on the cylinders based on alpha-LIS of \emph{contexts}.
Denote by $\rond S$ the set of alpha-LIS of finite contexts. This set is at most countable.
Using Formulas \eqref{formule:casc0} and \eqref{formule:casc1}, as developed in the proof of Theorem~\ref{fQbij}, it turns out that, whenever $\pi$ is stationary, all the $\pi\left( \alpha s \rond R\right)$, for $\alpha s\in\rond S$ are related by the linear system
\[
\pi\left(\alpha s \rond R\right) = \sum_{\beta t\in\rond S}\pi\left(\beta t\rond R\right) Q_{\beta t, \alpha s},
\]
where the square matrix $Q=\left( Q_{\alpha s,\beta t}\right) _{(\alpha s,\beta t)\in\rond S^2}$ is defined by
\[
Q_{\alpha s,\beta t}
=\sum _{\substack{c\in\rond C^f\\[2pt]c=t\cdots\\[1pt]c=\cdots [\alpha s]}}
\casc\left(\beta c\right).
\]
In this formula, $\rond C^f$ denotes the set of finite contexts, the notation $c = \cdots [\alpha s]$ means that $\alpha s$ is the alpha-LIS of $c$, while $c = t\cdots$ means that $t$ is a prefix of $c$.
In otherwords, $\left( \pi\left(\alpha s \rond R\right) \right)_{\alpha s\in\rond S}$ is a left-fixed vector of the matrix $Q$. 
The study of the matrix $Q$ indexed by the alpha-LIS of contexts is a key tool to characterize a stationary measure for the VLMC.
Our main result, namely Theorem~\ref{fQbij}, has the following weaker version that can be now stated.
\begin{theo*}
Let $(\rond T,q)$ be a probabilised context tree and $U$ the associated VLMC. Assume that  $\forall\alpha\in\rond A$, $\forall c\in\rond C$, $q_c(\alpha)\not=0$.
Then $U$ admits a unique stationary probability measure if and only if the three following points are satisfied:

\begin{itemize}
\item[(i)]
$\forall \alpha s\in\rond S$,  the cascade series $\displaystyle\sum _{c\in\rond C^f,~c=\cdots [\alpha s]}\casc (c)$ converge. The sum is denoted by $\kappa _{\alpha s}$.
\item[(ii)]
The matrix $Q$ admits a unique line of left-fixed vectors.
\item[(iii)]
For any left-fixed vector $\left( v_{\alpha s}\right)_{\alpha s\in\rond S}$ of $Q$,  $\displaystyle\sum _{\alpha s\in\rond S}v_{\alpha s}\kappa _{\alpha s}<+\infty$.
\end{itemize}
\end{theo*}
The state space $\rond R$ of a VLMC is uncountable, placing the question of existence and unicity of its invariant probability measures outside of the well marked out theory of Markov chains on countable state spaces.
Theorem~\ref{fQbij} comes down to searching and studying left-fixed vectors of the at most countable matrix $Q$.

\vskip 5pt
When $\rond S$ is finite, condition (iii) in the previous theorem is automatically satisfied as soon as (i) holds.
Furthermore, in that case, preceding condition (ii) gets a complete answer thanks to finite dimensional linear algebra.
In the very particular case of \emph{stable} context trees (see hereafter for a definition) having a \emph{finite} set of context alpha-LIS, Theorem~\ref{cor:finite} gives a complete characterization of VLMC's that admit stationary probability measures, which reduces to the convergence of the cascade series.

\vskip 5pt
Note that the characterization given in the previous theorem is expressed via the cascades and the probability distributions~$q_c$.
Nevertheless, the role of context alpha-LIS suggests that the \emph{shape} of the context tree matters a lot.

The case of \emph{stable} trees is particularly interesting, Section~\ref{sec:stable} is devoted to this case.
In particular, when a context tree is stable, the corresponding VLMC ends up owning renewal properties, which is not the case for a non-stable VLMC -- see Remark~\ref{rem:pasRenouveau}.

A tree is said \emph{stable} when it is stable by the shift. In other words, for any letter $\alpha\in\rond A$ and for any finite word $w$, if $\alpha w \in \rond T$ then $w \in \rond T$.
See Section~\ref{subsection:defstable} for a complete definition.
In the stable case, the crux of the matter is that the matrix $Q$ is always stochastic and can be interpreted as the transition matrix of some Markov chain
on the set of context alpha-LIS.
Indeed, when a VLMC $(U_n)$ is stable, if one denotes by $Z_n$ the alpha-LIS of $\pref(U_n)$, it turns out that the process $(Z_n)$ is an $\rond S$-valued semi-Markov chain.
This induced semi-Markov chain brings out some renewal times which are the moments $\pref(U_n)$ changes its alpha-LIS.
All this is detailed in Section~\ref{subsec:compareSM}.

It should be noticed that studying a stable VLMC $(U_n)$ is not just about studying the semi-Markov chain $(Z_n)$ mentioned above. 
Indeed, the trajectories of $(U_n)$ cannot be recovered from the trajectories of $(Z_n)$. 
See Remark \ref{rem:semiMarkov}. 
However, it is the properties of the matrix $Q$ detailed in Section \ref{subsection:properties_Q} that provide increasingly simple and manipulable necessary and sufficient condition for existence and unicity of a stationary probability measure for $(U_n)$ in Theorem \ref{th:stable} and Theorem \ref{cor:finite}. The latter theorem also provides the convergence of the distributions of $U_n$ to the stationary probability measure. 

\vskip 5pt
As a final remark, we add in Section \ref{subsec:resultsSM} another link between semi-Markov chains and VLMC: 
it is shown that any semi-Markov chain on a finite state space is a VLMC associated with some particular infinite stable probabilised context tree.
Consequently, one deduces from Theorem \ref{cor:finite} a necessary and sufficient condition for a non-null semi-Markov chain to admit a limit distribution. The same condition already appears in~\cite{barbu/limnios/08} for aperiodic irreducible semi-Markov chains as a sufficient condition.

\vskip 5pt
Throughout the text, without drowning the reader in a multitude of examples of context trees, we chose to present enough cases of context trees that:

- answer natural questions about the different assumptions

- sometimes provide explicit calculations

- illuminate results and proofs.


\vskip 5pt
Let us now indicate a non exhaustive range of domains where Variable Length Memory Chains are commonly used. 
VLMC are random models for character strings.
When they have a finite memory, they have been introduced in \cite{rissanen/83} to perform data compression.
They provide a parsimonious alternative to fixed order Markov chain models, in which the number of parameters to estimate grows exponentially fast with the order; they are also able to capture finer properties of character sequences. When they have infinite memory -- this will be our case of study -- they provide a tractable way to build models which are not finite order Markov chains. 
Furthermore they may be considered as a subclass of ``cha\^ines \`a  liaisons compl\`etes'' (\cite{doeblin/fortet/37}) or ``chains with infinite order'' (\cite{harris/55}).

Variable length memory chains are also a particular case of processes defined by a $g$-function (where the $g$-function is piecewise constant on a countable set of cylinders).
Stationary probability measures for VLMC are $g$-measures.
The question of uniqueness of $g$-measures has been adressed by many authors when the function $g$ is continuous (in this case, the existence is straightforward), see \cite{johansson/oberg/03}, \cite{fernandez/maillard/05}. Recently, interest raised also for the question of existence and uniqueness when $g$ is not continuous, see \cite{gallo/11}, \cite{gallo/garcia/13}, \cite{desantis/piccioni/12} for a perfect simulation point of view and the more ergodic theory flavoured \cite{gallo/paccaut/13} and \cite{ferreira/gallo/paccaut/19}.

VLMC are used in bioinformatics, linguistics or coding theory to modelize how random words grow or to classify words. In bioinformatics, both for protein families and DNA sequences, identifying patterns that have a biological meaning is a crucial issue. Using VLMC as a model enables to quantify the influence of a meaning pattern by giving a transition probability on the following letter of the sequence. In this way, these patterns appear as contexts of a context tree (\cite{bejerano/yona/01}). 
An appropriate model requires to consider possibly unbounded lengths.
In addition, when the context tree is recognised to be a signature of a family (of proteins say), this gives an efficient statistical method to test whether or not two samples belong to the same family (\cite{busch/etc/09}).

Therefore, estimating a context tree is an issue of interest and many authors (statisticians or not, applied or not) stress the fact that the height of the context tree should not be supposed to be bounded. This is the case in \cite{galves/leonardi/08} where the algorithm \texttt{CONTEXT} is used to estimate an unbounded context tree and also in \cite{garivier/leonardi/11}. Furthermore, as explained in \cite{csiszar/talata/06}, the height of the estimated context tree grows with the sample size so that estimating a context tree by assuming \emph{a priori} that its height is bounded is not realistic.

Classical random walks have independent and identically distributed increments. In the literature, Persistent Random Walks refer to random walks having a Markov chain of finite order as an increment process. For such walks, the dynamics of trajectories has a short memory of given length and the random walk itself is not Markovian any more. Recently, as pointed in 
\cite{cenac/chauvin/herrmann/vallois/13, CDLO, cenac:hal-01658494,CCPP4},
persistent random walks can be viewed as Random Walks with increments built from VLMC for an infinite context tree. 

In biology, persistent random walks are one possible model to address the question of anomalous diffusions in cells (see for instance \cite{FedTanZub}). Actually, such random walks are non Markovian, the displacements and the jumping times are correlated.

There is a large literature on constructing efficient estimators of context trees, as well for finite or infinite context trees.
Our point of view is not a statistical one, and we focus here on the probabilistic properties of infinite memory VLMC as random processes, and more specifically on the main property of interest for such processes: existence and uniqueness of a stationary measure.

In Section~\ref{sec:def}, the definitions of a general VLMC, LIS and alpha-LIS of finite words are given, leading to the main theorem (Theorem \ref{fQbij}). 
Section~\ref{sec:stable} is devoted to the stable case, providing a necessary and sufficient condition for the existence and unicity of an invariant probability measure for the VLMC. 
The correspondence with semi-Markov model is detailed.
Proofs are postponed in Section~\ref{sec:proofs}. 
Finally, Section~\ref{sec:open} is devoted to open problems and conjectures. 

\section{Definitions, notations and main results in the general case}
\label{sec:def}

\subsection{Probabilised context trees and VLMC}

In the whole paper, $\rond A$ denotes a finite set having at least two elements, called the \emph{alphabet}.
Its elements are called \emph{letters}.
All main results in the article hold for an arbitrary $\rond A$ but, for readability reasons, the proofs are written taking $\rond A=\set{0,1}$ whenever this assumption can be made without loss of generality.
Let $\rond R$ be the set of \emph{right-infinite words} on the alphabet, written by simple concatenation:
\[
\rond R=\set{\alpha\beta\gamma\cdots :~\alpha,\beta,\gamma\cdots\in\rond A}.
\]
The set of finite \emph{words}, sometimes denoted by $\rond A^*$ in the literature, will be denoted by $\rond W$:
\[
\rond W=\bigcup _{n\in\g N}\rond A^n,
\]
the set $\rond A^0:=\set{\emptyset}$ being reduced to the empty word\footnote{In the whole paper, $\g N=\set{0,1,\dots}$ denotes the set of non-negative integers.}.
When $v,w\in\rond W$ and $r\in\rond R$, the concatenation of $v$ and $w$ (\emph{resp.} $w$ and $r$) is denoted by $vw$ (\emph{resp.} $wr$).
Moreover, a finite word $w$ being given, \[w\rond R\]
denotes the cylinder made of right-infinite words having $w$ as a prefix.

\vskip 5pt
A VLMC is an $\rond R$-valued Markov chain, defined by a so-called \emph{probabilised context tree}.
We give hereunder a compact description.
One can refer to~\cite{cenac/chauvin/paccaut/pouyanne/12} for an extensive definition\footnote{In \cite{cenac/chauvin/paccaut/pouyanne/12}, and in most of the literature on the subject, VLMC are processes on \emph{left-}infinite words, growing to the right.
This convention forces to make frequently use of reversed words in the discourse.
Because of this drawback, we make here the opposite choice.}.

\vskip 5pt
A \emph{context tree} is a rooted tree $\rond T$ built on the alphabet $\rond A$, which has an \emph{at most countable set of infinite branches};
an infinite sequence $r\in\rond R$ is an \emph{infinite branch} of $\rond T$ whenever all its finite prefixes belong to $\rond T$.
As usual, the nodes of the tree are canonically labelled by words on $\rond A$.
In the example of Figure~\ref{fig:ex0}, the alphabet is $\set{0,1}$ and the tree has two infinite branches: $(01)^\infty$ and $1^\infty$. 
For a finite word $w\in\rond W$, $w^\infty$ denotes the right-infinite word $www\cdots$.
A node of a context tree $\rond T$ will be called a \emph{context} when it is a finite leaf or an infinite branch of $\rond T$.
The sets of all contexts, finite leaves and infinite branches are respectively denoted by
\[
\rond C,~\rond C^f
{\rm ~and~}
\rond C^i.
\]
These sets are at most countable.
A finite word $w\in\rond W$ will be called an \emph{internal node} when it is \emph{strictly} internal as a node of $\rond T$;
it will be called \emph{non-external} whenever it is internal or a context.
In the same vein a finite word or a right-infinite sequence will be said \emph{external} when it is strictly external and \emph{non-internal} when it is external or a context.
The set of internal words
is denoted by
\[
\rond I.
\]

\begin{rem}
An infinite tree on a finite alphabet being given, the fact that it is a context tree or not is not directly related to the growth of the number $f(n)$ of leaves at height $n$ when $n$ tends to infinity.
Indeed, $f(n)$ may grow slowly whereas the set of infinite branches is not countable. Conversely, $f(n)$ may grow rapidly while the set of infinite branches is countable.
One can refer to the first appendix in~\cite{ferreira/gallo/paccaut/19} for more precise statements.
\end{rem}

\begin{defi}[\emph{cont} of a non-internal word]
Let $\rond T$ be a context tree and $w$ be a non-internal finite or infinite word.
Then, $\pref (w)$ denotes the unique prefix of $w$ which is a context of $\rond T$.
\end{defi}
For a more visual representation, hang $w$ by its head (its left-most letter) and insert it into the tree, the head of $w$ being placed at the root;
the only context through which the word goes out of the tree is its \emph{cont}
-- see Figure~\ref{fig:ex0}.

\begin{figure}[h]
\begin{center}
\begin{tikzpicture}[scale=0.6]
\tikzset{every leaf node/.style={circle,fill,scale=0.5},every internal node/.style={circle,fill,circle,scale=0.01}}
\Tree
[.{} 
\edge[line width=2pt,red];[.{} 
{}; 
\edge[line width=2pt,red];[.{} 
\edge[line width=2pt,red];[.{} 
{}; 
\edge[line width=2pt,red];[.{} 
[.{} 
{}; 
[.{} 
[.{} 
{}; 
[.{} 
\edge[line width=2pt,dashed];\node[fill=white,draw=white]{};
{}; 
]
]
{}; 
]
]
\edge[line width=2pt,red];\node{};
]
]
{}; 
]
]
[.{} 
{}; 
[.{} 
{}; 
[.{} 
{}; 
[.{} 
{}; 
[.{} 
{}; 
[.{} 
{}; 
[.{} 
{}; 
[.{} 
{}; 
\edge[line width=2pt,dashed];\node[fill=white,draw=white]{};
]
]
]
]
]
]
]
]
]
\draw (0,0.5) node{$\emptyset$};
\draw (-1.7,-0.8) node{$0$};
\draw (1.6,-0.8) node{$1$};
\draw (-2.9,-2.4) node{$00$};
\draw (2.2,-2) node{$11$};
\draw (-2.9,-4.4) node{$0100$};
\draw (2.5,-8.5) node{$1^70$};
\draw (4.5,-9.5) node{$1^\infty$};
\draw (-1.5,-9.8) node{$(01)^\infty$};
\draw (0.5,-5.5) node{$c$};
\draw (10,-5) node{$c=01011=\pref ({\color{red}01011}1101000\cdots)$};
\end{tikzpicture}
\end{center}
\caption{an example of context tree on the alphabet $\rond A=\set{0,1}$.
It has two infinite branches: $1^\infty$ and $(01)^\infty$.
The \emph{cont} of any right-infinite word or finite word beginning by  $010111101000\cdots$ is the context $01011$.
\label{fig:ex0}}
\end{figure}

A \emph{probabilised context tree} is a context tree $\rond T$ endowed with a family of probability measures $q=\left( q_c\right) _{c\in\rond C}$ on $\rond A$ indexed by the (finite and infinite) contexts of $\rond T$.
To any probabilised context tree, one can associate a {\bf VLMC} (Variable Length Memory Chain), which is the $\rond R$-valued Markov chain $\left( U_n\right)_{n\geq 0}$ defined by its transition probabilities given by
\begin{equation}
\forall n\geq 0,~\forall\alpha\in\rond A,~\Proba\left( U_{n+1}=\alpha U_n | U_n\right)
=q_{\pref\left(U_n\right)}\left(\alpha\right).
\end{equation}

The set $\rond R$ is endowed with its cylinder $\sigma$-algebra, generated by the cylinders $w\rond R$, $w\in\rond W$.
In the whole paper, the left-most letter of the sequence $U_n\in\rond R$ is denoted by $X_n$ so that the random sequences grow by adding successive letters $X_0$, $X_1$, $X_2,\dots$ on the left of $U_0$:
\[
\forall n\geq 0,~U_{n+1}=X_{n+1}U_n.
\]

\begin{rem}
\label{pasTrivial}
A context tree is never empty because it contains at least its root.
The smallest context tree is thus reduced to its root $\emptyset$.
Once probabilised by a single probability measure $q _\emptyset$ on $\rond A$, this tree gives rise to the simplest VLMC which consists in a sequence of  \emph{i.i.d.\@} $q _\emptyset$-distributed random variables $\left( X_n\right) _n$.
Besides, the tree $\set\emptyset$ is the only context tree that does not get any internal node.
Since the combinatorial aspect of our study is heavily based on internal nodes of context trees (notion of LIS, see Section~\ref{subsec:cascade}), we make the following small restriction.
\begin{center}
-- In the whole paper, all context trees are supposed not to be reduced to their root. --
\end{center}
\end{rem}

\begin{rem}
\label{rem:notMarkov}
When the context tree has at least one infinite context, the initial letter process $\left( X_n\right) _{n\geq 0}$ is generally \emph{not} a Markov process.
When the context tree is finite, $\left( X_n\right) _{n\geq 0}$ is a usual $\rond A$-valued Markov chain
whose order is the height of the tree, \emph{i.e.} the length of its longest branch.
\end{rem}

\vskip 5pt
This section ends by two definitions that will be used in the sequel:
our main results on VLMC hold for non-null ones and the shift appears as a useful technical tool.

\begin{defi}[non-nullness]
\label{defi:non-null}
A probabilised context tree $(\rond T,q)$ is \emph{non-null} whenever $q_c(\alpha)\neq 0$ for every $c\in\rond C$ and every $\alpha\in\rond A$.
A \emph{non-null VLMC} is a VLMC defined by a non-null probabilised context tree.
\end{defi}

\begin{defi}[shift mapping]
\label{def:shift}
The \emph{shift mapping} $\sigma:\rond R\to\rond R$ is defined by $\sigma\left(\alpha\beta\gamma\delta\cdots\right) =\beta\gamma\delta\cdots$.
The definition is extended to finite words (with $\sigma(\emptyset)=\emptyset$).
\end{defi}

The $k$-th iteration of $\sigma$ is denoted by $\sigma ^k$ (and $\sigma ^0$ denotes the identity map on $\rond R$ or $\rond W$).

\subsection{LIS and alpha-LIS, cascades and cascade series}
\label{subsec:cascade}

As pointed out in the introduction, the study of invariant probability measures naturally leads to the following notion of Longest Internal Suffix.
If $w\in\rond W$ is a non-empty finite word, $w$ can be uniquely written as 
\[
w=\beta _1\beta _2\dots\beta _{p_w}\alpha _ws_w,
\]
where
 
$\bullet$ $p_w\geq 0$ and $\beta_i\in \rond A$, for all $i\in\{ 1, \dots , p_w\}$, 

$\bullet$ $\alpha _w\in\rond A$,

$\bullet$ $s_w$ is the longest internal strict suffix of $w$.

Note that $s_w$ may be the empty word.
When $p_w=0$, there are no $\beta$'s and $w=\alpha _ws_w$.

\begin{defi}[LIS and alpha-LIS]
\label{def:alphaLIS}
Let $\rond T$ be a context tree and $w$ a finite non-empty word on~$\rond A$.
With the notations above, the Longest Internal Suffix $s_w$ is abbreviated as the \emph{LIS} of $w$;
the non-internal suffix $\alpha _ws_w$ is called the \emph{alpha-LIS} of $w$.
\end{defi}

To compute the LIS of a non-empty finite word $w=\beta_1\beta_2\dots\beta_n$, check whether $\beta_2\beta_3\dots\beta_n$ is internal or not.
If it is internal, that is the LIS of $w$.
If not, check whether $\beta_3\beta_4\dots\beta_n$ is internal or not, etc.
The first time you get an internal suffix
(this happens inevitably because $\emptyset$ is always an internal word, the context tree being not reduced to its root, see Remark~\ref{pasTrivial}),
this suffix is the LIS of $w$.

Any word has an alpha-LIS, but the objects of main interest are the alpha-LIS of \emph{contexts}.
The set of alpha-LIS \emph{of finite contexts} of $\rond T$ will be denoted by $\rond S(\rond T)$, or more shortly by $\rond S$:
\[
\rond S=\set{\alpha _cs_c,~c\in\rond C^f};
\]
this is an at most countable set (like $\rond C$).
For any $u,v,w\in\rond W$, the notations
\begin{equation}
\label{notationsPrefLis}
v=u\cdots
{\rm ~~and~}
w=\cdots [u]
\end{equation}
stand respectively for ``$u$ is a prefix of $v$'' and ``$u$ is the alpha-LIS of $w$''.

\begin{exa}[computation of a LIS]
\label{exa:LIS}\ 

\begin{minipage}[t]{0.6\textwidth}
\vspace{0pt}
In this example, the alphabet is $\rond A=\set{0,1}$ and the context tree is defined by its finite contexts which are the following ones:
$(01)^p00$, $(01)^r1$, $01^r0$, $1^q00$, $1^q01$, $p\geq 0$, $q\geq 1$, $r\geq 2$.

Take for example the context $010100$, colored red in the context tree. Remove successively letters from the left until you get an internal word: $10100$ is external, $0100$ is noninternal, $100$ is noninternal, $00$ is noninternal.
In this sequence, the suffix $0$ is the first internal one:
this is the LIS of $010100$.
The last removed letter is $\alpha=0$ so that the alpha-LIS of $010100$ is $00$.
\end{minipage}
\begin{minipage}[t]{0.4\textwidth}
\vspace{0pt}
\centering
\begin{tikzpicture}[scale=0.42]
		\tikzset{every leaf node/.style={draw,circle,fill},every internal node/.style={draw,circle,scale=0.01}}
		\Tree [.{} [.{} {} [.{} [.{} {} [.{} [.{} \node[fill=red,draw=black]{}; [.{} [.{} {} [.{} \edge[line width=2pt,dashed];\node[fill=white,draw=white]{}; \edge[line width=2pt,dashed];\node[fill=white,draw=white]{}; ] ] {} ] ] {} ] ] [.{} {} [.{} {} [.{} {} [.{} {} [.{} {} [.{} \edge[line width=2pt,dashed];\node[fill=white,draw=white]{}; \edge[line width=2pt,dashed];\node[fill=white,draw=white]{}; ] ] ] ] ] ] ] ] [.{} [.{} {} {} ] [.{} [.{} {} {} ] [.{} [.{} {} {} ] [.{} [.{} {} {} ] [.{} [.{} {} {} ] [.{} [.{} {} {} ] [.{} \edge[line width=2pt,dashed];\node[fill=white,draw=white]{}; \edge[line width=2pt,dashed];\node[fill=white,draw=white]{}; ] ] ] ] ] ] ] ]
\end{tikzpicture}
\end{minipage}
\vskip 5pt
In the following array, the left-hand column consists in the list of alpha-LIS of all the finite contexts of the tree.
For every $\alpha s\in\rond S$, the list of all finite contexts having $\alpha s$ as an alpha-LIS is given in the right-hand column.
\begin{center}
\begin{tabular}{r|l}
$\alpha s\in\rond S$&\emph{finite contexts having $\alpha s$ as an alpha-LIS}\\
\hline
$00$&$1^q00$, $(01)^p00$, $p\geq 0$, $q\geq 1$\\
$101$&$1^q01, q\geq 1$\\
$01011$&$(01)^r1$, $r\geq 2$\\
$01^r0$, $r\geq 2$&$01^r0$
\end{tabular}
\end{center}
\end{exa}

\vskip 5Pt
\begin{rem}
The finiteness of the set $\rond C^i$ of infinite branches on one side, and that of the set $\rond S$ of context alpha-LIS on the other side are not related.
In Example~\ref{exa:pgpd}, one finds a context tree for which $\rond S$ is finite while $\rond C^i$ is infinite.
In the tree of Example~\ref{exa:LIS}, $\rond S$ is infinite while $\rond C^i$ is finite.
The left-comb of left-combs has infinite $\rond C^i$ and $\rond S$
(see Remark~\ref{rem:realisation}).
Finally, the double bamboo (see page~\pageref{tikz:doublebamboo})  has finite $\rond C^i$ and $\rond S$.
\end{rem}

\begin{defi}[cascade]
\label{def:cascade}
Let $\left(\rond T,q\right)$ be a probabilised context tree.
If $w\in\rond W$ writes $w=\beta _1\beta _2\dots\beta _{p}\alpha s$ where $p\geq 0$ and where $\alpha s$ is the alpha-LIS of $w$, the \emph{cascade} of $w$ is defined as
\[
\casc (w)=\prod _{1\leq k\leq p}q_{\pref\sigma ^k(w)}\left(\beta _k\right),
\]
where an empty product equals $1$, which occurs if and only if $w$ is equal to its own alpha-LIS.
In the above formula, $\sigma$ denotes the shift mapping, see Definition~\ref{def:shift}.
The cascade of $\emptyset$ is defined as being $1$.
\end{defi}

Note that $\casc (\alpha s) = 1$ for any $\alpha s\in\rond S$.
In Example \ref{exa:LIS}, $\casc(010100)=q_{101}(0)q_{0100}(1)q_{100}(0)q_{00}(1)$.

\begin{rem}
\label{rem:cascRem}
Assume that $\rond A=\set{0,1}$.
For any $w\in\rond W$, $\casc (w)=\casc (0w)+\casc (1w)$ if and only if $w$ is non-internal;
indeed, if $w$ is internal, the sum equals $2$ whereas $\casc (w)\leq 1$.
This equivalence generalizes straightforwardly to an arbitrary alphabet.
\end{rem}

\begin{defi}[cascade series]
\label{defi:cascade_series}
For every  $\alpha s\in\rond S$, the \emph{cascade series of $\alpha s$} (related to $(\rond T,q)$) is the at most countable family of cascades of the finite contexts having $\alpha s$ as their alpha-LIS.
In other words, with notations~\eqref{notationsPrefLis}, it is the family
\[
\left(\casc (c)\right) _{c\in\rond C^f,~c=\cdots [\alpha s]}.
\]
\end{defi}

Since the cascades are positive numbers, the summability of a family of cascades of a probabilised context tree is equivalent to the convergence of the series associated to any total order on the set of contexts indexing the family.
The assertion
\begin{equation}
\label{convCasc}
\forall \alpha s\in\rond S,~\sum _{\substack{{c\in\rond C^f}\\[2pt]{c=\cdots [\alpha s]}}}\casc (c)<+\infty
\end{equation}
will be called \emph{convergence of the cascade series}.
For every $\alpha s\in\rond S$ and $k\geq 1$, denote
\begin{equation}
\label{def:ck}
\kappa _{\alpha s}(k)
=\sum _{\substack{{c\in\rond C^f},~c=\cdots [\alpha s]\\[2pt] |c| = |\alpha s|+k-1}}\casc (c).
\end{equation}
When the cascade series converge, $\kappa _{\alpha s}$ denotes the sum of the cascade series relative to $\alpha s\in\rond S$:
\begin{equation}
\label{defKappa}
\kappa _{\alpha s}=\sum _{\substack{{c\in\rond C^f}\\[2pt]{c=\cdots [\alpha s]}}}\casc (c) = \sum_{k\geq 1} \kappa _{\alpha s}(k).
\end{equation}

In the following sections, the convergence of cascade series turns out to be an important part of the characterization of stationary probability measures.
This is made precise by Theorem~\ref{fQbij} and Theorem~\ref{th:stable}.
In some particular cases, the convergence of cascade series just becomes a necessary and sufficient condition for existence and unicity of an invariant probability measure (see Theorem~\ref{cor:finite}).

\subsection{Alpha-LIS matrix $Q$ and left-fixed vectors}
\label{subsec:Q}

For any $(\alpha s,\beta t)\in\rond S^2$, with notations~\eqref{notationsPrefLis}, define
\begin{equation}
\label{defQ}
Q_{\alpha s,\beta t}
=\sum _{\substack{c\in\rond C^f\\[2pt]c=t\cdots\\[2pt]c=\cdots [\alpha s]}}
\casc\left(\beta c\right)
\in [0,+\infty ].
\end{equation}
As the set $\rond S$ is at most countable, the family $Q=\left( Q_{\alpha s,\beta t}\right) _{(\alpha s,\beta t)\in\rond S^2}$ will be considered a matrix, finite or countable, for an arbitrary order on $\rond S$.
The convergence of the cascade series of $(\rond T,q)$ is sufficient to ensure the finiteness of $Q$'s entries.

\vskip 5pt
The matrix $Q$ plays a central role in the statement of Theorem~\ref{fQbij}, which is the main result of the paper.

\begin{defi}[left-fixed vector of a matrix]
Let $A=\left( a_{\ell ,c}\right)_{(\ell ,c)\in\rond E^2}$ be a matrix with real entries, indexed by a totally ordered set $\rond E$ supposed to be finite or denumerable.
A \emph{left-fixed vector} of $A$ is a row-vector $X=(x_{k})_{k\in\rond E}\in\g R^{\rond E}$, indexed by $\rond E$, such that $XA=X$.
In particular, this implies that the usual matrix product $XA$ is well defined, which means that for any $c\in\rond E$, the series $\sum _\ell x_\ell a_{\ell ,c}$ is convergent.
Note that, whenever $X$ and $A$ are infinite dimensional and have nonnegative entries, this summability does not depend on the chosen order on the index set $\rond E$.
\end{defi}

\subsection{Stationary measures for a VLMC}
\label{sec:general}

Definitions and notations of the previous sections allow us to state results on stationary measures for a VLMC.
In this section no assumption is made on the shape of the context tree. After two key lemmas, we state the main Theorem~\ref{fQbij} that establishes precise connections between stationary probability measures of the VLMC and left-fixed vectors of the matrix $Q$ defined in Section~\ref{subsec:Q}.
Theorem~\ref{fQbij} is valid for any context tree.
Section~\ref{sec:stable} shows what happens to this result when assumptions (stability, mainly) are made on the shape of the tree.
In particular, Remark~\ref{rem:finite} shows how Theorem~\ref{fQbij} (or Theorem~\ref{cor:finite}) applies in the case of finite trees.

\begin{defi}[stationary probability measure for a VLMC]
\label{def:mesInv}
Let $U=\left( U_n\right) _{n\geq 0}$ be a VLMC.
A probability measure $\pi$ on $\rond R$ is said \emph{$U$-stationary} (or also \emph{$U$-invariant}) whenever $\pi$ is the distribution of every $U_n$ as soon as it is the distribution of $U_0$.
\end{defi}

Assume that $\pi$ is a probability measure on $\rond R$, invariant for a VLMC defined on a given context tree.
As already mentioned in the introduction, $\pi\left(w\rond R\right) =q_{\pref (v)}\left(\alpha\right)\pi\left( v\rond R\right)$ for any letter $\alpha$ and any non-internal finite word $w=\alpha v$.
The cascade of $w$ is the product that arises after the largest number of possible iterations of that formula, so that $\pi\left( w\rond R\right) =\casc (w)\pi\left(\alpha _ws_w\rond R\right)$.
These formulae are the subject of the simple but very useful Lemma~\ref{lem:cascade}, named Cascade Formulae.
Equality~\eqref{cascade3} can be seen as a founding formula that leads to Theorem~\ref{fQbij}.

\begin{lem}
\label{lem:cascade}
{\bf (Cascade formulae)}

Let $(\rond T,q)$ be a probabilised context tree and $\pi$ be a stationary probability measure for the corresponding VLMC.

(i)
For every non-internal finite word $w$ and for every $\alpha\in\rond A$,
\begin{equation}
\label{cascade1}
\pi\left(\alpha w\rond R\right) =q_{\pref (w)}(\alpha )\pi\left(w\rond R\right) .
\end{equation}

(ii)
For every right-infinite word $r\in\rond R$ and for every $\alpha\in\rond A$,
\begin{equation}
\label{cascade2}
\pi\left(\alpha r\right) =q_{\pref (r)}(\alpha )\pi\left( r\right) .
\end{equation}

(iii)
For every finite non empty word $w$, if one denotes by $\alpha_ws_w$ the alpha-LIS of $w$, then
\begin{equation}
\label{cascade3}
\pi\left(w\rond R\right) =\casc(w)\pi\left(\alpha _ws_w\rond R\right) .
\end{equation}
\end{lem}

A proof of Lemma~\ref{lem:cascade} can be found at the beginning of Section~\ref{sec:proofs} on page~\pageref{proof:cascade}.

The following lemma ensures that a stationary probability measure weights finite words and only finite words.
\begin{lem}
\label{lem:pineq0}
Let $\pi$ be a stationary probability measure of a non-null VLMC.
Then 

(i) $\forall w\in\rond W$, $\pi\left( w\rond R\right)\neq 0$;

(ii) $\forall r\in\rond R$, $\pi (r)=0$.
\end{lem}
For a proof of this lemma, see Section~\ref{sec:proofs}, page~\pageref{proof:pineq0}.

\begin{rem}
Thanks to Lemma~\ref{lem:pineq0}(ii), when $\pi$ is a stationary probability measure, both members of Equality~\eqref{cascade2} vanish.
In fact, all formulae in Lemma~\ref{lem:cascade} remain true when $\pi$ is a $\sigma$-finite invariant measure.
In this case, Formula~\eqref{cascade2} may be an equality between two non-zero real numbers.
See Remark~\ref{sigmaFinite} and Section~\ref{sec:appendix} for further comments on $\sigma$-finite invariant measures.
\end{rem}

\vskip 5pt
\newcommand{\fff}{f}
Everything is now in place to state the main theorem.
Denote by $\rond M_1\left(\rond R\right)$ the set of probability measures on $\rond R$.
For a given context tree $\rond T$, define the mapping $\fff$ as follows:
\begin{equation}
\label{def:f}
\begin{array}{cccc}
\fff :&\rond M_1\left(\rond R\right) & \longrightarrow &[0,1]^{\rond S}
\\[3pt]
&\pi & \longmapsto & \Big( \pi\left(\alpha s\rond R\right)\Big) _{\alpha s\in\rond S}.
\end{array}
\end{equation}
\begin{theo}
\label{fQbij}
Let $(\rond T,q)$ be a non-null probabilised context tree and $U$ the associated VLMC.

\vskip 3pt
(i)
Assume that there exists a finite $U$-stationary probability measure $\pi$ on $\rond R$.
Then the cascade series~\eqref{convCasc} converge.
Furthermore, using notation~\eqref{defKappa},
\begin{equation}
\label{masseTotale}
\sum _{\alpha s\in\rond S}\pi\left(\alpha s\rond R\right)
\kappa _{\alpha s}=1.
\end{equation}

\vskip 3pt
(ii)
Assume that the cascade series~\eqref{convCasc} converge.
Then, $f$ induces a bijection between the set of $U$-stationary probability
measures on $\rond R$ and the set of left-fixed vectors $\left( v_{\alpha s}\right)_{\alpha s\in\rond S}$ of $Q$ that have non-negative entries and which satisfy
\begin{equation}
\label{posrec}
\sum _{\alpha s\in\rond S}v_{\alpha s}\kappa _{\alpha s}=1.
\end{equation}
\end{theo}

The proof of Theorem~\ref{fQbij} is given in Section~\ref{sec:proofs}, page~\pageref{proof:fQbij}.

\vskip 5pt
This theorem naturally calls for several questions and remarks:
for instance, does everything boil down to $Q$?
Can the theorem be extended to $\sigma$-finite invariant measures?
Can Theorem~\ref{fQbij} be improved for particular context trees?
For finite ones?
What role does the non-nullness assumption play?

\begin{rem}
\label{rem:QfaitPasTout}
One could be tempted to see $f(\pi )$ as an invariant measure for some Markov chain associated with the matrix $Q$, reducing the study of invariant probability measures of a VLMC to the study of stationary probability measures of the Markov chain associated with $Q$.
This is generally not true.

\vskip 5pt
First, even when it is finite-dimensional, $Q$ is generally not stochastic, excluding any hope of interpreting it as the transition matrix of some Markov chain.
Take for instance the small context tree on the alphabet $\rond A=\set{0,1}$ pictured hereunder.
It gets three context alpha-LIS we order the following way:
$00$, $10$ and $1$.
The matrix $Q$ writes straightforwardly as follows.
For instance, its first line's sum equals $1+q_{00}(1)$.

\begin{center}
\begin{tikzpicture}[scale=0.6]
\draw (0,0)--(-1,-1);\draw (0,0)--(1,-1);
\fill (1,-1) circle(0.15);\draw (1,-1.5) node{\small $1$};
\draw (-1,-1)--(-1.8,-2);\draw (-1,-1)--(-0.2,-2);
\fill (-1.8,-2) circle(0.15);\draw (-1.8,-2.5) node{\small $00$};
\draw (-0.2,-2)--(-0.8,-3);\draw (-0.2,-2)--(0.4,-3);
\fill (-0.8,-3) circle(0.15);\draw (-0.85,-3.5) node{\small $010$};
\fill (0.4,-3) circle(0.15);\draw (0.45,-3.5) node{\small $011$};
\draw (12,-2) node{
$Q=\begin{pmatrix}
\casc (000)&\casc (100)&\casc (100)\\[3pt]
\casc (0010)&\casc (1010)&\casc (1010)\\[3pt]
\casc (0011)&\casc (1011)&\casc (1011)+\casc (11)
\end{pmatrix}$
};
\end{tikzpicture}
\end{center}

Second, even when $Q$ is row-stochastic (which is the case when the context tree is \emph{stable}, see Proposition~\ref{pro:stochasticity}), its probabilistic interpretation is not that simple.
In the stable case, $Q$ can be seen as the transition matrix of the underlying Markov chain of some semi-Markov chain, namely the process of the context alpha-LIS of the VLMC.
Section~\ref{sec:SM} is devoted to this fact.
\vskip 5pt
Finally, in general, even in the case of stable VLMC, one cannot reconstruct the VLMC from the process of its alpha-LIS:
both processes are not equivalent, the VLMC being strictly richer than the process of its alpha-LIS.
See Remark~\ref{rem:semiMarkov} for an example and further comments.
\end{rem}

\begin{rem}
Non-nullness appears as some irreducibility assumption on the Markov process on right-infinite words.
One can find in~\cite{cenac/chauvin/paccaut/pouyanne/12} simple examples of not  non-null VLMC's defined on infinite context trees that admit infinitely-many invariant probability measures.
\end{rem}

\begin{rem}
\label{sigmaFinite}
One may wonder whether a non-null VLMC can admit invariant \emph{$\sigma$-finite} measures that have an infinite total mass. The answer is clearly affirmative as can be seen on the \emph{left comb}, which is the context tree shaped as follows, the alphabet being $\rond A=\set{0,1}$:
\begin{minipage}{35pt}
\begin{tikzpicture}[scale=0.3]
\tikzset{every leaf node/.style={draw,circle,fill},every internal node/.style={draw,circle,scale=0.01}}
\Tree [.{} [.{} [.{} [.{}  \edge[line width=2pt,dashed];\node[fill=white,draw=white]{};\edge[draw=white];\node[fill=white,draw=white]{}; {} ] {} ] {} ] {} ]
\end{tikzpicture}
\end{minipage}.
Once this tree has been probabilised by the non-null family $\left( q_{0^n1}\right) _{n\geq 0}$, define $c_n$ as being
\[
c_n:=\casc\left( 0^n1\right)=\prod _{k=0}^nq_{0^k1}(0).
\]
Then, as soon as $c_n$ tends to $0$ when $n$ tends to infinity whereas the series $\sum c_n$ diverges, the corresponding VLMC gets an invariant $\sigma$-finite measure with infinite total mass.
This can be straightforwardly checked
-- however, computation details can be found in~\cite{cenac/chauvin/paccaut/pouyanne/12}.

\vskip 5pt
Moreover, the same argument as in the proof of Lemma~\ref{lem:pineq0}(ii) shows that a $U$-invariant $\sigma$-finite measure always vanishes on \emph{rational} right-infinite words, \emph{i.e.}\@ on eventually periodic words.
One may thus wonder whether a non-null VLMC can admit invariant $\sigma$-finite measures that have an infinite total mass and take a positive value on some irrational infinite word.
The answer is also affirmative.
An example is developed in the appendix, based on a context tree which has irrational contexts and whose $Q$ matrix is (necessarily) transient.
\end{rem}

\section{The stable case}
\label{sec:stable}

In this section, a restriction on the shape of the tree is put, called stability, defined in Section~\ref{subsection:defstable}.
As already said in the introduction, although being very particular, the set of stable trees appears as a very rich class, notably through its links with semi-Markov chains.
These links, detailed in Section~\ref{subsec:compareSM}
(stochasticity and irreducibility of $Q$, construction of the induced semi-Markov chain denoted by $\left( Z_n\right) _{n\geq 0}$), exhibit renewal properties of the VLMC.

The extra structure brought by the stability enables to simplify the statement of Theorem~\ref{fQbij}, turning it into a necessary and sufficient condition for existence and unicity of a stationary probability measure, for countable $\rond S$ (Theorem~\ref{th:stable}) and finite $\rond S$ (Theorem~\ref{cor:finite}, where the convergence of the law of $\left( U_n\right)$ towards the invariant measure is also obtained).

It must be once again emphasized that the trajectories of the VLMC $\left( U_n\right)$ cannot be recovered from the trajectories of the underlying semi-Markov chain $\left( Z_n\right)$ (See Remark \ref{rem:semiMarkov}). Our results on stable VLMC cannot straightforwardly be deduced from those existing in the semi-Markov literature.

\subsection{Definitions}
\label{subsection:defstable}

\begin{pro}
\label{pro:defstable}
Let $\rond T$ be a context tree. The following conditions are equivalent.
\begin{enumerate}
	\item[(i)] $\forall \alpha \in \rond A$, $\forall w \in \rond W$, $\alpha w \in \rond T \Longrightarrow w \in \rond T$.
	In other words, $\sigma(\rond T )\subseteq\rond T$.
	\item[(ii)] If $c$ is a finite context and $\alpha \in \rond A$, then $\alpha c$ is non-internal.
	\item[(iii)] $\rond T\subseteq\rond A\rond T$, where $\rond A\rond T=\{\alpha w,~\alpha\in\rond A,~w\in\rond T\}$.
	\item[(iv)] For any VLMC $\left( U_n\right) _n$ associated with $\rond T$, the process $\left( C_n\right) _{n\in\g N}:=\left(\pref\left( U_n\right)\right)_{n\in\g N}$ is a Markov chain with state space $\rond C$.
\end{enumerate} 
\end{pro}

A proof of this Proposition~\ref{pro:defstable} can be found in Section~\ref{sec:proofsStable}, page~\pageref{proof:defstable}.

\begin{defi}[shift-stable tree, stable VLMC]
\label{def:stable}
A context tree is \emph{shift-stable}\footnote{This property of trees is also called \emph{$0$-subperiodic} by some authors, like \cite{Lyons90, LPbook} or \emph{shift-invariant} by \cite{Furstenberg}.},
shortened in the sequel as \emph{stable} when one of the four equivalent  conditions of Proposition~\ref{pro:defstable} is satisfied.
A VLMC is also called \emph{stable} when it is defined by a probabilised stable context tree.
\end{defi}


\vskip 5pt
The following two lemmas, which do not hold for general trees, will be used to get an accurate description of the structure of the context alpha-LIS process, as developed in Section~\ref{subsec:compareSM}.

\begin{lem}
\label{lem:contextStableTree}
Let $\rond T$ be a stable context tree.

(i) Any context alpha-LIS is a context.
In otherwords, $\rond S\subseteq\rond C$.

(ii) Assume that $c$ is a finite context having $\alpha s$ as an alpha-LIS.
Then all $\sigma ^k(c)$, $0\leq k\leq |c|-|\alpha s|$ are also contexts having $\alpha s$ as an alpha-LIS.
\end{lem}

\begin{proof}
Let $\alpha s\in\rond S$ and let $c=\cdots [\alpha s]\in\rond C^f$ (notation~\eqref{notationsPrefLis}).
Since $\rond T$ is stable, for any $k\in\g N$, the node $\sigma ^k(c)$ is either internal or a context.
By maximality of $s$, this implies that the $\sigma ^k(c)$, for $0\leq k\leq |c|-|\alpha s|$, have $\alpha s$ as a suffix and are noninternal, thus contexts.
This proves (ii), thus (i).
\end{proof}

\begin{lem}
\label{lem:fondLem}
Let $\rond T$ be a stable context tree and $c\in\rond C$. Let $\rond A_c:=\{\alpha\in\rond A, \alpha c\notin~\rond C\}$. Then, 
\begin{enumerate}
\item if $\rond A_c=\emptyset$, then $c$ does not admit any context LIS as a prefix;
\item
for every $\alpha\in\rond A_c$, there exists a unique context LIS $t_{\alpha}$ such that

(i) $c=t_{\alpha}\cdots$

(ii) $\alpha t_{\alpha}\in\rond C$.

Furthermore, for every $\beta\notin\rond A_c$, $\beta t_{\alpha}\notin\rond C$.
\end{enumerate}
\end{lem}

The proof of this lemma is given in Section~\ref{sec:proofs}, page~\pageref{proof:fondLem}.

\vskip 5pt
Note in passing the following formula, proven during the proof of Proposition~\ref{pro:defstable} and valid in the case of stable context trees:
if $s\in\rond R$ is a right-infinite word and if $\alpha\in\rond A$ is any letter, then 
\[
\pref(\alpha s) = \pref\left(\alpha\pref(s) \right).
\]
This formula is the foundation for the renewal properties of stable VLMC's, as described hereunder.
For any $n\geq 0$ and for any letter~$\beta$, because of this formula, $\pref\left( \beta U_n\right)$ depends on $U_n$ only through its \emph{cont}.
More precisely, if $C_n$ denotes $\pref\left( U_n\right)$, then $\pref\left( \beta U_n\right)=\pref\left( \beta C_n\right)$.
Furthermore, thanks to Lemma~\ref{lem:fondLem}, if $c$ is any finite context having $\alpha s$ as an alpha-LIS and if $\beta$ is any letter, two disjoint cases may occur:
either $\beta c$ is a context which has again $\alpha s$ as an alpha-LIS, or $\beta c$ is an external word, $\pref (\beta c)$ being its own alpha-LIS.
This fact contains in germ the announced renewal property of a stable VLMC, as completely formalized in Proposition~\ref{pro:compareSM}, the context alpha-LIS's constituting renewal patterns of a stable VLMC:
once $U_n$ has begun by a context alpha-LIS, the process will never make use of letters in the past beyond this alpha-LIS.

\begin{rem}
\label{rem:pasRenouveau}
A general (non-stable) VLMC does not enjoy such a renewal phenomenon.

\begin{center}
\begin{minipage}{0.3\textwidth}
\begin{tikzpicture}
\draw (0,0) node{
\newcommand{\diam}{6 pt}
\newcommand{\pente}{1.3}
\newcommand{\queue}[2]{
\fill (#1*\pente,#2) circle(\diam);
\fill (#1*\pente-\pente,#2-1) circle(\diam);\fill (#1*\pente+\pente,#2-1) circle(\diam);
\draw (#1*\pente-\pente,#2-1)--(#1*\pente,#2)--(#1*\pente+\pente,#2-1);
}
\newcommand{\queueFin}[2]{
\fill (#1*\pente,#2) circle(\diam);
\draw [densely dashed] (#1*\pente,#2)--++(-0.8*\pente,-0.8);
\draw [densely dashed] (#1*\pente,#2)--++(0.8*\pente,-0.8);
}

\begin{tikzpicture}[scale=0.25]
\queue 00
\queue{-1}{-1}
\queue{0}{-2}
\queue{-1}{-3}
\queue{-2}{-4}
\queue{-3}{-5}
\queue{-2}{-6}
\queue{-1}{-7}
\queue{-2}{-8}
\queue{-1}{-9}
\queue{0}{-10}
\queue{-1}{-11}
\queue{-2}{-12}
\queue{-3}{-13}
\queue{-4}{-14}
\queue{-5}{-15}
\queue{-4}{-16}
\queue{-5}{-17}
\queue{-4}{-18}
\queue{-5}{-19}
\queue{-6}{-20}
\queue{-5}{-21}
\queue{-4}{-22}
\queueFin{-3}{-23}
\end{tikzpicture}
};
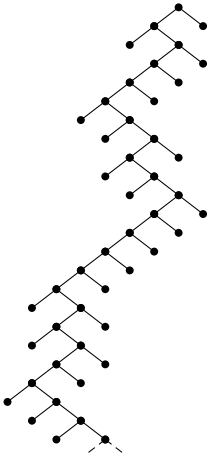
\draw (0,-3.4) node{The filament of all words};
\end{tikzpicture}
\end{minipage}
\begin{minipage}{0.65\textwidth}
Consider for instance the context tree built as follows on the alphabet $\rond A=\set{0,1}$.
Take the right-infinite word $u=0100011011000001\cdots$ obtained by concatenating all finite words ordered by increasing length and alphabetical order: $0$, $1$, $00$, $01$, $10$, $11$, $000$, etc.
Let $\rond T_u$ be the context tree spanned by $u$
-- namely the smallest context tree that contains $u$ as infinite branch.
We name $\rond T_u$ the \emph{filament of all words}.
Let also $U$ be a non-null VLMC obtained by probabilising $\rond T_u$.
Relatively to this tree, any finite word is the suffix of some internal node.
Let thus $w$ be an arbitrary finite prefix of $U_0$, and $p$ be a finite word such that $pw$ is internal.
With positive probability, $U_{|p|}=pw\cdots$ so that $\pref\left( U_{|p|}\right)$ has $pw$ as a strict prefix:
the transition from $U_{|p|}$ to $U_{|p|+1}$ depends on a prefix of $U_0$ strictly longer than~$w$.
Consequently, no finite prefix of $U_0$ can play the role of a renewal pattern for the random process~$U$.
\end{minipage}
\end{center}

Remark that this situation is generic in the following sense:
a right-infinite word $r$ on $\set{0,1}$ drawn uniformly at random has the following property.
For any finite word $w\in\rond W$, almost surely, $w$ is a pattern of $r$.
Thus, the phenomenon just described for $\rond T_u$ holds for any context tree having this infinite word $r$ as an infinite branch.
\end{rem}

\vskip 5pt
Let $\left( U_n\right) _n$ be a stable VLMC.
For every $n$, let $C_n=\pref\left( U_n\right)$.
As seen in Proposition~\ref{pro:defstable}, the process $\left( C_n\right) _n$ is a Markov chain.
In addition, $\rond C^f$ is an absorbing set for the chain $\left( C_n\right)_n$
-- as soon as a finite context is seen, all the following contexts will be finite.
This is a consequence of the renewal property described above.
Therefore, the chain induced by $\left(C_n\right)_n$ on the absorbing set $\rond C^f$ is again a Markov chain that enjoys the following properties.

\begin{lem}\label{lem:context_markov}
Let $U=\left( U_n\right) _n$ be a non-null stable VLMC.
For any $n$, let $C_n=\pref\left( U_n\right)$.
Then, the Markov chain induced by $\left(C_n\right)$ on $\rond C^f$
is irreducible and aperiodic.
\end{lem}

The proof of this lemma is made in Section~\ref{sec:proofs}, page~\pageref{proof:context_markov}.

\vskip 5pt
In view of this lemma, it would be tempting to try to study the recurrence properties of this Markov chain $\left( C_n\right) _n$ and then to apply the classical results on countable Markov chains to get a stationary probability measure for the VLMC itself. First, it appears that these recurrence properties are not at all obvious. 
Moreover, this would mean ignoring the crucial renewal properties of the alpha-LIS process, which are highlighted in Section \ref{subsec:compareSM}.
That is why it is more fruitful to work with the matrix~$Q$
-- in general a smaller matrix than the transition matrix of $\left( C_n\right) _n$.
Nevertheless, the irreducibility and aperiodicity of $\left( C_n\right) _n$ will help proving the convergence of the law of $(U_n)_n$ towards the invariant measure of the VLMC, in the case of finitely many alpha-LIS (see Theorem~\ref{cor:finite}).

\begin{defi}[stabilizable tree, stabilized of a tree]
A context tree $\rond T$ is \emph{stabilizable} whenever the stable tree $\displaystyle\bigcup_{n\in\g N}\sigma^n\left(\rond T\right)$ has at most countably many infinite branches, \emph{i.e.} when the latter is again a context tree.
When this occurs, $\displaystyle\bigcup_{n\in\g N}\sigma^n\left(\rond T\right)$ is called the \emph{stabilized} of $\rond T$;
it is the smallest stable context tree containing~$\rond T$.
\end{defi}

For example, the left-comb
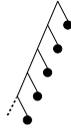
\begin{minipage}{35pt}
\begin{tikzpicture}[scale=0.3]
\tikzset{every leaf node/.style={draw,circle,fill},every internal node/.style={draw,circle,scale=0.01}}
\Tree [.{}
	[.{}
		[.{}
			[.{}
				[.{}  \edge[line width=2pt,dashed];\node[fill=white,draw=white]{};\edge[draw=white];\node[fill=white,draw=white]{}; {} ]
		{} ]
			{} ]
				{} ]
					{} ]
\end{tikzpicture}
\end{minipage}
is stable.
On the contrary, the bamboo blossom
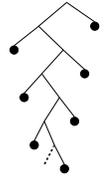
\begin{minipage}{45pt}
\begin{tikzpicture}[scale=0.3]
\tikzset{every leaf node/.style={draw,circle,fill},every internal node/.style={draw,circle,scale=0.01}}
\Tree [.{}
	[.{} {} 
		[.{} [.{} {} [.{} [.{} {} [.{} \edge[line width=2pt,dashed];\node[fill=white,draw=white]{};\edge[draw=white];\node[fill=white,draw=white]{}; {} ] ] {} ] ] {} ] ]
	{}
      ]
\end{tikzpicture}
\end{minipage}
is non-stable; it is stabilizable, its stabilized being the double bamboo
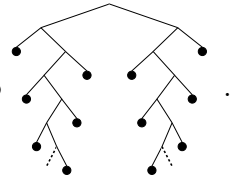
\begin{minipage}{80pt}
\begin{tikzpicture}[scale=0.3]
\label{tikz:doublebamboo}
\tikzset{every leaf node/.style={draw,circle,fill},every internal node/.style={draw,circle,scale=0.01}}
\Tree [.{}
	[.{} {} 
		[.{} [.{} {} [.{} [.{} {} [.{} \edge[line width=2pt,dashed];\node[fill=white,draw=white]{};\edge[draw=white];\node[fill=white,draw=white]{}; {} ] ] {} ] ] {} ] ] 
    \edge[draw=white];\node[fill=white,draw=white,scale=4]{}; 
		[.{} [.{} {} [.{} [.{} {} [.{} [.{} {} \edge[draw=white];\node[fill=white,draw=white]{};\edge[line width=2pt,dashed];\node[fill=white,draw=white]{};] {} ] ] {} ] ] {} ] ]
\end{tikzpicture}
\end{minipage}.

\vskip 10pt
\begin{rem}
A context tree is not necessarily stabilizable as the following examples, built on the alphabet $\set{0,1}$, show.

\vskip 10pt
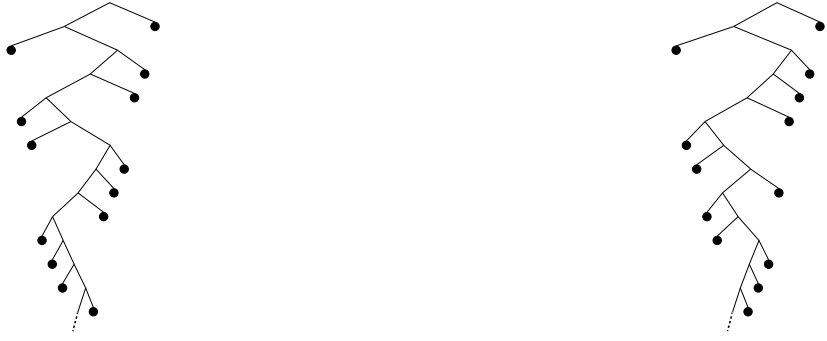
\begin{minipage}{0.56\linewidth}
\begin{center}
\begin{tikzpicture}[scale=0.3]
\tikzset{every leaf node/.style={draw,circle,fill},every internal node/.style={draw,circle,scale=0.01}}
\Tree [.{} 
		[.{} {}
		[.{} 
			[.{}
				[.{} {}
					[.{} {}
						[.{}
							[.{}
								[.{}
									[.{} {} 
										[.{} {} 
											[.{} {}
												[.{} 
													[.{} \edge[line width=2pt,dashed];\node[fill=white,draw=white]{};\edge[draw=white];\node[fill=white,draw=white]{}; ] {} 
												]
											]
										]
									]
									{} ]
								{} ]
							{} ]
					]
				]
				{} ]
			{} ]
		]
	{} ]
      
\end{tikzpicture}\\
\end{center}
This context tree consists in saturating the infinite word $010^21^2\dots 0^k1^k\cdots$ by adding hairs.
This filament tree is stabilizable, its stabilized being the context tree having the
$\{0^\ell 1^k0^{k+1}1^{k+1}\cdots\}$ and the $\{1^\ell0^k1^{k+1}0^{k+1}\cdots\}$, $k\geq 1, 0\leq \ell\leq k-1$ as internal nodes.
Its countably many infinite branches are the $0^k1^{\infty}$ and the $1^k0^{\infty}$, $k\geq 0$.
\end{minipage}
\hfill
\begin{minipage}{0.4\linewidth}
\begin{center}
\begin{tikzpicture}[scale=0.3]
\tikzset{every leaf node/.style={draw,circle,fill},every internal node/.style={draw,circle,scale=0.01}}
\Tree [.{} 
		[.{} {}
		[.{} 
			[.{}
				[.{}
					[.{} {}
						[.{} {}
							[.{}
								[.{} {}
									[.{} {} 
										[.{} 
											[.{}
												[.{} 
													[.{} \edge[line width=2pt,dashed];\node[fill=white,draw=white]{};\edge[draw=white];\node[fill=white,draw=white]{}; ] {} 
												]
											{} ]
										{} ]
									]
									 ]
								{} ]
							 ]
					] {}
				]
				{} ]
			{} ]
		]
	{} ]
      
\end{tikzpicture}\\
\end{center}
As defined in Remark~\ref{rem:pasRenouveau}, the filament of all words $\rond T_u$ is not stabilizable.
Indeed, any finite word belongs to the smallest stable tree that contains $\rond T_u$, the latter being thus the complete tree $\set{0,1}^{\g N}$, which has uncountably many infinite branches.
\end{minipage}
\end{rem}

\begin{rem}
	\label{rem:stabilized}
	Let $(\rond T,q)$ be a stabilizable probabilised context tree and $\widehat{\rond T}$ its stabilized.
	For every context $c$ of $\widehat{\rond T}$, define $\widehat{q}_c=q_{\pref(c)}$ where the function \emph{cont} is relative to $\rond T$. Then $(\rond T,q)$ and $(\widehat{\rond T},\widehat{q})$ define the same VLMC.

\begin{minipage}{0.4\textwidth}
This is straightforward because both VLMC, as Markov processes on $\rond{R}$, have the same transition probabilities.
The example of the opposite figure illustrates this construction for the bamboo blossom and its stabilized tree, the double bamboo.
\end{minipage}
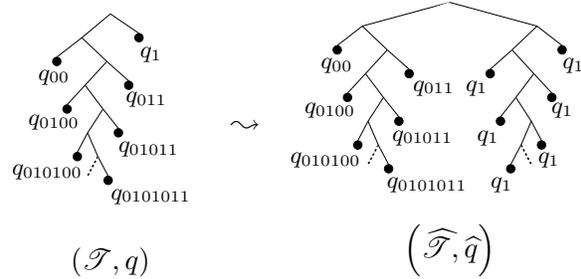
\begin{minipage}{0.6\textwidth}
\centering
\hskip 15pt
\begin{tikzpicture}[scale=0.3]
\tikzset{every leaf node/.style={draw,circle,fill},every internal node/.style={draw,circle,scale=0.001}}
\Tree [.{}
	[.{} {} 
		[.{} [.{} {} [.{} [.{} {} [.{} \edge[line width=2pt,dashed];\node[fill=white,draw=white]{};\edge[draw=white];\node[fill=white,draw=white]{}; {} ] ] {} ] ] {} ] ]
	{}
      ]
\draw (1.8,-1.8) node{{\footnotesize$q_1$}};
\draw (1.6,-3.8) node{{\footnotesize$q_{011}$}};
\draw (1.7,-6) node{{\footnotesize$q_{01011}$}};
\draw (1.8,-8.1) node{{\footnotesize$q_{0101011}$}};
\draw (-2.5,-2.8) node{{\footnotesize$q_{00}$}};
\draw (-2.5,-4.8) node{{\footnotesize$q_{0100}$}};
\draw (-2.8,-7) node{{\footnotesize$q_{010100}$}};
\draw (6,-5) node{$\leadsto$};
\draw (0,-11) node{$\left( \rond T,q\right)$};
\end{tikzpicture}
\begin{tikzpicture}[scale=0.3]
\tikzset{every leaf node/.style={draw,circle,fill},every internal node/.style={draw,circle,scale=0.001}}
\Tree [.{}
	[.{} {} 
		[.{} [.{} {} [.{} [.{} {} [.{} \edge[line width=2pt,dashed];\node[fill=white,draw=white]{};\edge[draw=white];\node[fill=white,draw=white]{}; {} ] ] {} ] ] {} ] ] 
    \edge[draw=white];\node[fill=white,draw=white,scale=4]{}; 
    \edge[draw=white];\node[fill=white,draw=white,scale=4]{}; 
		[.{} [.{} {} [.{} [.{} {} [.{} [.{} {} \edge[draw=white];\node[fill=white,draw=white]{};\edge[line width=2pt,dashed];\node[fill=white,draw=white]{};] {} ] ] {} ] ] {} ] ]
\draw (5.5,-2.8) node{{\footnotesize$q_1$}};
\draw (5,-4.8) node{{\footnotesize$q_1$}};
\draw (4.4,-7) node{{\footnotesize$q_1$}};
\draw (1.2,-3.8) node{{\footnotesize$q_1$}};
\draw (1.7,-6) node{{\footnotesize$q_1$}};
\draw (2.3,-8.1) node{{\footnotesize$q_1$}};
\draw (-5,-2.8) node{{\footnotesize$q_{00}$}};
\draw (-5.3,-4.8) node{{\footnotesize$q_{0100}$}};
\draw (-5.5,-7) node{{\footnotesize$q_{010100}$}};
\draw (-0.7,-3.8) node{{\footnotesize$q_{011}$}};
\draw (-1.2,-6) node{{\footnotesize$q_{01011}$}};
\draw (-1,-8.1) node{{\footnotesize$q_{0101011}$}};
\draw (0,-11) node{$\left( \widehat{\rond T},\widehat q\right)$};
\end{tikzpicture}
\end{minipage}
\end{rem}
\subsection{Stable VLMC and Semi-Markov Chains}
\label{sec:SM}
In this section, semi-Markov chains are defined, following~\cite{barbu/limnios/08}.
Section \ref{subsec:compareSM} is devoted to show that any stable VLMC~$\left( U_n\right) _{n\geq0}$ induces an underlying semi-Markov chain $\left( Z_n\right) _{n\geq0}$:
the state space is the set $\rond S$ of the context alpha-LIS and $Z_n$ is the alpha-LIS of the context $\pref\left( U_n\right)$.
This semi-Markov chain entirely describes the renewal property that arises in a stable VLMC and gives an explicit interpretation of the matrix $Q$.
Nevertheless, the trajectories of the VLMC cannot be recovered from those of the induced semi-Markov chain -- see Remark \ref{rem:semiMarkov}.
Despite this, interestingly, when the set of context alpha-LIS is finite, Theorem~\ref{cor:finite} and Theorem~\ref{th:CNSpourSM} below make it possible to derive equivalences between NSC for the VLMC to admit a stationary probability measure and NSC for the associated semi-Markov chain to have a limit distribution. 
This is developed in Section~\ref{sec:summarized}.
\subsubsection{Definitions}
\label{subsec:defSM}

Semi-Markov chains are defined thanks to so-called Markov renewal chains -- see~\cite{barbu/limnios/08}.

\begin{defi}[Markov Renewal Chain]
\label{def:MRC}
If $\rond E$ is any set, a Markov chain $\left( J_n,T_n\right) _{n\geq0}$ with state space $\rond E\times \g N$ is called a (homogeneous) \emph{Markov Renewal Chain} (shortly MRC) whenever the transition probabilities satisfy: $\forall n\in\g N, \forall a,b\in\rond E, \forall j,k\in\g N$,
\[
\Proba\left(  J_{n+1}=b, T_{n+1} = k\big| J_n = a, T_n =j   \right) = \Proba\left(  J_{n+1}=b, T_{n+1} = k\big| J_n = a  \right) =: p_{a,b}(k)
\]
and $\forall a,b\in\rond E$, $p_{a,b}(0) = 0$. For such a chain, the family $p=\left(p_{a,b}(k)\right)_{a,b\in\rond A, k\geq 1}$ is called its \emph{semi-Markov kernel}.
\end{defi}

\begin{defi}[Semi-Markov Chain]
\label{def:semiMarkov}
Let $(J_n,T_n)_{n\geq0}$ be a Markov renewal chain with state space $\rond E\times \g N$.
Assume that $T_0=0$.
For any $n\in \g N$,  let $S_n$ be defined by
\[
S_n = \sum_{i=0}^n T_i.
\]
The \emph{semi-Markov chain} associated with $\left( J_n,T_n\right) _{n\geq0}$ is the $\rond E$-valued process $\left( Z_j\right) _{j\geq0}$ defined by 
\[
\forall j \hbox{ such that } S_n\leq j < S_{n+1}, \hskip 5mm Z_j = J_n.
\]
\end{defi}
Note that the sequence $\left( S_n\right) _{n\geq 0}$ is almost surely increasing because of the assumption $p_{a,b}(0) = 0$ (instantaneous transitions are not allowed) that guarantees that $T_n\geq 1$ almost surely, for any $n\geq 1$.

The $S_n$ are \emph{jump times}, the $T_n$ are \emph{sojourn times} in a given state and $Z_j$ stagnates at a same state between two successive jump times.
The process $J=\left( J_n\right) _n$, called the \emph{internal (or underlying) chain} of the semi-Markov chain $\left( Z_n\right) _n$, is a Markov chain on $\rond E$.
For this Markov chain, the transition probability between states $a$ and $b$ is the number $p_{a,b}=\sum_{k\geq 1} p_{a,b}(k)$.

Definitions~\ref{def:MRC} and~\ref{def:semiMarkov} make transitions of $J$ to the same state between time $n$ and time $n+1$ possible. Nevertheless, one can boil down to the case where $p_{a,a}(k)=0$ for all $a\in\rond E, k\in\g N$, thus obtaining a semi-Markov chain \emph{with true jumps}. 
Indeed, suppose that there exist some $a\in\rond E$ and $k\in\g N$ such that $p_{a,a}(k)\not= 0$ for a certain semi-Markov chain $\left( Z_n\right)$. 
Consider the chain $\left( Z'_n\right)$ obtained from $\left( Z_n\right)$ by forgetting the jumps to the same position.
It is the semi-Markov chain associated with the MRC $\left( J'_n,T'_n\right) _{
n\geq 0}$ defined by $T'_0=0$, $J'_0=J_0$ a.s. and by the following semi-Markov kernel $p'$:
for $a,b\in\rond E, a\not= b$, $p'_{a,b}(1) = p_{a,b}(1)$ and for $k\geq 2$,
\begin{align}
\label{jumpSM}
p'_{a,b}(k) &= \Proba\left( J_1' = b, T'_1 = k\big| J_0 = a \right)\\
&= p_{a,b}(k) + \sum_{i=1}^{k-1} p_{a,a}(i)p'_{a,b}(k-i)
\end{align}
(and thus $p'_{a,a}(k)=0$ for any $k\geq 0$).
Note that even if the semi-Markov chains $\left( Z_n\right)$ and $\left( Z'_n\right)$ do not have the same internal chains, they get the same trajectories.
It is worth noticing that the conditional expectations of $T_1$ and $T_1'$ are simultaneously finite or infinite. Indeed, a straightforward calculation from \eqref{jumpSM} leads to: for $a\in\rond E$,
\begin{equation}
\label{jumpSMtime}
\Espe\left( T'_1 | J'_0 = a \right)\times\left( 1-\sum_{i\geq 1}p_{a,a}(i) \right) = \Espe\left( T_1 | J_0 = a \right).
\end{equation}
Moreover, denoting $p_{a,b} = \sum_{k\geq 1} p_{a,b}(k)$ and $p'_{a,b} = \sum_{k\geq 1} p'_{a,b}(k)$, one gets $p'_{a,b} = \frac{p_{a,b}}{\sum_{c\not= a}p_{a,c}}$, as shortly mentioned in \cite{barbu/limnios/08}.
Since we make use of both versions of a semi-Markov chain in the paper -- with true jumps or not, it seemed important to us to devote these few lines to underline how they are connected.


\subsubsection{A semi-Markov chain induced by a stable VLMC}
\label{subsec:compareSM}

A stable VLMC always induces a semi-Markov chain, as described in the following.

Let $\left( U_n\right) _{n\geq 0}$ be a stable non-null VLMC and assume that $C_0 = \pref\left( U_0\right)$ is a finite context.
Recall that $\rond S$ denotes the set of context alpha-LIS of the VLMC.
For every $n\geq 0$, let $C_n$ be the context of $U_n$ and $Z_n$ be the alpha-LIS of $C_n$:
\begin{equation}
\label{def:Zn}
C_n := \pref\left( U_n\right)  \hskip 5mm \hbox{ and }  \hskip 5mm Z_n := \alpha_{C_n}s_{C_n}.
\end{equation}

\begin{figure}[h]
\begin{center}
{\scriptsize
\tikzstyle{fleche}=[->,>=latex,thick]
\tikzstyle{doublefleche}=[<->,>=latex]
\tikzstyle{alphaLIS}=[line width=2pt,color=blue]
\tikzstyle{pref}=[line width=2pt,color=orange]
\begin{tikzpicture}[scale=0.6]
\draw[fleche] (-20,1)--(-20,13);
\draw (-20.1,2)--(-19.9,2);
\draw (-20.5,2) node{$S_n$};
\draw (-20.1,4)--(-19.9,4);
\draw (-21,4) node{$S_n+1$};
\draw (-20.1,8)--(-19.9,8);
\draw (-21,8) node{$S_n+i$};
\draw (-20.1,10)--(-19.9,10);
\draw (-22.2,10) node{$S_{n+1}=S_n+i+1$};
\draw[doublefleche] (-18.2,2)--(-18.2,10);
\draw (-18.8,6) node{$T_{n+1}$};
\draw (-20.5,13.8) node{time};
\draw[dotted,thick] (-21,5.5)--(-21,6.3); 
\draw[alphaLIS] (-3,2)--(-8,2); 
\draw (0,2)--(-8,2);
\draw[dotted,thick] (0,2)--(2,2);
\draw (2.8,2) node{$U_{S_n}$};
\draw (-3,1.9)--(-3,2.1);
\draw (-7,1.9)--(-7,2.1);
\draw (-8,1.9)--(-8,2.1);
\draw (-5,1.7) node{$s$};
\draw (-7.5,1.7) node{$\alpha$};
\draw [decorate,decoration={brace,amplitude=8pt},yshift=0pt]
(-3,1.6) -- (-8,1.6) {}; 
\draw (-5.5,0.8) node{$C_{S_n}$};
\draw[alphaLIS] (-3,4)--(-8,4); 
\draw[pref] (-8,4)--(-9,4);
\draw (0,4)--(-9,4);
\draw (-3,3.9)--(-3,4.1);
\draw (-7,3.9)--(-7,4.1);
\draw (-8,3.9)--(-8,4.1);
\draw (-9,3.9)--(-9,4.1);
\draw[dotted,thick] (0,4)--(2,4);
\draw (3,4) node{$U_{S_n+1}$};
\draw (-5,3.7) node{$s$};
\draw (-7.5,3.7) node{$\alpha$};
\draw [decorate,decoration={brace,amplitude=8pt},yshift=0pt]
(-3,3.6) -- (-9,3.6) {}; 
\draw (-6,2.8) node{$C_{S_n+1}$};
\draw[dotted,thick] (-7,5.5)--(-7,6.3); 
\draw[alphaLIS] (-3,8)--(-8,8); 
\draw[pref] (-8,8)--(-15,8);
\draw (0,8)--(-15,8);
\draw[dotted,thick] (0,8)--(2,8);
\draw (3,8) node{$U_{S_n+i}$};
\draw (-3,7.9)--(-3,8.1);
\draw (-7,7.9)--(-7,8.1);
\draw (-8,7.9)--(-8,8.1);
\draw (-9,7.9)--(-9,8.1);
\draw (-14,7.9)--(-14,8.1);
\draw (-15,7.9)--(-15,8.1);
\draw (-5,7.7) node{$s$};
\draw (-7.5,7.7) node{$\alpha$};
\draw (-8.5,7.7) node{$\beta_1$};
\draw (-14.5,7.7) node{$\beta_i$};
\draw [decorate,decoration={brace,amplitude=8pt},yshift=0pt]
(-3,7.6) -- (-15,7.6) {}; 
\draw (-9,6.8) node{$C_{S_n+i}$};
\draw [decorate,decoration={brace,amplitude=8pt},yshift=0pt]
(-8,13) -- (-3,13) {}; 
\draw (-5.5,14) node{$J_{n} = Z_{j} = \alpha s$};
\draw (-5.5,15) node{$S_{n} \leq j \leq S_{n+1} -1$};
\draw[alphaLIS] (-12,10)--(-16,10);  
\draw (0,10)--(-16,10);
\draw[dotted,thick] (0,10)--(2,10);
\draw (3,10) node{$U_{S_{n+1}}$};
\draw (-12,9.9)--(-12,10.1);
\draw (-15,9.9)--(-15,10.1);
\draw (-16,9.9)--(-16,10.1);
\draw (-15.5,9.7) node{$\beta$};
\draw (-13.5,9.7) node{$t$};
\draw [decorate,decoration={brace,amplitude=8pt},yshift=0pt]
(-12,9.6) -- (-16,9.6) {}; 
\draw (-14,8.8) node{$C_{S_{n+1}}$};
\draw[alphaLIS] (-12,12)--(-16,12);   
\draw[pref] (-16,12)--(-17,12); 
\draw (0,12)--(-17,12);
\draw[dotted,thick] (0,12)--(2,12);
\draw (3.2,12) node{$U_{S_{n+1}+1}$};
\draw (-12,11.9)--(-12,12.1);
\draw (-15,11.9)--(-15,12.1);
\draw (-16,11.9)--(-16,12.1);
\draw (-17,11.9)--(-17,12.1);
\draw (-15.5,11.7) node{$\beta$};
\draw (-13.5,11.7) node{$t$};
\draw [decorate,decoration={brace,amplitude=8pt},yshift=0pt]
(-12,11.6) -- (-17,11.6) {}; 
\draw (-14,10.8) node{$C_{S_{n+1}+1}$};
\draw [decorate,decoration={brace,amplitude=8pt},yshift=0pt]
(-16,13) -- (-12,13) {}; 
\draw (-14,14) node{$J_{n+1} = Z_{S_{n+1}} = \beta t$};
\end{tikzpicture}
}
\end{center}
\caption{Evolution of a VLMC $\left( U_j\right)$ between two ``jumping'' times $S_n$ and $S_{n+1}$. 
On the figure, right-infinite words $U_{S_n}, U_{S_n+1}, \dots$ grow to the left when time increases from the bottom to the top. 
Their respective \emph{cont}'s (which are contexts) are $C_{S_n}, C_{S_n+1}, \dots$, they are colored.
On the figure, the successive alpha-LIS are marked in blue, they stagnate at $\alpha s$ during the time $T_{n+1} = S_{n+1} - S_n$ and jump at $\beta t$ at time $S_{n+1} =S_n+i+1$ . 
\label{fig:SM}}
\end{figure}
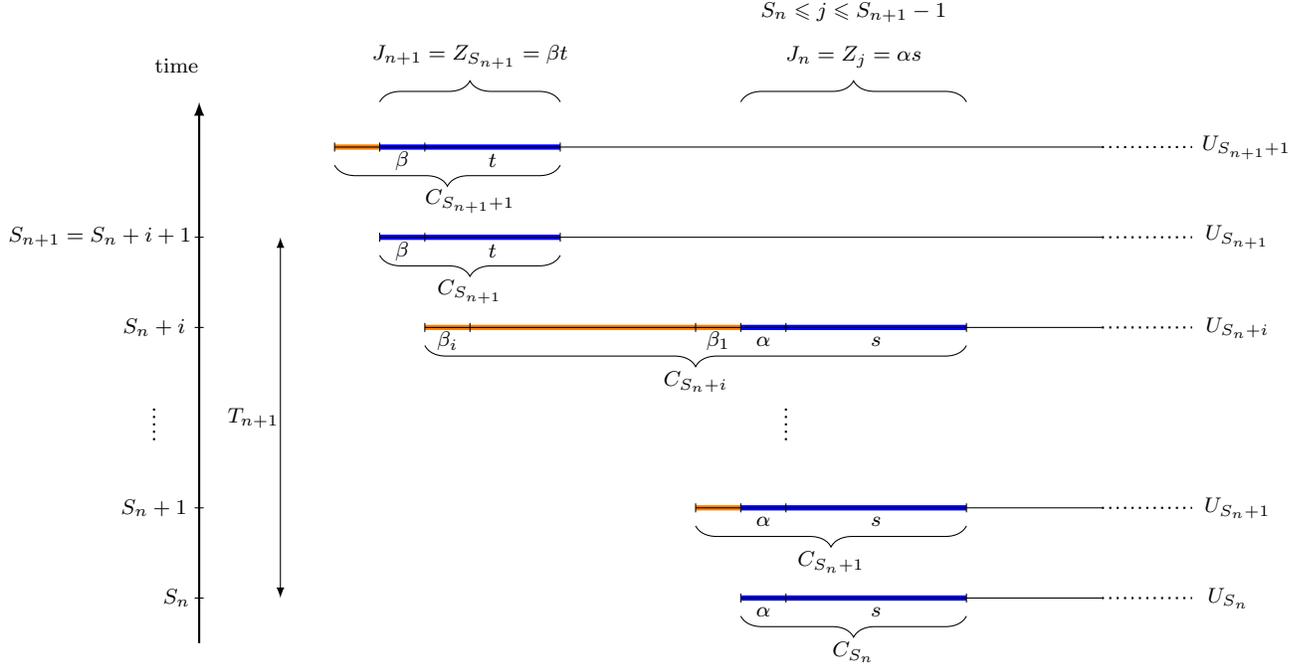

Let us describe the evolution of these two processes, when the VLMC $\left( U_j\right) _j$ is growing by adding successively a letter on the left. 
One can refer to Figure~\ref{fig:SM} as a visual support of this description.
For $j\geq 0$, assume that $C_j = \cdots [\alpha s]$ has $\alpha s$ as an alpha-LIS.
When adding a letter $\beta$, two cases can occur (recall that since the context tree is stable, if $c$ is a context and $\beta\in\rond A$, then $\beta c$ is non-internal -- see Proposition~\ref{pro:defstable}(ii)):

-- either $\beta C_j$ is a context and then $C_{j+1} = \beta C_j = \cdots [\alpha s]$.
In this case the process $Z$ stagnates at $\alpha s$;

-- or $\beta C_j$ is \emph{not} a context and then by Lemma \ref{lem:fondLem}, $C_j$ begins with some LIS $t$ and $\beta t$ is a context being its own alpha-LIS.
In that case, $C_{j+1} = \beta t$ and $Z$ jumps at $\beta t$. 
Notice that the term \emph{jumps} is not completely adequate because $\alpha s = \beta t$ could occur.
With this evolution in mind, let $\left( S_n\right) _{n\geq 0}$ be the increasing sequence of times defined by $S_0 = 0$ and for any $n\geq 1$,
\begin{equation}
\label{def:Sn}
S_n := \inf \set{ k>S_{n-1}, \left| C_k\right|\leq\left| C_{k-1}\right|},
\end{equation}
with the usual convention that it equals $+\infty$ whenever $\forall k> S_{n-1}, \left| C_{k}\right| > \left| C_{k-1}\right|$.
Let also $T_0=0$ and, for every $n\geq 1$, denote by $T_n$ the difference
\begin{equation}
\label{def:Tn}
{T}_n := S_n - S_{n-1}.
\end{equation}
Finally, for any $n\geq 0$, let 
\begin{equation}
\label{def:Jn}
{J}_n := Z_{S_n}.
\end{equation}
With these notations, the processes $\left( T_n\right) _{n\geq 0}$ and  $\left( J_n\right) _{n\geq 0}$ evolve as follows.
Assume that $J_n = C_{S_n}= Z_{S_n} = \alpha s \in\rond S$ for some $n\geq 0$. 
For $i\geq 1$, when adding a letter $\beta$, as long as $\beta C_{S_n+i-1}$ remains a context, then 
$Z_{S_n+i} = Z_{S_n} = \alpha s  = J_n$.
The first time when $\beta C_{S_n+i}$ is not a context (we shall see that this occurs almost surely if and only if Assumption~\eqref{hypck} is fulfilled), then $S_{n+1} = S_n +i$, $C_{S_n + i} = \beta t\in\rond S$ and $J_{n+1} = \beta t$. 
It turns out that $\left( Z_n\right) _{n\geq 0}$ is a semi-Markov chain having $\left( J_n, T_n\right) _{n\geq 0}$ as an underlying (Markov renewal) chain, as specified in the following proposition.

\begin{pro}
\label{pro:compareSM}
Let $\left( U_n\right) _{n\geq 0}$ be a stable non-null VLMC such that
\begin{equation}
\label{hypck}
\forall\alpha s\in\rond S,~~
\lim_{k\rightarrow\infty}  \kappa_{\alpha s}(k) = 0,
\end{equation}
where $ \kappa_{\alpha s}(k) $ is defined in \eqref{def:ck}.
Assume that $C_0=\pref\left( U_0\right)$ is a finite word.
Then with the above notations~\eqref{def:Zn}, \eqref{def:Sn}, \eqref{def:Tn} and \eqref{def:Jn},
\begin{itemize}
\item[(i)]
 $S_n$ and $T_n$ are almost surely finite.
 Furthermore, for every $\alpha s\in\rond S$ and every $n\geq 1$,
 \[
 \Espe\left( T_{n} \big| J_{n-1} = \alpha s\right) = \kappa_{\alpha s}\in [0,+\infty ].
 \]
 See~\eqref{defKappa} where $\kappa_{\alpha s}=\sum _{k\geq 1}\kappa_{\alpha s}(k)$ is defined;
\item[(ii)]
the jump times $S_n$ can also be written $S_n =  \inf \set{ k> S_{n-1}, C_k \in\rond S}$;
\item[(iii)]
$\left( Z_n\right) _{n\geq 0}$ is an $\rond S$-valued semi-Markov chain associated with the Markov renewal  chain $\left( J_n,T_n\right) _{n\geq 0}$. 
The associated semi-Markov kernel writes: $\forall \alpha s, \beta t \in\rond S$, $\forall k\geq 1$,
\[
p_{\alpha s, \beta t}(k) = \sum_{\substack{c\in\rond C,~c=t\cdots\\[2pt]c= \cdots [\alpha s]\\[2pt]|c| = |\alpha s| + k-1}} \casc\left(\beta c\right) .
\]
Moreover, $Q$ is the transition matrix of the $\rond S$-valued Markov chain $\left( J_n\right) _{n\geq 0}$.
\end{itemize}
\end{pro}
One can find a proof of Proposition~\ref{pro:compareSM} on page~\pageref{proof:compareSM}.

\begin{rem}
\label{rem:semiMarkov}
The semi-Markov chain $\left( Z_n\right)$ contains less information than the chain $\left( U_n\right)$. 
To illustrate this, here is an example with a finite context tree on the alphabet $\set{0,1}$.

\vskip 5pt
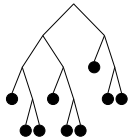
\begin{minipage}{0.4\textwidth}
\centering
\begin{tikzpicture}[scale=0.4]
		\tikzset{every leaf node/.style={draw,circle,fill},every internal node/.style={draw,circle,scale=0.01}}
\Tree [.{} [.{} [.{} {} [.{} {} {} ] ] [.{} {} [.{} {} {} ] ] ] [.{} {} [.{} {} {} ] ] 
	]
\end{tikzpicture}
\end{minipage}
\begin{minipage}{0.6\textwidth}
\centering
\begin{tabular}{r|l}
alpha-LIS $\alpha s$&contexts having $\alpha s$ as an alpha-LIS  \\
\hline
10&10,010,110,0010,0110\\
000&000\\
111&111,0111\\
0011&0011
\end{tabular}
\end{minipage}
\vskip 5pt
In this example, 0010 and 0110 are two contexts of the same length, with the same context alpha-LIS 10 and beginning by the same context LIS 0. Hence if we know that ${ J}_n=10$, ${ S}_{n+1}-S_{n}=3$ and ${J}_{n+1}=10$, then ${ Z}_j$ is uniquely determined between the two successive jump times, whereas
there are two possibilities to reconstruct the VLMC $\left( U_n\right)$. With the notations above, there are two cascade terms in $p_{10,10}(3) $:
\begin{align*}
p_{10,10}(3) &= \Proba\left(C_{S_n+1} = 010, C_{S_n+2} = 0010,C_{S_n+3} = 10010 | C_{S_n} =10\right)\\
& \ \ + \Proba\left(C_{S_n+1} = 110, C_{S_n+2} = 0110,C_{S_n+3} = 10110 | C_{S_n} =10\right)\\
&= q_{10}(0) q_{010}(0) q_{0010}(1) + q_{10}(1) q_{110}(0) q_{0110}(1)\\
&= \casc (10010) + \casc (10110).
\end{align*}

\vskip 5pt
\end{rem}
\subsection{Properties of $Q$ in the stable case}
\label{subsection:properties_Q}

For a given probabilised context tree, the matrix $Q$, that has been defined in Section~\ref{subsec:Q} by Formula~\eqref{defQ}, plays a central role in our main Theorem~\ref{fQbij}.
In the case of stable trees, Proposition~\ref{pro:compareSM} gives a probabilistic interpretation of $Q$ as the transition matrix of some Markov chain.
This section is devoted to gathering properties of $Q$
(or of the Markov chain $Q$ is the transition matrix of).

\begin{defi}
\label{def:rowStochastic}
A square (finite or denumerable) matrix $\left( a_{r,c}\right)_{r,c}$ having non-negative entries  is said to be \emph{row-stochastic} whenever all its rows (are summable and) sum to $1$, \emph{i.e.}
\[
\forall r,~\sum _{c}a_{r,c}=1.
\]
\end{defi}

The following assertion is a consequence of Proposition~\ref{pro:compareSM}, (iii).
Remember that the numbers $\kappa_{\alpha s}(k)$ are defined by~\eqref{def:ck}.
Notice also that one can also make a direct combinatorial proof using Lemma~\ref{lem:fondLem}.

\begin{pro}
\label{pro:stochasticity}
Let $\left(\rond T,q\right)$ be a stable probabilised context tree.
Assume that
\begin{equation}
\label{hypckBis}
\forall\alpha s\in\rond S,~~
\lim_{k\rightarrow\infty}  \kappa_{\alpha s}(k) = 0.
\end{equation}
Then, the matrix $Q$ has finite entries and is row-stochastic.
\end{pro}

The row-stochasticity of $Q$ writes
\[
\forall \alpha s\in\rond S,~\sum _{\beta t\in\rond S}Q_{\alpha s,\beta t}=1.
\]

\begin{rem}
\label{rem:realisation}
Any stochastic matrix with strictly positive coefficients $A=\left( a_{i,j}\right) _{i\geqslant 0,j\geqslant 0}$ is the matrix~$Q$ associated with some non-null probabilised stable context tree.
It may be realised for instance with a left-comb of left-combs as follows.

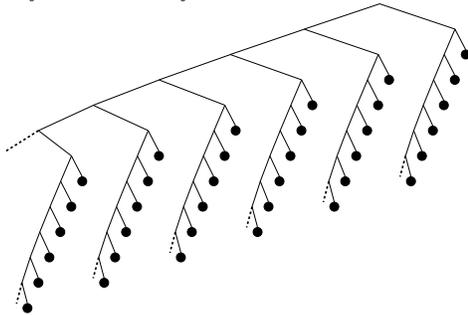
\begin{minipage}{0.6\textwidth}
\centering
\begin{tikzpicture}[scale=0.32]
		\tikzset{every leaf node/.style={draw,circle,fill},every internal node/.style={draw,circle,scale=0.01}}
\Tree [.{}
	[.{}
		[.{}
			[.{} 
					[.{} 
						[.{} \edge[line width=2pt,dashed];\node[fill=white,draw=white]{};  [.{} 
		[.{}
			[.{}
				[.{}
					[.{} 
						[.{} \edge[line width=2pt,dashed];\node[fill=white,draw=white]{}; {} ]
						{} ]
					{} ]
			{} ]
		{} ]
	{} ]]
						[.{} 
		[.{}
			[.{}
				[.{}
					[.{} 
						[.{} \edge[line width=2pt,dashed];\node[fill=white,draw=white]{}; {} ]
						{} ]
					{} ]
			{} ] 
		{} ]
	{} ]
					]
				[.{} 
		[.{}
			[.{}
				[.{}
					[.{} 
						[.{} \edge[line width=2pt,dashed];\node[fill=white,draw=white]{}; {} ]
						{} ]
					{} ]
			{} ] 
		{} ]
	{} ]
 			]
		[.{} 
		[.{}
			[.{}
				[.{}
					[.{} 
						[.{} \edge[line width=2pt,dashed];\node[fill=white,draw=white]{}; {} ]
						{} ]
					{} ]
			{} ] 
		{} ]
	{} ]
		] [.{} 
		[.{}
			[.{}
				[.{}
					[.{} 
						[.{} \edge[line width=2pt,dashed];\node[fill=white,draw=white]{}; {} ]
						{} ]
					{} ]
			{} ] 
		{} ]
	{} ]
	]	 
	[.{} 
		[.{}
			[.{}
				[.{}
					[.{} 
						[.{} \edge[line width=2pt,dashed];\node[fill=white,draw=white]{}; {} ]
						{} ]
					{} ]
			{} ] 
		{} ]
	{} ]
	]
\end{tikzpicture}
\end{minipage}
\begin{minipage}{0.39\textwidth}
The \emph{left-comb of left-combs} is the context tree on the alphabet $\set{0,1}$ as drawn on the left:
the finite contexts are the $0^p10^q1$, $p,q\geq 0$.
A left-comb of left-combs is a stable context tree.
Its has infinitely many infinite branches, namely $0^\infty$ and the $0^p10^\infty$, $p\geq 0$.
\end{minipage}

\vskip 10pt
For any $p,q\geq 0$, the alpha-LIS of $0^p10^q1$ is $10^q1$.
In particular, the set $\rond S$ of alpha-LIS of contexts is infinite.
In this case, for any $q\geq 0$, the set of contexts having $10^q1$ as an alpha-LIS is also infinite.

Probabilise this context tree by a family $\left( q_c\right) _c$ of probability measures on $\{ 0,1\}$.
Denote, for every $q,p\geq 0$,
\[
c_{q,p}=\casc\left( 0^p10^q1\right)
=\prod _{0\leq k\leq p-1}q_{0^k10^q1}(0).
\]
Assumption~\eqref{hypckBis} is equivalent to $c_{q,p}$ converging to $0$ when $p$ tends to $\infty$, for any $q$.
The square matrix~$Q$ is infinite and, under the latter assumption, its entries write
\[
Q_{10^q1,10^p1}=\casc\left( 10^p10^q1\right) =c_{q,p}-c_{q,p+1}.
\]

A row-stochastic positive infinite matrix $A$ being given, a simple calculation shows that if one defines the probability measures $q_{0^p10^q1}$ of a left-comb of left-combs by
\[
	q_{0^p10^q1}(1)=\frac{a_{q,p}}{1-\sum_{k=0}^{p-1}a_{q,k}},
\]
then $Q_{10^q1,10^p1}=a_{q,p}$.
The question whether any stochastic matrix (with some zero coefficients) can be realized as the $Q$ matrix of some non-null stable VLMC seems to be more difficult.
Namely, zero coefficients in $Q$ assuming non-zero $q_c(\alpha)$ constraint the shape of the context tree.
\end{rem}

\begin{pro}
\label{pro:irreducibility}
Let $\left(\rond T,q\right)$ be a non-null stable probabilised context tree.
Then the matrix $Q$ is irreducible.
\end{pro}

See Section~\ref{sec:proofs} page~\pageref{proof:irreducibility} for a proof of this proposition.

\subsection{Stationary measure for a stable VLMC vs recurrence of $Q$}
\label{subsection:existence_stable}

The following result links the existence and the uniqueness of a stationary probability measure of a VLMC to the recurrence of $Q$.
Let us recall the definition of recurrence and state a necessary and sufficient condition to get a (unique) invariant probability measure for stable trees.
In the sequel, a \emph{stochastic} matrix is a row-stochastic one -- see Definition~\ref{def:rowStochastic}.
Note that the powers of a stochastic matrix are well defined and also stochastic.

\begin{defi}
Let $A=\left( a_{i,j}\right) _{i,j}$ be a stochastic irreducible countable matrix.
Denote by $a_{i,j}^{(k)}$ the $(i,j)$-th entry of the matrix $A^k$.
The matrix $A$ is \emph{recurrent} whenever there exists $i$ such that
\[
\sum_{k=1}^{\infty}a_{i,i}^{(k)}=1.
\]
\end{defi}
Any stochastic irreducible countable matrix may be viewed as the transition matrix of an irreducible Markov chain with countable state space. 
The recurrence means that there is a state $i$ (and this is true for every state because of irreducibility) for which the first return time is a.s. finite.
When in addition the expectation of the return times are finite, the matrix is classically called \emph{positive recurrent}.

\begin{theo}
\label{th:stable}
Let $(\rond T,q)$ be a non-null probabilised context tree.
Assume that $\rond T$ is stable.
Then, the following assertions are equivalent.
\begin{enumerate}
\item The VLMC associated with $(\rond T,q)$ has a unique stationary probability measure
\item The VLMC associated with $(\rond T,q)$ has at least a stationary probability measure
\item The three following conditions are satisfied:
\begin{itemize}
\item[$(c_1)$] the cascade series~\eqref{convCasc} converge
\item[$(c_2)$] $Q$ is recurrent
\item[$(c_3)$] $\sum _{\alpha s\in\rond S} v_{\alpha s}\kappa_{\alpha s}<+\infty$, where $\left( v_{\alpha s}\right) _{\alpha s}$ is the unique non-negative left-fixed vectors of $Q$, up to multiplication by a positive real number.
\end{itemize}
\end{enumerate}
\end{theo}

A proof of Theorem~\ref{th:stable} is given in Section~\ref{sec:proofsStable}, page~\pageref{proof:stable}.
Notice that Theorem~\ref{th:stable} is a direct consequence of Theorem~\ref{fQbij} and of the fact that $Q$ is stochastic.
In the present article, the stochasticity of $Q$ is deduced from its interpretation as the transition matrix of some semi-Markov chain (Proposition~\ref{pro:compareSM}).
Notice, as already mentioned just before Proposition~\ref{pro:stochasticity}, that this stochasticity can also be proved by a direct combinatorial proof.
In this sense, Theorem~\ref{th:stable} can be understood as being independent from the fact that the process $\left( Z_n\right) _n$ of successive context alpha-LIS of the VLMC $\left( U_n\right) _n$ is a semi-Markov chain (our current notations).

\begin{rem}
Actually, as shown in the end of the proof, when $Q$ is recurrent and when the series $\sum _{\alpha s\in\rond S} v_{\alpha s}\kappa_{\alpha s}$ converges, then $Q$ is positive recurrent.
Furthermore, all the $v_{\alpha s}$ are then positive, thanks to Lemma~\ref{lem:pineq0}.
\end{rem}

\begin{rem}
There exist non-null stable probabilised context trees such that $(c_1)$ and $(c_2)$ are fulfilled, but not $(c_3)$, hence with no stationary probability measure. 
Here is an example based on a left-comb of left-combs, already introduced in Remark~\ref{rem:realisation}.

\vskip 5pt
Let $v_p=\frac 1{p+1}-\frac 1{p+2}$ and $R_p=\sum _{q\geq p}v_q=\frac 1{p+1}$ for every $p\geq 0$
(more generally, on can build similar examples based on positive sequences $\left( v_p\right) _p$ such that $\sum _{p\geq 0}v_p=1$ and $\sum _ppv_p$ diverges).
Define $S$ by
\[
S(x)=\sum _{q\geq 0}v_qx^{\frac 1{q+1}}.
\]
This series is normally convergent on the real interval $[0,1]$ so that $S$ is continuous on $[0,1]$ and satisfies $S(0)=0$ and $S(1)=1$.
Furthermore, $S$ is derivable and increasing on $[0,1]$ since the derived series converges normally on any compact subset of $]0,1]$.
Finally, $S(x)\geq v_qx^{\frac 1{q+1}}$ on $[0,1]$ for every $q\geq 0$.
Consequently, for every $t>0$, there exists $C_t>0$ such that
\begin{equation}
\label{S-1plate}
\forall x\in [0,1],~~S^{-1}(x)\leq C_tx^t.
\end{equation}
Take now the probabilised left-comb of left-combs defined by the relations (see notations in Remark~\ref{rem:realisation})
\[
\forall q,p\geq 0,~c_{q,p}=S^{-1}\left( R_p\right)^{\frac 1{q+1}} .
\]
Note that these equations fully define the corresponding VLMC because the probabilities $q_{0^p10^q1}$ are characterized by these $c_{q,p}$ \emph{via} the equalities $q_{0^p10^q1}(0)=c_{q,p+1}/c_{q,p}$.
The definition of $S$ implies that $\sum _{q\geq 0}v_qc_{q,p}=R_p$ for every $p\geq 0$, which precisely means that $v=vQ$
(the row-vector $v$ is a left-fixed vector for $Q$).
Besides, for any $q\geq 0$, applying~\eqref{S-1plate} for $t=2(q+1)$ leads to inequalities
\[
\forall p\geq 0,~c_{q,p}\leq C_{2(q+1)}\left(\frac 1{p+1}\right)^2.
\]

Thus, the positive sequences $\left( v_q\right) _q$ and $\left( c_{q,p}\right) _{p,q}$ satisfy the following properties.
\begin{enumerate}
\item $\forall q\geq 0, \sum_p c_{q,p}<\infty$,
\item $\forall p\geq 0, \sum_{q\geq 0} v_qc_{q,p}=\sum_{q\geq p}v_q$,
\item $\sum_q v_q<\infty$,
\item $\sum_{q,p\geq 0}v_qc_{q,p}=+\infty$.
\end{enumerate}
In terms of the VLMC, with general notations of Section~\ref{subsec:cascade}, these properties translate into:
\begin{enumerate}
\item the cascade series converge (for $\alpha s=10^q1$, $\kappa_{\alpha s}=\sum_p c_{q,p}$),
\item $v=\left( v_{\alpha s}\right) _{\alpha s\in\rond S}$ is a left-fixed vector for $Q$,
\item $\sum_{\alpha s\in\rond S} v_{\alpha s}<\infty$,
\item $\sum_{\alpha s\in\rond S} v_{\alpha s}\kappa_{\alpha s}=+\infty$.
\end{enumerate}
	Therefore, $(c_1)$ is fulfilled and $(c_3)$ is not. Finally, the stability of the context tree and the convergence of cascade series imply the stochasticity of $Q$  by Proposition~\ref{pro:stochasticity}, which force the vector $u=(1,1,\dots,1,\dots)^{\top}$ to be a right-fixed vector for $Q$. Moreover, $\langle v,u\rangle =\sum_{\alpha s\in\rond S} v_{\alpha s}<\infty$. Remarking that $Q$ is aperiodic (for it is strictly positive) and using Remark 7.1.17 p. 207 of \cite{kitchens/98}, this implies the positive recurrence of $Q$. 
\end{rem}

\begin{rem}
	One may wonder whether $(c_1)\Longrightarrow (c_2)$. The answer is no. There exists a VLMC defined by a stable tree such that the cascade series converge and the matrix $Q$ is transient.

To build such an example, recall that, by Remark \ref{rem:realisation}, any stochastic matrix with strictly positive coefficients can be realized as the matrix $Q$ of a stable tree (take for example a left-comb of left-combs). The matrix $A=(a_{i,j})_{i\geqslant 1,j\geqslant 1}$ defined by
\begin{itemize}
\item $a_{i,i+1}=1-\frac{1}{(i+1)^2}$ for all $i\geqslant 1$,
\item $a_{i,j}=\frac{1}{(i+1)^22^{j-1}}$ if $j\geqslant i+2$,
\item $a_{i,j}=\frac{1}{(i+1)^22^{i+1-j}}$ if $j\leqslant i$
\end{itemize}
is stochastic and transient.
Indeed, if one associates a Markov chain to the stochastic matrix $A$ and if one denotes by $T_1$ the return time to the first state,
\[
\Proba(T_1=\infty)\geqslant\prod_{i\geqslant 1}a_{i,i+1}\geqslant\prod_{i\geqslant 2}\left(1-\frac{1}{i^2}\right) =\frac12.
\]
Consider now the VLMC defined by a left-comb of left-combs probabilised in the unique way such that $Q_{10^q1,10^p1}=a_{q,p}$ for every $(p,q)$, like in Remark~\ref{rem:realisation}.
A simple computation shows that the series of cascade converges (geometrically).
Simultaneously, since $Q$ is transient, Theorem~\ref{th:stable} shows that the VLMC admits no stationary probability measure.
\end{rem}

Notice that Theorem~\ref{th:stable} also provides results for non-stable trees as the following corollary shows, using Remark~\ref{rem:stabilized}.

\begin{cor}
	\label{cor:stabilized}
	Let $\left(\rond T,q\right)$ be a non-null probabilised context tree. Suppose that $\rond T$ is stabilizable and denote by $\widehat{\rond T}$ its stabilized.
	Using the notations of Remark~\ref{rem:stabilized}, if $\left(\widehat{\rond T},\widehat{q}\right)$ satisfies the conditions of Theorem~\ref{th:stable}, then the VLMC associated with $\left(\rond T,q\right)$ admits a unique invariant probability measure. If not, it does not admit any invariant probability measure.
	In particular, a VLMC associated to a stabilizable context tree never admits several stationary probability measures.
\end{cor}

\vskip 10pt
When the matrix $Q$ is finite dimensional, stochastic and irreducible, it admits a unique left-fixed vector up to scalar multiplication.
This leads to the following theorem.

\begin{theo}[finite number of alpha-LIS]
\label{cor:finite}
Let $(\rond T,q)$ be a non-null probabilised context tree and $U=\left( U_n\right) _n$ be the VLMC it defines.
Assume that $\rond T$ is stable and that $\#\rond S<\infty$.
Then (i), (ii) and (iii) are equivalent.

(i) $U$ admits at least a stationary probability measure.

(ii) $U$ admits a unique stationary probability measure.

(iii) The cascade series~\eqref{convCasc} converge.

Moreover, whenever one of the previous assertion is true then, for every distribution of $U_0$ that does not charge any infinite context, for every finite word $w$,
\[
\Proba\left( U_n\in w\rond R\right)
\tooSous{n\to\infty}
\pi\left( w\rond R\right)
\]
where $\pi$ denotes the unique $U$-invariant probability measure.
\end{theo}

The proof of Theorem~\ref{cor:finite} is made in Section~\ref{sec:proofs}, page~\pageref{proof:finite}.

\begin{rem}[Case of finite trees]
\label{rem:finite}
Assume that  $U$ is a non-null VLMC defined by a finite context tree. One gets an equivalent process $\widehat U$ by properly probabilising the \emph{stabilized} context tree -- see Remark~\ref{rem:stabilized}.
Since there are finitely many contexts, all the cascade series converge -- they are all finite sums.
Then, Theorem~\ref{cor:finite} applies, showing that $\widehat U$ -- thus $U$ -- always admits a unique stationary probability measure.
This is not surprising because in that case,  $U$ can be seen as an ordinary irreducible Markov chain whose order is the height of its context tree -- see Remark~\ref{rem:notMarkov}.
\end{rem}

The following example shows how one can apply Theorem~\ref{cor:finite}.

\begin{exa}
\label{exa:pgpd}
The so-called \emph{left-comb of right-combs} is particularly simple because if has only one context alpha-LIS.
The left-comb of right-combs augmented by a cherry stem, a variation of the former one, gets four context alpha-LIS.
Because of Theorem~\ref{cor:finite}, both corresponding VLMC have a (unique) stationary probability measure if and only if their cascade series converge.

\vskip 10pt
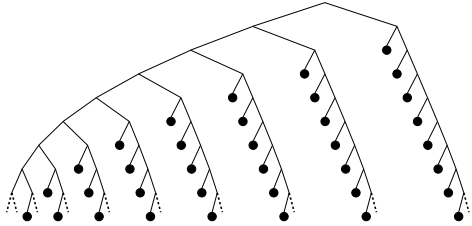
\begin{minipage}{0.4\textwidth}
\centering
\begin{tikzpicture}[scale=0.3]
\tikzset{
every leaf node/.style={draw,circle,fill},
every internal node/.style={draw,circle,scale=0.01}
}
\Tree [.{} 
[.{} 
[.{} 
[.{} 
[.{} 
[.{} 
[.{} 
[.{} 
[.{} 
\edge[line width=2pt,dashed];\node[fill=white,draw=white]{};
\edge[line width=2pt,dashed];\node[fill=white,draw=white]{};
]
[.{} 
{} 
\edge[line width=2pt,dashed];\node[fill=white,draw=white]{};
]
]
[.{} 
{} 
[.{} 
{} 
\edge[line width=2pt,dashed];\node[fill=white,draw=white]{};
]
]
]
[.{} 
{} 
[.{} 
{} 
[.{} 
{} 
\edge[line width=2pt,dashed];\node[fill=white,draw=white]{};
]
]
]
]
[.{} 
{} 
[.{} 
{} 
[.{} 
{} 
[.{} 
{} 
\edge[line width=2pt,dashed];\node[fill=white,draw=white]{};
]
]
]
]
]
[.{} 
{} 
[.{} 
{} 
[.{} 
{} 
[.{} 
{} 
[.{} 
{} 
\edge[line width=2pt,dashed];\node[fill=white,draw=white]{};
]
]
]
]
]
]
[.{} 
{} 
[.{} 
{} 
[.{} 
{} 
[.{} 
{} 
[.{} 
{} 
[.{} 
{} 
\edge[line width=2pt,dashed];\node[fill=white,draw=white]{};
]
]
]
]
]
]
]
[.{} 
{} 
[.{} 
{} 
[.{} 
{} 
[.{} 
{} 
[.{} 
{} 
[.{} 
{} 
[.{} 
{} 
\edge[line width=2pt,dashed];\node[fill=white,draw=white]{};
]
]
]
]
]
]
]
]
[.{} 
{} 
[.{} 
{} 
[.{} 
{} 
[.{} 
{} 
[.{} 
{} 
[.{} 
{} 
[.{} 
{} 
[.{} 
{} 
\edge[line width=2pt,dashed];\node[fill=white,draw=white]{};
]
]
]
]
]%
]
]
]
]
\end{tikzpicture}

\end{minipage}
\hskip 10pt
\begin{minipage}{0.58\textwidth}
The left-comb of right-combs, built on the alphabet $\set{0,1}$, is drawn on the left.
Its finite contexts are the $0^p1^q0$, $p\geq 0$, $q\geq 1$.
It has infinitely many infinite branches, namely the $0^p1^\infty$, $p\geq 0$.
This context tree is stable and all finite contexts have $10$ as an alpha-LIS.
The matrix~$Q$, which is thus $1$-dimensional, is reduced to~$(1)$.
The convergence of the unique cascade series consists in the summability of the double sum
\[
\sum _{p\geq 0,q\geq 1}~\prod _{j=0}^{p-1}q_{0^j1^q0}(0)\prod _{k=1}^{q-1}q_{1^k0}(1).
\]
\end{minipage}

\vskip 10pt
\begin{minipage}{0.6\textwidth}
The left-comb of right-combs with a cherry stem consists in simply replacing the context $10$ of the preceding tree by the cherries $100$ and $101$.
The tree is still stable and it has four context alpha-LIS, as resumed in the array.
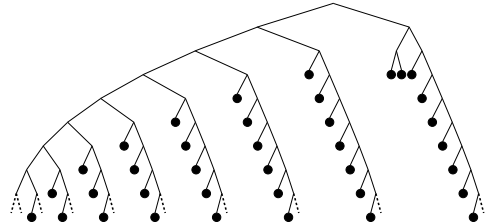
\begin{center}
\begin{tabular}{r|l}
alpha-LIS $\alpha s$&contexts having $\alpha s$ as an alpha-LIS\\
\hline
$100$&$100$\\
$101$&$101$\\
$010$&$0^p10$, $p\geq 1$\\
$110$&$0^p1^q0$, $p\geq 0$, $q\geq 2$
\end{tabular}
\end{center}
\end{minipage}
\begin{minipage}{0.4\textwidth}
\centering
\begin{tikzpicture}[scale=0.3]
\tikzset{
every leaf node/.style={draw,circle,fill},
every internal node/.style={draw,circle,scale=0.01}
}
\Tree [.{} 
[.{} 
[.{} 
[.{} 
[.{} 
[.{} 
[.{} 
[.{} 
[.{} 
\edge[line width=2pt,dashed];\node[fill=white,draw=white]{};
\edge[line width=2pt,dashed];\node[fill=white,draw=white]{};
]
[.{} 
{} 
\edge[line width=2pt,dashed];\node[fill=white,draw=white]{};
]
]
[.{} 
{} 
[.{} 
{} 
\edge[line width=2pt,dashed];\node[fill=white,draw=white]{};
]
]
]
[.{} 
{} 
[.{} 
{} 
[.{} 
{} 
\edge[line width=2pt,dashed];\node[fill=white,draw=white]{};
]
]
]
]
[.{} 
{} 
[.{} 
{} 
[.{} 
{} 
[.{} 
{} 
\edge[line width=2pt,dashed];\node[fill=white,draw=white]{};
]
]
]
]
]
[.{} 
{} 
[.{} 
{} 
[.{} 
{} 
[.{} 
{} 
[.{} 
{} 
\edge[line width=2pt,dashed];\node[fill=white,draw=white]{};
]
]
]
]
]
]
[.{} 
{} 
[.{} 
{} 
[.{} 
{} 
[.{} 
{} 
[.{} 
{} 
[.{} 
{} 
\edge[line width=2pt,dashed];\node[fill=white,draw=white]{};
]
]
]
]
]
]
]
[.{} 
{} 
[.{} 
{} 
[.{} 
{} 
[.{} 
{} 
[.{} 
{} 
[.{} 
{} 
[.{} 
{} 
\edge[line width=2pt,dashed];\node[fill=white,draw=white]{};
]
]
]
]
]
]
]
]
[.{} 
[.{} 
{} 
{} 
]
[.{} 
{} 
[.{} 
{} 
[.{} 
{} 
[.{} 
{} 
[.{} 
{} 
[.{} 
{} 
[.{} 
{} 
\edge[line width=2pt,dashed];\node[fill=white,draw=white]{};
]
]
]
]
]%
]
]
]
]
\end{tikzpicture}

\end{minipage}

\vskip 5pt
In this last example, the convergence of the cascade series is equivalent to the finiteness of both sums
\[
\kappa _{010}=\sum _{p\geq 1}\prod _{k=1}^{p-1}q_{0^k10}(0)
{\rm ~~and~~}
\kappa _{110}=\sum _{p\geq 0,~q\geq 2}~\prod _{j=0}^{p-1}q_{0^j1^q0}(0)\prod _{k=2}^{q-1}q_{1^k0}(1).
\]

\end{exa}

\subsection{A semi-Markov chain is a stable VLMC}
\label{subsec:resultsSM}

In this section, it is shown that any semi-Markov chain on a finite state space is a VLMC associated with some particular infinite stable probabilised context tree.
Consequently, one deduces from Theorem \ref{cor:finite} a necessary and sufficient condition for a non-null semi-Markov chain to admit a limit distribution. This condition already appears in \cite{barbu/limnios/08}.

\begin{defi}
\label{def:bcomb}
If $b\geq 2$, the \emph{$b$-comb}
 is the context tree on an alphabet $\rond A$ of cardinality $b$ having $\set{\alpha ^k\beta :\alpha,\beta\in\rond A, \alpha\neq\beta,~k\geq 1}$ as a set of finite contexts.
\end{defi}
\begin{figure}[h]
\begin{center}
\newcommand{\segGeom}[5]{\draw [#5] (0,0) plot[domain=0:#3]({#1+\x*cos(#4)},{#2+\x*sin(#4)});}
\newcommand{\pointCart}[4]{\fill [#4] (0,0) plot [domain=0:360]({#1+#3*cos(\x)},{#2+#3*sin(\x)});}
\newcommand{\pointGeom}[6]{\fill [#6] (0,0) plot [domain=0:360]({#1+#3*cos(#4)+#5*cos(\x)},{#2+#3*sin(#4)+#5*sin(\x)});}

\begin{tikzpicture}[scale=0.8]
\newcommand{\gdePenteN}{1.8}
\newcommand{\gdePenteE}{0.6}
\newcommand{\gdePenteW}{-\gdePenteE}
\newcommand{\gdePenteS}{-\gdePenteN}
\newcommand{\noeud}[2]{-#1*#2,-#1}
\newcommand{\absNoeud}[2]{-#1*#2}
\newcommand{\titePenteNE}{1}
\newcommand{\titePenteNW}{0.6}
\newcommand{\titePenteNS}{0.2}
\newcommand{\titePenteEN}{0.9}
\newcommand{\titePenteEW}{0.1}
\newcommand{\titePenteES}{-0.3}
\newcommand{\titePenteWN}{-\titePenteES}
\newcommand{\titePenteWE}{-\titePenteEW}
\newcommand{\titePenteWS}{-\titePenteEN}
\newcommand{\titePenteSN}{-\titePenteNS}
\newcommand{\titePenteSE}{-\titePenteNW}
\newcommand{\titePenteSW}{-\titePenteNE}
\newcommand{\rayon}{0.1}
\newcommand{\feuille}[3]{
\draw (\noeud{#2}{#1})--++(-#3,-1);
\fill (\absNoeud{#2}{#1}-#3,-#2-1) circle(\rayon);
}
\draw (0,0)--(\noeud{3.4}{\gdePenteN});
\draw (0,0)--(\noeud{3.4}{\gdePenteE});
\draw (0,0)--(\noeud{3.4}{\gdePenteW});
\draw (0,0)--(\noeud{3.4}{\gdePenteS});
\draw [dashed] (\noeud{3.4}{\gdePenteN})--(\noeud{4.1}{\gdePenteN});
\draw [dashed] (\noeud{3.4}{\gdePenteE})--(\noeud{4.3}{\gdePenteE});
\draw [dashed] (\noeud{3.4}{\gdePenteW})--(\noeud{4.3}{\gdePenteW});
\draw [dashed] (\noeud{3.4}{\gdePenteS})--(\noeud{4.1}{\gdePenteS});
\feuille{\gdePenteN}{1}{\titePenteNE}
\feuille{\gdePenteN}{1}{\titePenteNW}
\feuille{\gdePenteN}{1}{\titePenteNS}
\feuille{\gdePenteE}{1}{\titePenteEN}
\feuille{\gdePenteE}{1}{\titePenteEW}
\feuille{\gdePenteE}{1}{\titePenteES}
\feuille{\gdePenteW}{1}{\titePenteWN}
\feuille{\gdePenteW}{1}{\titePenteWE}
\feuille{\gdePenteW}{1}{\titePenteWS}
\feuille{\gdePenteS}{1}{\titePenteSN}
\feuille{\gdePenteS}{1}{\titePenteSE}
\feuille{\gdePenteS}{1}{\titePenteSW}
\feuille{\gdePenteN}{2}{\titePenteNE}
\feuille{\gdePenteN}{2}{\titePenteNW}
\feuille{\gdePenteN}{2}{\titePenteNS}
\feuille{\gdePenteE}{2}{\titePenteEN}
\feuille{\gdePenteE}{2}{\titePenteEW}
\feuille{\gdePenteE}{2}{\titePenteES}
\feuille{\gdePenteW}{2}{\titePenteWN}
\feuille{\gdePenteW}{2}{\titePenteWE}
\feuille{\gdePenteW}{2}{\titePenteWS}
\feuille{\gdePenteS}{2}{\titePenteSN}
\feuille{\gdePenteS}{2}{\titePenteSE}
\feuille{\gdePenteS}{2}{\titePenteSW}
\feuille{\gdePenteN}{3}{\titePenteNE}
\feuille{\gdePenteN}{3}{\titePenteNW}
\feuille{\gdePenteN}{3}{\titePenteNS}
\feuille{\gdePenteE}{3}{\titePenteEN}
\feuille{\gdePenteE}{3}{\titePenteEW}
\feuille{\gdePenteE}{3}{\titePenteES}
\feuille{\gdePenteW}{3}{\titePenteWN}
\feuille{\gdePenteW}{3}{\titePenteWE}
\feuille{\gdePenteW}{3}{\titePenteWS}
\feuille{\gdePenteS}{3}{\titePenteSN}
\feuille{\gdePenteS}{3}{\titePenteSE}
\feuille{\gdePenteS}{3}{\titePenteSW}
\end{tikzpicture}
\end{center}
\caption{
\label{fig:bcomb}
The $b$-comb for $b=4$}
\end{figure}
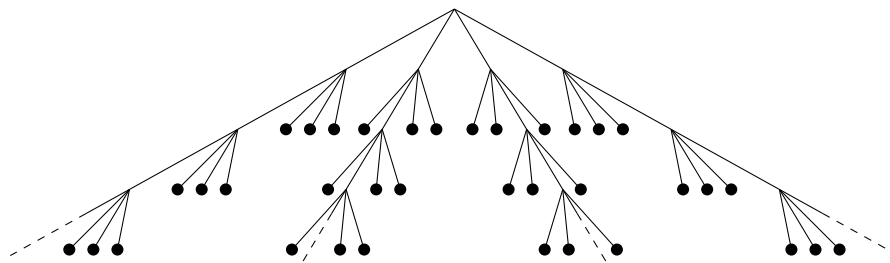

\begin{theo}
\label{semiMarkovIsVLMC}
Let $b$ be an integer, $b\geq 2$.
Every semi-Markov chain with true jumps on a state space having $b$ elements is
the process of initial letters of a VLMC on the $b$-comb.
\end{theo}
 
In the proof, placed in Section~\ref{sec:proofs} on page~\pageref{proof:semiMarkovIsVLMC}, the correspondance between the $b$-comb and the semi-Markov chain is made explicit. More precisely, the probability distributions at each context of the $b$-comb are given, such that the initial letter process of the VLMC has the same distribution as a given semi-Markov chain with $b$ states.

\begin{theo}
\label{th:CNSpourSM}
Let $\left( Z_n\right) _{n\geq 0}$ be a semi-Markov chain with true jumps on a finite state space $\rond E$.
Denote by $p=\left( p_{\alpha ,\beta}(k)\right) _{\alpha,\beta\in\rond E, k\geq 1}$ its semi-Markov kernel and assume that for any $\alpha,\beta\in\rond E, \alpha\not= \beta, k\geq 1$, $p_{\alpha ,\beta}(k)\not= 0$.
Then, the following properties are equivalent.
\begin{itemize}
\item[(i)]
$\left( Z_n\right) _{n\geq 0}$ admits a limit distribution.

\item[(ii)]
For every $\alpha\in\rond E$, the series
\[
m_{\alpha}:=\sum _{k\geq 1}k\left(\sum _{\gamma\in\rond E}p_{\alpha,\gamma}(k)\right)
\]
is convergent.

\end{itemize}
\end{theo}

A proof of Theorem~\ref{th:CNSpourSM} can be found in Section~\ref{sec:proofs}, page~\pageref{proof:limitSMC}.

\begin{rem}
\label{rem:equivBL}
The sum $m_{\alpha}$ is readily seen as a mean sojourn time: $m_{\alpha} = \Espe\left( T_1 \big| J_0 = \alpha\right)$. Theorem~\ref{th:CNSpourSM} establishes that  $m_{\alpha}<\infty$ for any $\alpha\in\rond E$ is a \emph{necessary and sufficient} condition for a semi-Markov chain with true jumps and with a positive semi-Markov kernel to admit a limit distribution. 
Thus, the sufficient assumption $m_{\alpha}<\infty$ for any  $\alpha\in\rond E$ in \cite{barbu/limnios/08} becomes a NSC when also assuming that for any $\alpha,\beta\in\rond E, \alpha\not= \beta, k\geq 1$, $p_{\alpha ,\beta}(k)\not= 0$.
\end{rem}

\subsection{From a VLMC to its induced SMC and back (finite number of alpha-LIS)}
\label{sec:summarized}

The above allows us to go a little further for a non-null stable VLMC $(U_n)$ and its associated semi-Markov chain $(Z_n)$ of its successive context alpha-LIS, in the case when there are finitely many alpha-LIS's. 
Remark \ref{rem:semiMarkov} asserts that  one cannot recover the VLMC $\left( U_n\right)$ from the semi-Markov chain $\left( Z_n\right)$ (see Section \ref{subsec:compareSM}).
Nevertheless,  one may ask whether the NSC for existence of a limit distribution for the semi-Markov chain $\left( Z_n\right)$ is the same as the NSC for existence and unicity of a stationary probability measure for the VLMC $\left( U_n\right)$.
The answer is yes.

Indeed, under the assumptions of Theorem~\ref{cor:finite} (finite number of alpha-LIS), the induced  $\rond S$-valued semi-Markov chain $\left( Z_n\right)$ has a finite number of states. Thus, Theorem \ref{th:CNSpourSM} applies and gives a NSC for $\left( Z'_n\right)$, the semi-Markov chain with true jumps deduced from $\left( Z_n\right)$ by formulas \eqref{jumpSM}.
This NSC writes $m'_{\alpha s}< +\infty$ where
\[
m'_{\alpha s} = \Espe\left( T_1' | J_0' = \alpha s \right) = \sum_{k\geq 1} k \left( \sum_{\beta t\not= \alpha s} p'_{\alpha s, \beta t}\right).
\]
Besides, thanks to \eqref{jumpSMtime}, $m'_{\alpha s}< +\infty$ is equivalent to $m_{\alpha s}< +\infty$ since, as already noticed in Remark~\ref{rem:equivBL}, $m_{\alpha s} =\Espe\left( T_1 | J_0 = \alpha s \right) $.
Thanks to Proposition \ref{pro:compareSM}(i) and its proof, $\Espe\left( T_1 | J_0 = \alpha s \right)  = \kappa_{\alpha s}$, so that
\[
\kappa_{\alpha s} = m_{\alpha s}.
\]
Moreover, in Theorem \ref{cor:finite}, $\kappa_{\alpha s}<+\infty$ for any $\alpha s\in\rond S$ is the NSC for existence and unicity of a stationary probability measure for a stable VLMC with a finite number of alpha-LIS. Summarizing, the following holds.
\begin{pro}
\label{pro:summarized}
Let $\left( U_n\right) _n$ be a non-null stable VLMC admitting a finite number of alpha-LIS. Let $\left( Z_n\right) _n$ be the $\rond S$-valued process of its alpha-LIS -- see Formula~\eqref{def:Zn}. Then, the following properties are equivalent.
\begin{itemize}
\item[(i)]
$\left( U_n\right) _{n\geq 0}$ admits a unique stationary probability measure.
\item[(ii)]
The cascade series~\eqref{convCasc} converge.
\item[(iii)]
$\left( Z_n\right) _{n\geq 0}$ admits a limit distribution.
\end{itemize}
\end{pro}
\section{Proofs}
\label{sec:proofs}

\subsection{Proofs of Section~\ref{sec:def} (general case)}

\begin{proof}[Proof of lemma \ref{lem:cascade} (Cascade formulae)]
\label{proof:cascade}

\ 

(i)
Assume first that $\pi\left(w\rond R\right)\neq 0$.
Then, since $w$ is noninternal, $\pref (w)$ is well defined so that, by stationarity,
\begin{align*}
\pi\left(\alpha w\rond R\right)
&=\Proba _\pi\left( U_1\in\alpha w\rond R\right)\\
&=\Proba _\pi\left( U_{1}\in\alpha w\rond R|U_0\in w\rond R\right)
\Proba _\pi\left( U_0\in w\rond R\right)\\
&
=q_{\pref\left( w\right)}(\alpha)\pi\left(w\rond R\right)
\end{align*}
proving~\eqref{cascade1}.
If $\pi\left(w\rond R\right) =0$, then, by stationarity, $\pi\left(\alpha w\rond R\right) =\Proba _\pi\left( U_1\in\alpha w\rond R\right)\leq\Proba _\pi\left( U_0\in w\rond R\right)=0$ so that~\eqref{cascade1} remains true.

(ii)
By an argument similar to the one above, if $\pi (r)=0$, then
\[
\pi (\alpha r)=\Proba _\pi\left( U_1=\alpha r\right)\leq\Proba _\pi\left( U_0=r\right) =\pi (r)=0.
\]
If, on the contrary, $\pi (r)\neq 0$, then
\[
\pi (\alpha r)
=\Proba _\pi\left( U_1=\alpha r\right)
=\Proba _\pi\left( U_{1}=\alpha r|U_0=r\right)\Proba _\pi\left( U_0=r\right)
=q_{\pref (r)}(\alpha )\pi (r),
\]
which proves (ii).

(iii)
Direct induction from Formula~\eqref{cascade1}.
\end{proof}
\begin{proof}[Proof of lemma \ref{lem:pineq0}]
\label{proof:pineq0}
(i) We prove that if $w$ is a finite word and if $\alpha\in\rond A$, then $\left[\pi\left(\alpha w\rond R\right) =0\right]\Rightarrow\left[\pi \left(w\rond R\right)=0\right]$.
An induction on the length of $w$ is then sufficient to prove the result since  $\pi (\rond R)=1$.
Remember that $\rond I$ stands for the set of finite internal words.
\begin{enumerate}
\item
Assume that $w\notin\rond I$ and that $\pi\left(\alpha w\rond R\right) =0$.
Then, as a consequence of the cascade formula~\eqref{cascade1}, $0=\pi\left(w\rond R\right)q_{\pref (w)}(\alpha )$.
As no $q_c$ vanishes, $\pi\left(w\rond R\right) =0$.
\item
Assume now $w\in\rond I$ and $\pi\left(\alpha w\rond R\right) =0$.
Then, by disjoint union and Lemma~\ref{lem:cascade}, since $\pi$ is stationary,
\[
0=\pi\left(\alpha w\rond R\right)
=\sum _{\substack{c\in\rond C^f\\[1pt]c=w\cdots}}\pi\left(\alpha c\rond R\right)
+\sum _{\substack{c\in\rond C^i\\[1pt]c=w\cdots}}\pi\left(\alpha c\right)
=\sum _{\substack{c\in\rond C^f\\[1pt]c=w\cdots}}\pi\left(c\rond R\right)q_c(\alpha )
+\sum _{\substack{c\in\rond C^i\\[1pt]c=w\cdots}}\pi\left(c\right)q_c(\alpha ).
\]
As no $q_c$ vanishes, all the $\pi\left(c\rond R\right)$ and the $\pi (c)$ necessarily vanish so that, by disjoint union,
\[
\pi\left(w\rond R\right)
=\sum _{\substack{c\in\rond C^f\\[1pt]c=w\cdots}}\pi\left(c\rond R\right)
+\sum _{\substack{c\in\rond C^i\\[1pt]c=w\cdots}}\pi\left(c\right)
=0.
\]
\end{enumerate}

(ii) Denote $r=\alpha _1\alpha _2\cdots$ and $r_n=\alpha _n\alpha _{n+1}\cdots$ its $n$-th suffix, for every $n\geq 1$.
Since $\pi$ is stationary, an elementary induction from Formula~\eqref{cascade2} implies that, for every $m\geq 1$,
\begin{equation}
\label{pir}
\pi\left(r\right)
=\left(\prod _{k=1}^mq_{\pref (r_{k+1})}\left(\alpha _k\right)\right)\pi\left(r_{m+1}\right).
\end{equation}

\begin{enumerate}
\item
Assume first that $r=st^\infty$ is ultimately periodic, where $s$ and $t$ are finite words, $t=\beta _1\cdots\beta _T$ being nonempty.
Then, because of~\eqref{pir}, $\pi (r)\leq\pi\left(t^\infty\right)$ and $\pi\left(t^\infty\right)=\rho\pi\left(t^\infty\right)$ where
\[
\rho=
\prod _{k=1}^T
q_{\pref (\beta _{k+1}\cdots\beta _Tt^\infty)}\left(\beta _{k}\right).
\]
In this product, the term obtained for $k=T$ writes $q_{\pref\left( t^\infty\right)}\left(\beta _T\right)$.
Since the probability measures $q_c$ are all assumed never to vanish, they cannot take $1$ as a value so that $\rho <1$, which implies that $\pi\left(t^\infty\right)=0$.
Note that this argument proves that an invariant probability measure $\pi$ vanishes on ultimately periodic infinite words as soon as the $q_c$ never take $1$ as a value
(this assumption is weaker than non-nullness).

\item
Assume on the contrary that $r$ is aperiodic.
Then, $m\neq n\Longrightarrow r_n\neq r_m$ for all $n,m\geq 1$:
the $r_n$ are all distinct among the infinite branches of the context tree.
Thus, by disjoint union, 
\[
\sum _{n\geq 1}\pi\left(r_n\right)
\leq\sum _{c\in\rond C^i}\pi\left(c\right)
\leq \pi\left(\rond R\right)=1,
\]
which implies in particular that $\pi\left(r_n\right)$ tends to $0$ when $n$ tends to infinity.
Since $\pi\left(r\right)\leq\pi\left(r_n\right)$ because of Formula~\eqref{pir}, this leads directly to the result.
\end{enumerate}
\end{proof}
\begin{proof}[Proof of theorem~\ref{fQbij}.]
\label{proof:fQbij}
The proof is given for the alphabet $\rond A=\set{0,1}$.
It can be straightforwardly adapted to the case of an arbitrary finite alphabet.

Proof of {\it (i)}.
If $\pi$ is a stationary probability measure, 
disjoint union, Lemma~\ref{lem:pineq0}(ii) and the Cascade Formula~\eqref{cascade3} imply that
\[
1=\sum _{c\in\rond C^f}\pi\left(c\rond R\right)
=\sum _{c\in\rond C^f}\casc (c)\pi\left(\alpha _cs_c\rond R\right).
\]
Gathering together all the contexts that have the same alpha-LIS leads to
\[
1=\sum _{\alpha s\in\rond S}\pi\left(\alpha s\rond R\right)\left(\sum _{c\in\rond C^f,~c=\cdots [\alpha s]}\casc (c)\right).
\]
Now, by Lemma~\ref{lem:pineq0}(i), $\pi\left(\alpha s\rond R\right)\neq 0$ for all $\alpha s\in\rond S$. This forces the sums of cascades to be finite.

\vskip 5pt
Proof of {\it (ii)}.

1) \emph{Injectivity}.
Let $\pi$ be a stationary probability measure on $\rond R$.
As the cylinders based on finite words generate the whole $\sigma$-algebra, $\pi$ is determined by the $\pi\left(w\rond R\right)$, $w\in\rond W$.
Now write any $w\in\rond W\setminus\{\emptyset\}$ as $w=p\alpha s$ where $\alpha s$ is the alpha-LIS of $w$ and $p\in\rond W$
(beware, $\alpha s$ may not be the alpha-LIS of a context).
As $\pi$ is stationary, the cascade formula~\eqref{cascade3} entails $\pi\left(w\rond R\right) =\casc (w)\pi\left(\alpha s\rond R\right)$.
As a consequence, $\pi$ is determined by its values on the words $\alpha s$ where $ s\in\rond I$ is internal and $\alpha\in\rond A$.
Now, as $ s\in\rond I$, by disjoint union, cascade formula~\eqref{cascade1} and Lemma~\ref{lem:pineq0}(ii),
\[
\pi\left(\alpha s\rond R\right)
=\sum _{c\in\rond C^f,~c= s\cdots}\pi\left( c\rond R\right) q_c\left(\alpha\right).
\]
This means that $\pi$ is in fact determined by the $\pi\left(c\rond R\right)$ where $c$ is a finite context.
Lastly, as above, the stationarity of $\pi$, the cascade formula~\eqref{cascade3} and the decomposition of any context $c$ into $c=p_c\alpha _cs_c$ where $\alpha _cs_c$ is the alpha-LIS of $c$ together imply that $\pi$ is determined by the $\pi\left(\alpha s\rond R\right)$ where $s\in\rond S$
(remember, $\rond S$ denotes the set of all alpha-LIS \emph{of contexts}).
This proves that the restriction of $\fff$ to stationary measures is one-to-one.

\vskip 5pt
2) \emph{Image of a stationary probability measure}.
Let $\pi\in\rond M_1\left(\rond R\right)$ be stationary.
By disjoint union, as above, if $\alpha s\in\rond S$,
\[
\pi\left(\alpha s\rond R\right)
=\sum _{c\in\rond C^f,~c=s\cdots}\pi\left(c\rond R\right) q_c\left(\alpha\right) .
\]
Applying the Cascade Formula~\eqref{cascade3} to all contexts in the sum and noting that $\casc (\alpha c)=q_c(\alpha )\casc (c)$, one gets
\[
\pi\left(\alpha s\rond R\right)
=\sum _{c\in\rond C^f,~c=s\cdots}\casc (\alpha c)\pi\left(\alpha _cs_c\rond R\right).
\]
Gathering together all the contexts that have the same alpha-LIS entails
\[
\pi\left(\alpha s\rond R\right)
=\sum _{\beta t\in\rond S}\pi\left(\beta t\rond R\right)\left( \sum _{\substack{{c\in\rond C^f}\\[1pt]{c=s\cdots =\cdots [\beta t]}}}\casc (\alpha c)\right)
=\sum _{\beta t\in\rond S}\pi\left(\beta t\rond R\right)Q_{\beta t,\alpha s}.
\]
This means that the row vector $\left(\pi\left(\alpha s\rond R\right)\right) _{\alpha s\in\rond S}$ is a left-fixed vector for the matrix $Q$.
We have shown that $\fff$ sends a stationary probability measure to a left-fixed vector for~$Q$ with positive entries.
Moreover, as in the proof of (i), Equality~\eqref{masseTotale} holds.

\vskip 5pt
3) \emph{Surjectivity}.

Let $\left(v_{\alpha s}\right) _{\alpha s\in\rond S}\in [0,+\infty [^{\rond S}$ be a row vector, left-fixed by~$Q$, that satisfies $\sum _{\alpha s\in\rond S}v_{\alpha s}\kappa _{\alpha s}=1$.
Let $\mu$ be the function defined on $\rond S$ by $\mu\left( \alpha s\right)=v_{\alpha s}$.
Denoting by $\alpha _ws_w$ the alpha-LIS of any finite non-empty word $w$, the function~$\mu$ extends to any finite non-empty word in the following way:
\begin{equation}
\label{prolongementV}
\forall w\in\rond W\setminus\{\emptyset\},~\mu (w)
=\casc (w)\sum _{\substack{{c\in\rond C^f}\\[1pt]{c=s_w\cdots}}}\casc (\alpha_wc)\mu\left(\alpha _cs_c\right)\in [0,+\infty ].
\end{equation}
Notice that this definition actually extends $\mu$ because of the fixed vector property, and that, at this moment of the proof, $\mu (w)$ might be infinite.
Notice also that this implies $\mu(w)=\casc(w)\mu(\alpha_ws_w)$ for any $w\in\rond W$, $w\neq\emptyset$.

For every $n\geq 1$ and for all $w\in\rond W$ such that $|w|=n$, define $\pi_n\left(w\right) =\mu\left( w\right)$.
This clearly defines a $[0,+\infty ]$-valued measure $\pi_n$ on $\rond A_{n}=\prod_{1\le k\le n}\rond A$.
Besides, $\pi _1$ is a probability measure.
Indeed, because of Definition~\eqref{prolongementV} and Remark \eqref{rem:cascRem},
\[
\mu (0)+\mu (1)
=\sum _{c\in\rond C^f}\left(\casc\left( 0c\right) +\casc\left( 1c\right)\right)\mu\left(\alpha _cs_c\right)
=\sum _{c\in\rond C^f}\casc\left( c\right)\mu\left(\alpha _cs_c\right)
\]
which can be written
\[
\mu (0)+\mu (1)
=\sum _{\alpha s\in\rond S}\mu\left(\alpha s\right)\sum _{\substack{{c\in\rond C^f}\\[1pt]{c=\cdots [\alpha s]}}}\casc (c)
=\sum _{\alpha s\in\rond S}v_{\alpha s}\kappa _{\alpha s}=1,
\]
the last equality coming from the assumption on $(v_{\alpha s})_{\alpha s}$.

In view of applying Kolmogorov extension theorem, the consistency condition states as follows:
$\pi_{n+1}(w\rond A)=\pi_n(w)$ for any $w\in\rond W$ of length $n$.
This is true because
\begin{equation}
\label{eq:compatibility}
\mu(w0)+\mu(w1)=\mu(w).
\end{equation}
Indeed, for any $a\in\rond A$, since $s_w$ is internal, $s_wa$ is either internal or a context.
Furthermore,
\begin{itemize}
	\item if $s_wa\in\rond I$ then $s_{wa}=s_wa$, $\alpha_{wa}=\alpha_{w}$ and $\casc (wa)=\casc (w)$ so that 
\begin{equation}
\label{eq:pourCompatibility}
\mu (wa)=\casc(w)\sum_{\substack{{c\in\rond C^f}\\[1pt]{c=s_wa\cdots}}}\casc\left(\alpha_wc\right)\mu\left(\alpha_c s_c\right) ;
\end{equation}
	\item if $s_wa\in\rond C$ then denote $\kappa =s_wa$ so that $\casc (wa)=\casc (w)\casc\left(\alpha _w\kappa\right)$, $\alpha _{wa}=\alpha _\kappa$ and $s_{wa}=s_\kappa$.
	Thus, $\mu (wa)
	=\casc (wa)\sum _{c=s_\kappa\cdots}\casc\left(\alpha _\kappa c\right)\mu \left(\alpha _cs_c\right)
	=\casc (w)\casc\left(\alpha _w\kappa\right)\mu\left(\alpha _\kappa s_\kappa\right)$, which implies that~\eqref{eq:pourCompatibility} still holds, the sum being reduced to one single term since $s_wa$ is itself a context.
	\end{itemize}
Valid in all cases, Formula~\eqref{eq:pourCompatibility} easily implies Claim~\eqref{eq:compatibility}. Consequently all the $\pi_n$ are probability measures.
By Kolmogorov extension theorem, there exists a unique probability measure $\pi$ on $\rond R$ such that $\pi_{|\rond A_{n}}=\pi_n$ for every $n$.
Note that this result implies that $v_{\alpha s}=\pi\left(\alpha s\rond R\right)\leq 1$, for every $\alpha s\in\rond S$.

Furthermore, $\pi (c)=0$ for any infinite context $c$.
Indeed, one has successively,
\begin{align*}
1&=\sum _{c\in\rond C^f}\pi\left(c\rond R\right)+\sum _{c\in\rond C^i}\pi\left(c\right)\\
&=\sum _{c\in\rond C^f}\mu (c)+\sum _{c\in\rond C^i}\pi\left(c\right)
\end{align*}
and, besides,
\[
\sum _{c\in\rond C^f}\mu (c)
=\sum _{c\in\rond C^f}\casc (c)\mu\left(\alpha _cs_c\right)
=\sum _{\alpha s\in\rond S}\mu\left(\alpha s\right)\sum _{\substack{{c\in\rond C^f}\\[1pt]{c=\cdots [\alpha s]}}}\casc (c)
=\sum _{\alpha s\in\rond S}v_{\alpha s}\kappa _{\alpha s}=1
\]
so that $\sum _{c\in\rond C^i}\pi\left(c\right)=0$.

Finally, the stationarity of $\pi$ follows from  the identity $\mu(0w)+\mu(1w)=\mu(w)$ for any finite word $w$.
Namely:
\begin{itemize}
\item if $w\notin\rond I$, then for $a\in\rond A$, $s_{aw}=s_w$, $\alpha_{aw}=\alpha_w$ hence
\[
\mu(0w)+\mu(1w)
=\left(\casc (0w)+\casc (1w)\right)\sum_{\substack{{c\in\rond C^f}\\[1pt]{c=s_w\cdots}}}\casc\left(\alpha_wc\right)\mu\left(\alpha_cs_c\right) .
\]
Now, Remark~\ref{rem:cascRem} entails the claim.
\item if $w\in\rond I$, then for $a\in\rond A$, $s_{aw}=w$, $\alpha_{aw}=a$ and $\casc(aw)=1$ thus
\[
\mu(aw)=\sum_{\substack{{c\in\rond C^f}\\[1pt]{c=w\cdots}}}\casc (ac)\mu\left(\alpha_cs_c\right) .
\]
Using again Remark~\ref{rem:cascRem}, it comes
\begin{align*}
\mu(0w)+\mu(1w)&=\sum_{c\in\rond C^f,~c=w\cdots}\left(\casc(0c)+\casc(1c)\right)\mu(\alpha_cs_c)\\
&=\sum_{c\in\rond C^f,~c=w\cdots}\casc(c)\mu(\alpha_cs_c)\\
&=\sum_{c\in\rond C^f,~c=w\cdots}\mu(c)\\
&=\sum_{c\in\rond C^f,~c=w\cdots}\pi\left(c\rond R\right)\\
&=\pi\left(w\rond R\right)-\sum_{\substack{{c\in\rond C^i}\\[1pt]{c=w\cdots}}}\pi\left(c\right) =\mu (w),
\end{align*}
\end{itemize}
the last equality being valid because $\pi$ vanishes on infinite contexts.
Since $f(\pi )=\left( v_{\alpha s}\right) _{\alpha s}$, this concludes the proof.
\Ffin
\end{proof}

\subsection{Proofs of Section~\ref{sec:stable} (stable case)}
\label{sec:proofsStable}

\begin{proof}[Proof of Proposition \ref{pro:defstable}]
\label{proof:defstable}

$(i)\implies (ii)$.
Take $c\in\rond C$ and $\alpha\in\rond A$.
Assume that $\alpha c\in\rond I$.
So, if $\beta\in\rond A$ is any letter, then $\alpha c\beta\in\rond T$.
The item $(i)$ implies therefore that $c\beta\in\rond T$, which contradicts $c\in\rond C$.

$(ii)\implies (i)$. Take $\alpha\in\rond A$ and $w\in\rond W$ such that $\alpha w\in\rond T$. If  $w\notin\rond T$ then there exists a finite context $c$ such that $w=cw'$ with $w'\neq\emptyset$. It comes $\alpha cw'\in\rond T$, which implies $\alpha c\in\rond I$ and this contradicts $(ii)$.

$(i)\iff(iii)$ is straightforward.

$(ii)\implies (iv)$
What needs to be proved is that $\pref\left(U_{n+1}\right)$ only depends on $U_n$ through $\pref\left(U_n\right)$.
In other words, we shall prove that for all $s\in\rond R, \alpha \in \rond A$, $\pref\left(\alpha s\right)$ only depends on $s$ through $\pref(s)$.
This is clear because
\[
\pref(\alpha s) = \pref\left(\alpha\pref(s) \right).
\]
Indeed, $\pref(s)\in\rond C$ and if $\pref(s)\in\rond C^f$, then (ii) implies $\alpha\pref(s)\notin\rond I$ (if $\pref(s)$ is an infinite context, this means that $s=\pref(s)$ and the above equality is straightforward). Therefore $\alpha\pref(s)$ writes $cw$ with $c\in\rond C$ and $w\in\rond W$. On one hand, this entails $\pref\left(\alpha\pref(s) \right)=c$. On the other hand, this means that $cw$ is a prefix of $\alpha s$ thus $\pref(\alpha s)=c$.

$(iv)\implies (ii)$
We shall prove the contrapositive.
Assume there exists $c\in\rond C^f$ and $\alpha\in\rond A$ such that $\alpha c\in\rond I$.
%
Let $s\in\rond R$ such that $\pref(s)=c$. As $\alpha c\in\rond I$, $\pref(\alpha s)$ is a context which has $\alpha c$ as a strict prefix.
Therefore, $\pref(\alpha s)$ does not only depend on $\pref(s)$,
but on a prefix of $s$ strictly longer than $\pref(s)$.
Thus, conditionally to $C_n$, the transition to the context $C_{n+1}$ does not depend on $C_n$ but on a strictly longer prefix of $U_n$.
This proves that $\left(C_n\right) _n$ is not Markovian.
\end{proof}
\begin{proof}[Proof of Lemma \ref{lem:fondLem}]
\label{proof:fondLem}
1. Assume that $t$ is a context LIS such that $c=t\cdots$, then, by definition of a context LIS, there exists $\alpha\in\rond A$ such that $\alpha t$ is a context alpha-LIS. Since $\rond T$ is stable, Lemma~\ref{lem:contextStableTree} implies that $\alpha t$ is a context, therefore $\alpha c\notin\rond C$ because two different contexts cannot be prefix of one another. Thus $\rond A_c\neq\emptyset$.
	
2. Let $\alpha$ be in $\rond A_c$ so that $\alpha c$ is not a context. As $\rond T$ is stable, $\alpha c$ is not internal, thus it is an external node.
Let $c'=\pref (\alpha c)$.
The context $c'$ is a strict prefix of $\alpha c$.
Since $\rond T$ is stable, $\sigma (c')$ is non-external.
But $\sigma (c')$ cannot be a context because it is a prefix of $c$ which is a context.
Thus $\sigma (c')$ is internal.
This implies that $t_{\alpha}:=\sigma (c')$ is the LIS of $c'$ and a prefix of $c$ as well, and that $\alpha t_{\alpha}=c'\in\rond C$.
Besides, whenever (i) and (ii) are satisfied,  $t_\alpha$ writes necessarily $t_\alpha =\sigma\left(\pref\left( \alpha c\right)\right)$.
Thus existence and unicity of $t_\alpha$ are proven.
Finally, since $c'=\alpha t_{\alpha}=\pref (\alpha c)$, $t_{\alpha}$ is a prefix of $c$, so that for every $\beta\notin\rond A_c$, $\beta t_{\alpha}$ is a prefix of the context $\beta c$.
Consequently, $\beta t_{\alpha}\notin\rond C$ because two different contexts cannot be prefix of one another.
\end{proof}
\begin{proof}[Proof of Lemma \ref{lem:context_markov}] 
\label{proof:context_markov}
\emph{Irreducibility.}
Let $c$ and $c'$ be finite contexts.
Denote $c'=\alpha_k\dots\alpha_2\alpha_1$, $k\geq 1$.
In order to prove that the Markov chain $(C_n)$ has a non null transition from $c$ to $c'$, let us add the successive letters of $c'$ (starting from $\alpha _1$) to the left of $c$ and prove that at each time, the transition is possible and non-null.
Assume that $C_0=c$ and consider the word $\alpha_1c$.
As $\rond T$ is shift-stable $\alpha_1c\notin\rond I$ and $c_1:=\pref(\alpha_1c)$ is a (possibly not strict) prefix of $\alpha_1c$ which may be written $c_{1}=\alpha_1w_1$ with $w_1$ a prefix (possibly empty) of $c$.
The transition equals $q_c(\alpha_1)$ and is therefore non-null.
Let us add the second letter and consider the word $\alpha_{2}\alpha_1w_1=\alpha_{2}c_1$.
Again $\alpha_{2}\alpha_1w_1\notin\rond I$ and $c_2:=\pref\left( \alpha_{2}\alpha_1w_1\right)$ is a prefix of $\alpha_{2}\alpha_1w_1$.
Here, the point is that $c_2$ cannot be $\alpha_{2}$ otherwise we would have $\sigma^{k-2}(c')=\alpha_{2}\alpha_1\notin\rond T$ which would contradict the shift-stability of~$\rond T$.
Therefore $c_{2}=\alpha_{2}\alpha_1w_2$ with $w_2$ a prefix of $w_1$.
By adding successively the letters of $c'$, with the same arguments, one gets a sequence of contexts $c_j=\pref\left(\alpha _jc_{j-1}\right) =\alpha _j\alpha _{j-1}\dots \alpha _1\dots$.
The last step necessarily writes $c_k=\alpha _k\dots\alpha _1=c'$ because $c'$ is a context.
At each step, the transition is non zero because all the $q_c(\alpha)$ are non-null.

\emph{Aperiodicity.}
Let us prove that, given $c\in\rond C^f$, the g.c.d.\@ of the lengths of the admissible paths from $c$ to itself equals $1$.
Let $k$ be the length of $c$.
The above proof of the irreducibility shows that there exists an admissible path of length $k$ joining $c$ to itself.
Using similar arguments, let $\alpha$ be any letter and let $c'$ be the context $c'=\pref\left( \alpha c\right)$.
There is an admissible path of length $1$ joining $c$ to $c'$.
Using again the proof of the irreducibility, there is also an admissible path of length $k$ joining $c'$ to $c$.
Combining these two paths provides a path of length $k+1$ that joins $c$ to itself.
Since $gcd(k,k+1)=1$, the chain is aperiodic.
\end{proof}

\begin{proof}[Proof of Proposition~\ref{pro:compareSM}]
\label{proof:compareSM}

The initial context $C_0=\pref\left( U_0\right)$ is assumed to be finite.
Since the context tree is stable, this implies that all $C_n=\pref\left( U_n\right)$ are also almost surely finite words -- see Proposition~\ref{pro:defstable}(ii).
Thanks to this fact, the definition of $\left( Z_n\right) _n$ makes sense (see~\eqref{def:Zn}).

\vskip 5pt
Let us prove (i), i.e.  that ${T}_n$ is almost surely finite (and $S_n$ as well), by induction on $n\geq 1$. Remember that $S_0 = 0$. 
To lighten the computation, assume that $C_0$ is a context alpha-LIS. 
If not, $C_0$ writes $\beta_1\dots\beta_p\alpha s$ and a term $\casc(C_0)$ has to be added to the successive equalities without modifying the argumentation.
\[
\Proba(T_1=+\infty) = \sum_{\alpha s\in\rond S}\Proba(T_1=+\infty|C_{0}=\alpha s)\Proba(C_{0}=\alpha s).
\]
It is sufficient to prove that, for all $\alpha s\in\rond S$, $\Proba(T_1=+\infty|C_{0}=\alpha s)=0$.
Now 
\[\Proba(T_1=+\infty|C_{0}=\alpha s)=\lim_{k\to\infty}\Proba(T_1\geq k|C_{0}=\alpha s).
\]
With the description of the process in Section \ref{subsec:compareSM}, see also Figure \ref{fig:SM},
\begin{eqnarray*}
\Proba\left( T_{1}\geq k|C_{S_{0}}=\alpha s\right)&=&\sum_{\substack{ \beta_1, \dots , \beta_{k-1}\in\rond A \\[1pt]\forall i \leq k-1,~\beta_i\dots \beta_1\alpha s\in \rond C }} q_{\alpha s}\left( \beta_1\right)q_{\beta_1\alpha s}\left( \beta_2\right) \dots q_{\beta_{k-2}\dots \beta_1\alpha s}\left( \beta_{k-1}\right)\\
&=&\sum_{\substack{ \beta_1, \dots ,\beta_{k-1}\in\rond A \\[1pt]\forall i \leq k-1,~\beta_i\dots \beta_1\alpha s\in \rond C }}\casc\left( \beta_{k-1}\dots  \beta_1\alpha s \right)\\
&=&\sum_{\substack{c\in\rond C,~c=\cdots [\alpha s]\\[1pt]|c|=|\alpha s| +k-1}}\casc c \\
&=& \kappa_{\alpha s}(k),
\end{eqnarray*}
which, by assumption, tends to $0$ when $k$ tends to infinity.
Consequently, $T_1$ and $S_1 = T_1$ is a.s. finite.

\vskip 5pt
Now, assume that, for all $n\geq 1$, $S_{n-1}$ is a.s. finite. Repeat the above argument, replacing $S_0$ by $S_{n-1}$ and $T_1$ by $T_n$. It appears that for  all $n\geq 1$,
\[
\Proba\left( T_{1}\geq k|C_{S_{0}}=\alpha s\right) = \Proba\left( T_{n}\geq k|C_{S_{n-1}}=\alpha s\right) =  \kappa_{\alpha s}(k),
\]
so that $T_n$ is  a.s. finite and $S_n$ as well.
Note that this proves in passing that the $T_n$ are almost surely finite if and only if all the $\kappa_{\alpha s}(k)$ tend to $0$ when $k$ tends to infinity, which has been evoked in the description of the process $\left( Z_n\right) _n$, a few lines before Proposition~\ref{pro:compareSM}'s statement.

Remembering the description of the process at the beginning of Section~\ref{subsec:compareSM}, (ii) is straightforward.
Moreover, $C_{S_{n-1}} = J_{n-1}$ so that summing on $k$ gives $\Espe\left( T_{n} \big| J_{n-1} = \alpha s\right) = \kappa_{\alpha s}$.
This makes the proof of (i) complete.

\vskip 5pt
For (iii), based on the description of the process and the finiteness of the $T_n$ and $S_n$, it is clear that the distribution of $(J_{n+1},T_{n+1})$ conditioned on the past only depends on $\left( J_n,T_n\right)$, so that $\left( J_n,T_n\right) _{n\geq 0}$ is a Markov process.  
For $j\geq 0$, $k\geq 2$, $\alpha s\in\rond S$, $\beta t\in\rond S$,
\[
\Proba \left( J_{n+1} = \beta t,T_{n+1} = k \big| J_n = \alpha s, T_n = j \right)
= \Proba \left(
Z_{S_n +k} = \beta t, Z_{S_n +k-1}=\dots =Z_{S_n +1}=\alpha s
\big|
Z_{S_n} = \alpha s
\right).
\]
Notice that this expression does not depend on $j$, ensuring that $\left( J_n,T_n\right) _{n\geq 0}$ is a Markov renewal chain.
Continuing the computation leads to
\[
\begin{array}{l}
\Proba \left( {J}_{n+1} = \beta t,{T}_{n+1} = k \big| {J}_n = \alpha s, {T}_n = j \right)
=
\Proba \left( {J}_{n+1} = \beta t,{T}_{n+1} = k \big| {J}_n = \alpha s \right)  \\[10pt]
= 
\displaystyle\sum_{\beta_1, \dots, \beta_{k-1}}\Proba \left( U_{S_n +k} = \beta t\cdots, U_{S_n +k-1} =\beta_{k-1}\cdots\beta_1 \alpha s\cdots, \dots , U_{S_n +1} =  \beta_1\alpha s\cdots \big| U_{S_n} = \alpha s\cdots \right) \\[10pt]
=  \displaystyle\sum_{\beta_1, \dots, \beta_{k-1}}  q_{\beta_{k-1}\cdots\beta_1\alpha s}(\beta) \dots q_{\beta_1\alpha s}(\beta_2)q_{\alpha s}(\beta_1),
\end{array}
\]
where the sum concerns the letters $\beta _1,\dots ,\beta _{k-1}$ such that $c=\beta_{k-1}\cdots\beta_1\alpha s$ is a context that begins with the LIS~$t$, and $\beta c$ is not a context.
This can be shortly written under the form
\begin{equation}
\label{kernel}
\Proba \left( {J}_{n+1} = \beta t,{T}_{n+1} = k \big| {J}_n = \alpha s, {T}_n = j \right) 
=  \displaystyle\sum_{\substack{c\in\rond C,~c=t\dots\\[1pt]c= \dots [\alpha s]\\[1pt]|c| = |\alpha s| + k-1}} \casc (\beta c).
\end{equation}
For $k=1$, the calculation reduces to the following: let $j\geq 0$, $\alpha s\in\rond S$, $\beta t\in\rond S$ and assume that $\alpha s$ begins with the LIS~$t$. Then
\[
\Proba \left( {J}_{n+1} = \beta t,{T}_{n+1} = 1 \big| {J}_n = \alpha s, {T}_n = j \right)
=\Proba \left( Z_{S_n +1} = \beta t \big| Z_{S_n} = \alpha s \right)
=q_{\alpha s}(\beta)
\]
which equals $\casc (\beta c)$ for $c = \alpha s$.

This computations prove that the semi-Markov renewal kernel of the Markov renewal chain $\left( J_n,T_n\right) _{n\geq 0}$ is indeed given by \eqref{kernel}.
Moreover, summing on $k$ in~\eqref{kernel} gives
\[
\Proba \left( {J}_{n+1} = \beta t \big| {J}_n = \alpha s \right)  =  \displaystyle\sum_{\substack{c\in\rond C,~c=t\dots\\[1pt]c= \dots [\alpha s]}} \casc \left(\beta c\right) = Q_{\alpha s, \beta t}.
\]
The latter provides that $\left( J_n\right) _n$ is a Markov process, with transition matrix $Q$ and therefore that $Q$ is stochastic.
It gives a proof of Proposition \ref{pro:stochasticity}.
\end{proof}
\begin{proof}[Proof of Proposition \ref{pro:irreducibility}]
\label{proof:irreducibility}
Let $U$ be the VLMC defined by $\left(\rond T,q\right)$.
As before in the text, for every $n$, let $C_n=\pref\left( U_n\right)$ and $Z_n$ be the alpha-LIS of $C_n$.
Let also $J$ be the internal chain of the semi-Markov process $Z=\left( Z_n\right) _n$ -- see Definition~\ref{def:semiMarkov} and Proposition~\ref{pro:compareSM}.
Since the context tree is assumed to be stable, any context alpha-LIS is a finite context (Lemma~\ref{lem:contextStableTree}).
Besides, the Markov chain $C=\left( C_n\right) _n$ induced on $\rond C^f$ has been shown to be irreducible (Proposition~\ref{lem:context_markov}).
Therefore, two arbitrary finite contexts are joined by an admissible path relative to the Markov chain $C$.
In particular, two arbitrary context alpha-LIS $\alpha s$ and $\beta t$ are joined by an admissible path relative to the Markov chain $C$.
Taking the alpha-LIS of such a path of contexts provide a path of context alpha-LIS joining $\alpha s$ to $\beta t$ for the process $Z$, which means that conditioning by $Z_0=\alpha s$, there is some $n\geq 0$ such that the event $Z_n=\beta t$ occurs with positive probability.
Restricting this path to jump times provides an admissible path joining $\alpha s$ to $\beta t$ relative to the Markov process $J$.
\end{proof}

\begin{proof}[Proof of Theorem \ref{th:stable} (Invariant probability measures for a stable VLMC)]
\label{proof:stable}

\ 

({\it 3.$\implies$1.}) Since $Q$ is recurrent and irreducible, there exists a unique line $\g Rv$ of left-fixed vectors for~$Q$.
Let $v=\left( v_{\alpha s}\right) _{\alpha s\in\rond S}$ be such a vector having non-negative entries
(see for example \cite[Theorem 5.4]{seneta/06}).
Theorem~\ref{fQbij}(ii) coupled with the assumption on the series $\sum _{\alpha s\in\rond S} v_{\alpha s}\kappa_{\alpha s}$ entails directly the existence and uniqueness of a stationary probability measure.

({\it 2.$\implies$3.})
If there exists a stationary probability measure, then Theorem~\ref{fQbij}(i) and Lemma~\ref{lem:pineq0} assert that the cascade series converge and that $Q$ admits at least one left-fixed vector $v$ with positive entries such that $\sum _{\alpha s\in\rond S} v_{\alpha s}\kappa_{\alpha s}<\infty$.
Besides, every $\kappa _{\alpha s}$ is greater than $1$.
Indeed, the cascade of any alpha-LIS is $1$ and, in the stable case, any alpha-LIS is a context (Lemma~\ref{lem:contextStableTree}).
Thus, $v$ is summable and $Q$ is positive recurrent
(see for instance~\cite[Corollary of Theorem 5.5]{seneta/06}).
Since it is irreducible (Proposition~\ref{pro:irreducibility}), it admits a unique direction of left-fixed vectors $\g Rv$, proving ({\it 3.}) by Theorem~\ref{fQbij}.
\end{proof}
\begin{proof}[Proof of Theorem \ref{cor:finite} (finite number of alpha-LIS)]
\label{proof:finite}

\

Thanks to Theorem~\ref{th:stable}, (i) and (ii) are equivalent.
Moreover, (i)$\implies$(iii) is contained in Theorem~\ref{fQbij}(i).
Assume reciprocally that the cascade series converge.
Since $Q$ is stochastic, irreducible and finite dimensional, it admits a unique direction of left-fixed vectors, so that Theorem~\ref{fQbij}(ii) allows us to conclude.
	
\emph{Convergence towards $\pi$.}
Denote by $\pi _{\rond C^f}$ the measure on $\rond C^f$ induced by $\pi$, defined by $\pi _{\rond C^f}(c)=\pi\left(c\rond R\right)$.
For any $n\geq 0$, denote also by $C_n$ the \emph{cont} of $U_n$.
Thanks to Lemma~\ref{lem:pineq0}(ii), $\pi _{\rond C^f}$ is a probability measure on $\rond C^f$.
Thus, since the process induced by $\left(C_n\right) _n$ on $\rond C^f$ is an irreducible aperiodic Markov chain (Lemma~\ref{lem:context_markov}) on a denumerable state space that admits $\pi _{\rond C^f}$ as an invariant probability measure, the distribution of $C_n$ converges to $\pi _{\rond C^f}$ as soon as $C_0$ is finite (and $\pi _{\rond C^f}$ is the unique invariant probability measure on $\rond C^f$ for the Markov chain induced by $\left( U_n\right) _n$ on $\rond C^f$).
In particular, if $\nu$ is a distribution of $U_0$ that does not charge infinite contexts, for any context $c\in\rond C^f$,
\begin{equation}
\label{convContext}
	\Proba_{\nu}{\left( U_n\in c\rond R\right)}
\tooSous{n\to\infty}
\pi\left( c\rond R\right) .
\end{equation}

Now, for any finite word $w$, use again the notation $w=\beta _1\cdots\beta _{p_w}\alpha _ws_w$ where the $\beta _k$  and $\alpha _w$ are letters, $p_w$ is a nonnegative integer and $s_w$ the LIS of $w$.
Take $w\in\rond W$.

\cercle1
Assume first that $w$ is noninternal and that $n\geq p_w$.
Then,
\[
	\Proba_{\nu}{\left( U_n\in w\rond R\right)}
	=\casc (w)\Proba_{\nu}{\left( U_{n-p_w}\in\alpha _ws_w\rond R\right)}
\]
Since $\alpha _ws_w$ is a context (Lemma~\ref{lem:contextStableTree}(i)), this entails by~\eqref{convContext} that
\[
	\Proba_{\nu}{\left( U_n\in w\rond R\right)}
\tooSous{n\to\infty}
\casc(w)\pi\left(\alpha _ws_w\rond R\right)
=\pi\left( w\rond R\right).
\]

\cercle2
Assume now that $w$ is an internal word.
In this case, for any $n\geq 0$, by disjoint union, since $U_n$ does not charge any infinite word (because $U_0$ does not)
\[
	\Proba_{\nu}{\left(U_n\in w\rond R\right)}
	=\sum _{c\in\rond C^f,~c=w\cdots}\Proba_{\nu}{\left( U_n\in c\rond R\right)}.
\]
Distinguish then the long enough contexts from other ones by defining $\alpha _n(w)$ and $\beta _n(w)$ as the real numbers
\[
\alpha _n(w)
	=\sum _{\substack{c\in\rond C^f,~c=w\cdots\\[2pt]p_c\leq n}}\Proba_{\nu}{\left(U_n\in c\rond R\right)}
{\rm ~~and~~}
\beta _n(w)
	=\sum _{\substack{c\in\rond C^f,~c=w\cdots\\[2pt]p_c\geq n+1}}\Proba_{\nu}{\left(U_n\in c\rond R\right)}.
\]
Deal first with $\alpha _n(w)$ that can be written,
\begin{equation}
\label{alpha}
\alpha _n(w)
=\sum _{\alpha s\in\rond S}
\sum _{\substack{c\in\rond C^f\\[1pt]c=w\cdots=\cdots [\alpha s]\\[1pt]p_c\leq n}}
	\Proba_{\nu}{\left( U_n\in c\rond R\right)}
=\sum _{\alpha s\in\rond S}
\sum _{\substack{c\in\rond C^f\\[1pt]c=w\cdots=\cdots [\alpha s]\\[1pt]p_c\leq n}}
	\casc (c)\Proba_{\nu}{\left( U_{n-p_c}\in\alpha s\rond R\right)}.
\end{equation}
For a given $\alpha s\in\rond S$, by hypothesis, the cascade series
\[
\sum _{\substack{c\in\rond C^f\\[1pt]c=\cdots [\alpha s]}}\casc (c)
\]
converges.
Thus, the convergence in the last sum of~\eqref{alpha} is dominated so that, using~\eqref{convContext} again,
\[
\sum _{\substack{c\in\rond C^f\\[1pt]c=w\cdots=\cdots [\alpha s]\\[1pt]p_c\leq n}}
\casc (c)\Proba_{\nu}{\left(U_{n-p_c}\in\alpha s\rond R\right)}
\tooSous{n\to\infty}
\sum _{\substack{c\in\rond C^f\\[1pt]c=w\cdots=\cdots [\alpha s]}}
\casc (c)\pi\left(\alpha s\rond R\right).
\]
Since $\rond S$ is finite, one gets finally
\[
\alpha _n(w)
=\sum _{\alpha s\in\rond S}
\sum _{\substack{c\in\rond C^f\\[1pt]c=w\cdots=\cdots [\alpha s]\\[1pt]p_c\leq n}}
\casc (c)\Proba_{\nu}{\left(U_{n-p_c}\in\alpha s\rond R\right)}
\tooSous{n\to\infty}
\sum _{\alpha s\in\rond S}
\sum _{\substack{c\in\rond C^f\\[1pt]c=w\cdots=\cdots [\alpha s]}}
\casc (c)\pi\left(\alpha s\rond R\right)
=\pi\left(w\rond R\right).
\]
\cercle3
Extend now the definition of the $\alpha _n$ and $\beta _n$ to any finite word, by denoting
\[
\alpha _n(w)
=\Proba_{\nu}{\left(U_n\in w\rond R\right)}
{\rm ~~and~~}
\beta _n(w)=0
\]
whenever $w$ is noninternal.
With this notation, if $w$ is \emph{any} finite word
\begin{equation}
\label{sumAlphaBeta}
	\Proba_{\nu}{\left(U_n\in w\rond R\right)}
=\alpha _n(w)+\beta _n(w)
\end{equation}
and it is shown in \cercle1 and \cercle2 that
\begin{equation}
\label{convAlpha}
\alpha _n(w)
\tooSous{n\to\infty}
=\pi\left(w\rond R\right)
{\rm ~~and~~}
\beta _n(w)\geq 0.
\end{equation}
Take finally any finite word $w$ and denote its length by $N$.
By disjoint union, if $n\geq 1$,
\begin{equation}
\label{unEgale}
	1=\sum _{v\in\rond W,~|v|=N}\Proba_{\nu}{\left( U_n\in v\rond R\right)}
=\sum _{v\in\rond W,~|v|=N}\alpha _n(v)+\sum _{v\in\rond W,~|v|=N}\beta _n(v).
\end{equation}
Using~\eqref{convAlpha}, this sum being finite, one gets
\begin{equation}
\label{alphaConvUn}
\sum _{v\in\rond W,~|v|=N}\alpha _n(v)
\tooSous{n\to\infty}
\sum _{v\in\rond W,~|v|=N}\pi\left( v\rond R\right)
=1,
\end{equation}
the last equality resulting from disjoint union.
Putting~\eqref{unEgale} and~\eqref{alphaConvUn} together shows that
\[
\lim _{n\to\infty}\sum _{v\in\rond W,~|v|=N}\beta _n(v)=0.
\]
In particular, if $w$ is \emph{any} finite word,
\begin{equation}
\label{convBeta}
\lim _{n\to\infty}\beta _n(w)=0.
\end{equation}
Thus,~\eqref{sumAlphaBeta},~\eqref{convAlpha} and~\eqref{convBeta} show the result.
\end{proof}

\begin{proof}[Proof of Theorem \ref{semiMarkovIsVLMC} (A semi-Markov chain on a finite state space is a stable VLMC)]
\label{proof:semiMarkovIsVLMC}

\ 

Let $\left( Z_n\right) _{n\geq 0}$ be a semi-Markov chain with true jumps on a state space of cardinality $b$.
Without lack of generality, for technical convenience, take $\rond A=\g Z /b\g Z$.
Denote by $\left( J_N\right) _{N\geq 0}$ the internal chain of $\left( Z_n\right) _n$, namely the process of the successive \emph{different} states of $\left( Z_n\right) _{n\geq 0}$ and by $\left( T_N\right) _{N\geq 0}$ the process of the successive sojourn times in the visited states, with the convention $T_0=0$.
Assuming that $\left( Z_n\right) _{n\geq 0}$ is semi-Markov amounts to supposing that almost surely, for every $N\geq 0$, $\alpha\in\rond A$, $k\geq 1$,
\[
\Proba\left( J_{N+1}=\alpha ,T_{N+1}=k\big| J_0,\cdots ,J_N,T_1,\cdots ,T_N\right)
=\Proba\left( J_{N+1}=\alpha ,T_{N+1}=k\big| J_N\right).
\]
Denote by $p=\left( p_{\alpha ,\beta}(k)\right) _{\alpha,\beta\in\rond A, k\geq 1}$ the semi-Markov kernel of the process:
\[
p_{\alpha ,\beta}(k)=\Proba\left( J_{N+1}=\beta ,T_{N+1}=k\big|J_N=\alpha\right).
\]
For any $n\geq 0$, define the right-infinite random sequence
\[
V_n=Z_nZ_{n-1}\cdots Z_1Z_0\left(1+Z_0\right)Y_{-2},
\]
where $Y_{-2}$ is any non atomic distribution on $\rond R$.
Remember that the process $\left( Z_n\right) _{n\geq 0}$ takes its values in the set 
$\rond A=\g Z /b\g Z$.
The semi-Markov property of $\left( Z_n\right) _{n\geq 0}$ guarantees that the $\rond R$-valued process $\left( V_n\right) _{n\geq 0}$ is Markovian.

\vskip 5pt
Define now  $\left( U_n\right) _{n\geq 0}$ as the VLMC on the $b$-comb with alphabet $\rond A$, with the following parameters.
As an initial state, take $U_0=Z_0\left( 1+Z_0\right)Y_{-2}$.
As a transition probability associated with the context $\beta ^\ell\gamma$ where $\beta ,\gamma\in\rond A$, $\beta\neq\gamma$ and $\ell\geq 1$, take the probability measure $q_{\beta ^\ell\gamma}$ on $\rond A$:
{\small
\begin{equation}
\label{quelsQ!}
\forall\alpha\in\rond A,
q_{\beta ^\ell\gamma}(\alpha )=
\left\{
\begin{array}{l}
\displaystyle
\frac
{p_{\beta,\alpha}(\ell )}
{\displaystyle\sum _{k\geq \ell}\left(\sum _{\delta\in\rond A}p_{\beta,\delta}(k)\right)}
{\rm ~if~}\alpha\neq\beta\\[50pt]
\displaystyle
1-\frac
{\displaystyle\sum _{\delta\in\rond A}p_{\beta,\delta}(\ell )}
{\displaystyle\sum _{k\geq \ell}\left(\sum _{\delta\in\rond A}p_{\beta,\delta}(k)\right)}
=\frac
{\displaystyle\sum _{k\geq \ell +1}\left(\sum _{\delta\in\rond A}p_{\beta,\delta}(k)\right)}
{\displaystyle\sum _{k\geq \ell}\left(\sum _{\delta\in\rond A}p_{\beta,\delta}(k)\right)}
{\rm ~if~}\alpha =\beta.
\end{array}
\right.
\end{equation}
}

Note that $q_{\beta ^\ell\gamma}$ does not depend on $\gamma$ and that the convergence of the series is guaranteed by the properties of the semi-Markov kernel $p$.

\vskip 5pt
We show that both $\rond R$-valued Markov processes $\left( U_n\right) _{n\geq 0}$ and $\left( V_n\right) _{n\geq 0}$ have the same distribution stating that, for any $\alpha\in\rond A$ and $n\geq 0$, almost surely,
\[
\Proba\left( V_{n+1}=\alpha V_n\big| V_n\right)
=q_{\pref\left( V_n\right)}(\alpha ).
\]
Since $V_0$ and $U_0$ have the same distribution, this will entail the result.

\vskip 5pt
Let $r\in\rond R$, not of the form $\beta ^\infty$, $\beta\in\rond A$;
write $r=\beta^\ell s$ where $\ell\geq 1$, $\beta\in\rond A$ and $s\in\rond R$ starts with a letter different from $\beta$.
Denote, by $S_n$ the partial sum of $T_k$, as in Definition~\ref{def:semiMarkov}.
For any $\alpha\in\rond A$, $\alpha\neq\beta$ and for any $n\geq 0$, as soon as $N$ is such that $S_N\leq n\leq S_{N+1}-1$,
\[
\Proba\left( V_{n+1}=\alpha V_n\big| V_n=r\right)
=\Proba\left( J_{N+1}=\alpha ,T_{N+1}=\ell \big| J_N=\beta, T_{N+1}\geq\ell, V_{n-\ell}=s\right).
\]
Because of the semi-Markov property, this entails that
\[
\Proba\left( V_{n+1}=\alpha V_n\big| V_n=r\right)
=\Proba\left( J_{N+1}=\alpha ,T_{N+1}=\ell \big| J_N=\beta, T_{N+1}\geq\ell\right).
\]
Using the semi-Markov kernel, this leads to
\[
\Proba\left( V_{n+1}=\alpha V_n\big| V_n=r\right)
=\frac{\Proba\left( J_{N+1}=\alpha ,T_{N+1}=\ell \big| J_N=\beta\right)}
{\Proba\left( T_{N+1}\geq\ell \big| J_N=\beta\right)}
=\frac
{p_{\beta,\alpha}(\ell )}
{\displaystyle\sum _{\delta\in\rond A}\sum _{k\geq \ell}p_{\beta,\delta}(k)}.
\]
This shows that Formulae~\eqref{quelsQ!} are the suitable ones to ensure that both Markov processes $\left( U_n\right) _{n\geq 0}$ and $\left( V_n\right) _{n\geq 0}$ have the same distribution.
\end{proof}
\begin{proof}[Proof of Theorem \ref{th:CNSpourSM} (Limit distribution of a semi-Markov chain)]
\label{proof:limitSMC}

\

Let $b$ be the cardinal of the state space $\rond E$.
Let $U$ be the VLMC defined on the $b$-comb (Definition~\ref{def:bcomb}) by the transition probabilities of Formula~\eqref{quelsQ!}.
This context tree (the $b$-comb) is defined on the alphabet $\rond A=\rond E$;
it admits a finite number of context alpha-LIS, namely the words $\alpha\beta$ where $\alpha$ and $\beta$ are two distincts elements of $\rond E$;
furthermore, in this tree, the set of contexts having $\alpha\beta$ as an alpha-LIS is $\set{\alpha ^\ell\beta:~\ell\geq 1}$.
The processes $\left( Z_n\right) _n$ and $U$ are related by Theorem~\ref{semiMarkovIsVLMC}.
Moreover, Theorem \ref{cor:finite} applies to $U$.
Indeed, thanks to Formulas~\eqref{quelsQ!}, the assumption on the positivity of the semi-Markov kernel $p$ implies that the associated VLMC $U$ is non-null.

\vskip 5pt
By Theorem \ref{cor:finite}, the convergence of the cascade series implies the convergence $\Proba\left( U_n\in w\rond R\right)
\to
\pi\left( w\rond R\right)$
when $n$ tends to infinity,
for any finite word~$w$. 
As a consequence, as soon as the cascade series converge, $\Proba\left( Z_n = \beta | Z_0 = \alpha \right)$ converges to $\pi\left(\beta\rond R\right)$ when $n$ tends to infinity, and this limit does not depend on $\alpha$. 
In other words,  $\left( Z_n\right) _{n\geq 0}$
admits $\left(\pi\left( \beta\rond R\right)\right) _{\beta\in\rond E}$ as a limit distribution.

\vskip 5pt
Conversely, assume that $\left( Z_n\right) _n$ admits a limit distribution $\mu$.
In terms of the VLMC $\left( U_n\right) _n$, this implies that $\Proba\left( U_n\in\alpha\rond R\right)$ tends to $\mu (\alpha )$ when $n$ tends to infinity, for every $\alpha\in\rond E$.
Let $\alpha,\beta\in\rond E$ be two distinct letters, so that $\alpha\beta\in\rond S$.
For every $n\geq 1$,
\[
\Proba\left( U_n\in\alpha\beta\rond R\right)
=\Proba\left( Z_n=\alpha,Z_{n-1}=\beta\right)
=p_{\beta,\alpha}(1)\Proba\left( Z_{n-1}=\beta\right)
\]
so that 
\begin{equation}
\label{deuxLettres}
\Proba\left( U_n\in\alpha\beta\rond R\right)\tooSous{n\to\infty}p_{\beta,\alpha}(1)\mu (\beta ).
\end{equation}

This has two consequences.

Firstly, $\mu (\alpha )\neq 0$ for every $\alpha\in\rond E$.
Indeed, since $\mu$ is a probability measure, all these numbers cannot simultaneously vanish;
let thus $\beta\in\rond E$ such that $\mu (\beta )\neq 0$.
Let also $\alpha\in\rond E\setminus\set{\beta}$.
For any $n\geq 1$, as before,
\[
\Proba\left( Z_nZ_{n-1}=\alpha\beta\right)
=p_{\beta,\alpha}(1)\Proba\left( Z_{n-1}=\beta\right).
\]
Assume that $\mu (\alpha )=0$.
Since $\Proba\left( Z_nZ_{n-1}=\alpha\beta\right)\leq\Proba\left( Z_n=\alpha\right)$, the left hand side of this equality tends to $0$ when $n$ tends to infinity, which is impossible because the limit right hand side, namely $p_{\beta,\alpha}(1)\mu (\beta )$, is non-null.
Thus, $\mu (\alpha )\neq 0$.
 
Secondly, for every $\ell\geq 1$,
\[
\Proba\left( U_n\in\alpha ^\ell\beta\rond R\right)
=\casc\left(\alpha ^\ell\beta\right)\Proba\left( U_{n-l+1}\in\alpha\beta\rond R\right)
\tooSous{n\to\infty}\casc\left(\alpha ^\ell\beta\right)p_{\beta,\alpha}(1)\mu (\beta ).
\]
Thus, if $L$ is a positive integer,
\begin{equation}
\label{limitNiveau}
\sum _{\ell =1}^L\Proba\left( U_n\in\alpha ^\ell\beta\rond R\right)
\tooSous{n\to\infty}
\left[\sum _{\ell =1}^L\casc\left(\alpha ^\ell\beta\right)\right]
p_{\beta,\alpha}(1)\mu (\beta ).
\end{equation}
For any $n$, the left hand side of Formula~\eqref{limitNiveau} if less than $1$ because it is bounded above by the whole sum $\sum _{c\in\rond C^f}\Proba\left( U_n\in c\rond R\right) +\sum _{c\in\rond C^i}\Proba\left( U_n=c\right)$ which equals $1$ by disjoint union.
Thus, the right hand side is also bounded.
In particular, since $p_{\beta,\alpha}(1)$ and $\mu (\beta )$ are non-null, the series of positive numbers $\sum _{\ell}\casc\left(\alpha ^\ell\beta\right)$ converge.
Besides, the cascade series of the context alpha-LIS $\alpha\beta$ is precisely $\kappa _{\alpha\beta}=\sum _{\ell\geq 1}\casc\left(\alpha ^\ell\beta\right)$.
This allows us to conclude that the existence of a limit distribution for the process $\left( Z_n\right) _n$ implies the convergence of all the cascade series.

\vskip 5pt
It suffices now to show that condition $(ii)$ is equivalent to the convergence of every cascade series
\[
\sum _{\ell\geq 1}\casc\left(\alpha ^\ell\beta\right).
\]
Let $\alpha$ and $\beta$ be two distinct elements of $\rond A$.
For any $\ell\geq 1$, since
\[
\casc\left(\alpha ^\ell\beta\right)=\prod _{k=1}^{\ell -1}q_{\alpha ^k\beta} (\alpha),
\]
an immediate reading of Formula~\eqref{quelsQ!} shows that
\[
\casc\left(\alpha ^\ell\beta\right)=
\frac
{\displaystyle\sum _{k\geq \ell}\left(\sum _{\gamma\in\rond A}p_{\alpha,\gamma}(k)\right)}
{\displaystyle\sum _{k\geq 1}\left(\sum _{\gamma\in\rond A}p_{\alpha,\gamma}(k)\right)},
\]
so that the sum of these cascades writes
\[
\sum _{\ell\geq 1}\casc\left(\alpha ^\ell\beta\right)
=\frac
{\displaystyle\sum _{k\geq 1}k\left(\sum _{\gamma\in\rond A}p_{\alpha,\gamma}(k)\right)}
{\displaystyle\sum _{k\geq 1}\left(\sum _{\gamma\in\rond A}p_{\alpha,\gamma}(k)\right)}\in [0,+\infty],
\]
leading to the result.
Note that the cascade series $\sum _{\ell\geq 1}\casc\left(\alpha ^\ell\beta\right)$ does not depend on $\beta$, because all the probability measures one has to associate with the leaves of a given context of the comb are all the same ones, as can be seen on Formula~\eqref{quelsQ!}.
\end{proof}

\section{Open problems and conjectures}
\label{sec:open}
\subsection{Right-fixed vectors for $Q$}
\label{rem:conjQvp1}

Take a probabilised context tree.
When the tree  is stable and whenever the sequence $\left(\kappa _{\alpha s}(n)\right) _{n}$ converge to $0$ for every $\alpha s\in\rond S$, the square matrix $Q$ can be seen as the transition matrix of some $\rond S$-valued Markov chain, so that it turns out to be stochastic -- see Proposition~\ref{pro:stochasticity}.
This is not true in general if one removes the stability assumption (Remark~\ref{rem:QfaitPasTout}).
We nevertheless make the following conjecture.
\begin{conj}
For any probabilised context tree, whenever the sequence $\left(\kappa _{\alpha s}(n)\right) _{n}$ converge to $0$ for every $\alpha s\in\rond S$, the matrix $Q$ always admits $1$ as a \emph{right}-eigenvalue.
\end{conj}
In particular, thanks to Theorem~\ref{fQbij}, if a context tree has a finite set of alpha-LIS and if this conjecture is true, then the corresponding VLMC always admits at least one invariant probability measure as soon as its (finitely many) cascade series converge.

\subsection{Convergence of cascade series}

Consider two very simple examples on the alphabet $\rond A=\set{0,1}$, pictured hereunder:
the left comb and the bamboo blossom -- see~\cite{cenac/chauvin/paccaut/pouyanne/12} for a complete treatment of stationary probability measures for these VLMC.
It turns out that the left comb gets one context alpha-LIS and thus one cascade series, that can be convergent or not depending on the distributions $q_c$.
The bamboo blossom gets two context alpha-LIS, both cascade series being always convergent with geometrical rates whatever the (non-null) distributions $q_c$ are.
This phenomenon, which seems to be generalizable, leads us to the following conjecture.

\begin{center}
\begin{tikzpicture}[scale=0.5]
\tikzset{every leaf node/.style={draw,circle,fill,scale=0.8},every internal node/.style={draw,circle,scale=0.01}}
\Tree [.{} [.{} [.{} [.{} [.{}  \edge[line width=2pt,dashed];\node[fill=white,draw=white]{};\edge[draw=white];\node[fill=white,draw=white]{}; {} ] {} ] {} ] {} ] {} ]
\draw (0,-7) node{The left comb};
\end{tikzpicture}
\hskip 100pt
\begin{tikzpicture}[scale=0.4]
\tikzset{every leaf node/.style={draw,circle,fill},every internal node/.style={draw,circle,scale=0.01}}
\Tree [.{} [.{} {} [.{} [.{} {} [.{} [.{} {} [.{} \edge[line width=2pt,dashed];\node[fill=white,draw=white]{};\edge[draw=white];\node[fill=white,draw=white]{}; {} ] ] {} ] ] {} ] ] {} ]
\draw (0,-8.5) node{The bamboo blossom};
\end{tikzpicture}
\end{center}

\begin{conj}
Take a non-null probabilised context tree.
When the tree does not have any infinite shift-stable subtree, all the cascade series converge, with geometrical rates.
\end{conj}

\subsection{Vanishing of cascades and $\sigma$-finite invariant measures}

Take a \emph{stable} probabilised context tree.
As recalled just above (Section~\ref{rem:conjQvp1}), whenever the sequence $\left(\kappa _{\alpha s}(n)\right) _{n}$ converge to $0$ for every $\alpha s\in\rond S$ (we call this assumption \emph{vanishing of cascades}), the square matrix $Q$ is stochastic by Proposition~\ref{pro:stochasticity}.
Moreover, Theorem~\ref{fQbij} or Theorem~\ref{th:stable} asserts that the convergence of cascade series is a necessary condition for the VLMC to admit an invariant probability measure.
As stated herunder, the vanishing of cascades is conjectured to be a necessary condition for the VLMC to admit an invariant \emph{$\sigma$-finite} measure.

\begin{conj}
Let $U$ be a VLMC defined by a probabilised \emph{stable} context tree.
Assume that $U$ admits an invariant $\sigma$-finite measure.
Then, for every $\alpha s\in\rond S$, the sequence $\left(\kappa _{\alpha s}(n)\right) _{n}$ tends to $0$ when $n$ tends to infinity
(and, consequently, $Q$ is stochastic).
\end{conj}

\section{Appendix: an example of invariant $\sigma$-finite measure that charges irrational infinite contexts}
\label{sec:appendix}

A soon as a non-null VLMC admits an invariant probability measure, all infinite words are negligible -- see Lemma~\ref{lem:pineq0}(ii).
Besides, the same argument as in the proof of that lemma shows that an invariant $\sigma$-finite measure always vanishes on \emph{rational} right-infinite words, \emph{i.e.}~on eventually periodic words.
This appendix provides an example of non-null VLMC that admits an invariant $\sigma$-finite measure which gets positive values on infinitely many (necessarily irrational) contexts\footnote{In particular, one cannot get rid of the finiteness assumption of an invariant measure to prove Lemma~\ref{lem:pineq0}(ii).}.

\vskip 10pt
In this appendix, a ``$\sigma$-finite measure'' denotes a positive non-zero measure on $\rond R$ which is finite on all cylinders $c\rond R$ based on finite contexts $c$, and also, necessarily, on infinite contexts.
In particular, since the contexts induce a partition of $\rond R$, such a measure is truly $\sigma$-finite.
As usual in the field of Markov chains, when $U$ is a VLMC, the definition of an $U$-invariant probability measure can be extended to $\sigma$-finite measures using the transition probability kernel $P_U$, defined by 
\[
P_U(r,B)=\sum _{\alpha\in\rond A}q_{\pref (r)}(\alpha )\ind{\alpha r\in B}
\] 
on Borel sets $B$ and right-infinite words $r$:
this kernel acts on $\sigma$-finite measures $\pi$ through the formula 
\[\pi P_U(B):=\int _{\rond R}P_U(x,B)d\pi(x)
\] 
(this is an action on the right), and a $\sigma$-finite measure $\pi$ is said $U$-invariant whenever $\pi P_U=\pi$.

\vskip 5pt
In what follows, we describe the announced example in the form of a sequence of hints and assertions that can be easily (but sometimes laboriously) verified.

\vskip 5pt
Consider the irrational right-infinite word $a=1010^210^310^4\dots$.
Define as follows \emph{ the stabilised arithmetic tree}, denoted by $\rond T_a$:
it is the \emph{stable} context tree on the alphabet $\rond A=\set{0,1}$ spanned by $a$, \emph{i.e.} the smallest context tree that contains all the shifted words $\sigma ^n(a)$, $n\geq 0$ (see Definition~\ref{def:shift}).
On the left side of the following picture, one can find drawings of the successive context trees $t_n$ spanned by the shifted infinite word $\sigma ^n(a)$, $n\geq 1$.
They are used to give a representation of $\rond T_a$ on the right side of the picture.

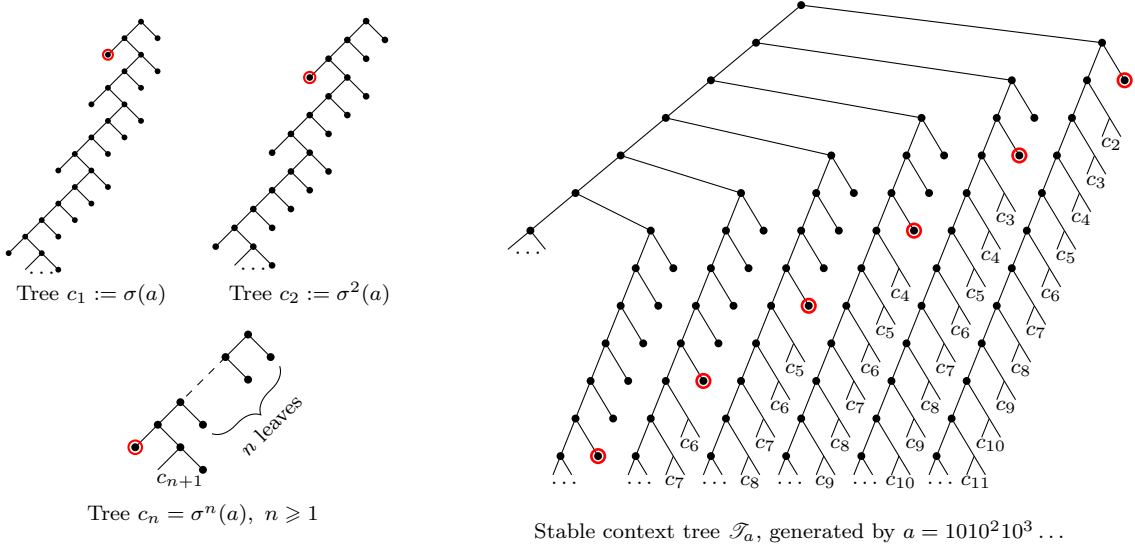
\begin{center}
\begin{minipage}{0.4\textwidth}
\begin{center}
\newcommand{\pente}{1}
\newcommand{\pt}{5}
\newcommand{\peigne}[3]{
\foreach \j in {1,...,#3} \draw (#1-\j*\pente+\pente,#2-\j+1)--++(-\pente,-1);
\foreach \j in {1,...,#3} \draw (#1-\j*\pente+\pente,#2-\j+1)--++(\pente,-1);
\foreach \j in {1,...,#3} \fill (#1-\j*\pente,#2-\j) circle (\pt pt);
\foreach \j in {1,...,#3} \fill (#1-\j*\pente+2*\pente,#2-\j) circle (\pt pt);
}
\newcommand{\alphaLis}[2]{\draw (#1,#2) [color=red,line width=0.8pt] circle (1.7*\pt pt);}
\begin{tikzpicture}[scale=0.22]
\fill (0,0) circle (\pt pt);
\peigne{0}{0}2
\peigne{-0*\pente}{-2}3
\peigne{-1*\pente}{-5}4
\peigne{-3*\pente}{-9}5
\draw (-6*\pente,-14)--++(-\pente,-1);
\draw (-6*\pente,-14)--++(\pente,-1);\fill (-5*\pente,-15) circle (\pt pt);
\draw (-6*\pente+0.1,-15.2) node{\scriptsize $\cdots$};
\draw (-3,-16.5) node{\scriptsize Tree $c_1:=\sigma (a)$};
\alphaLis{-2*\pente}{-2}
\end{tikzpicture}
\hskip 10pt
\begin{tikzpicture}[scale=0.25]
\fill (0,0) circle (\pt pt);
\peigne{0}{0}3
\peigne{-1*\pente}{-3}4
\peigne{-3*\pente}{-7}5
\draw (-6*\pente,-12)--++(-\pente,-1);
\draw (-6*\pente,-12)--++(\pente,-1);\fill (-5*\pente,-13) circle (\pt pt);
\draw (-6*\pente+0.1,-13.2) node{\scriptsize $\cdots$};
\draw (-3,-14.5) node{\scriptsize Tree $c_2:=\sigma ^2(a)$};
\alphaLis{-3*\pente}{-3}
\end{tikzpicture}

\vskip 5pt
\begin{tikzpicture}[scale=0.3]
\fill (0,0) circle (\pt pt);
\draw (0,0)--++(-1,-1);\fill (-1,-1) circle (\pt pt);
\draw (0,0)--++(1,-1);\fill (1,-1) circle (\pt pt);
\draw [dashed] (-1,-1)--++(-2,-2);\fill (-3,-3) circle (\pt pt);
\draw (-1,-1)--++(1,-1);\fill (0,-2) circle (\pt pt);
\draw (-3,-3)--++(1,-1);\fill (-2,-4) circle (\pt pt);
\draw (-3,-3)--++(-1,-1);\fill (-4,-4) circle (\pt pt);
\draw (-4,-4)--++(-1,-1);\fill (-5,-5) circle (\pt pt);
\draw (-4,-4)--++(1,-1);\fill (-3,-5) circle (\pt pt);
\draw [decorate,decoration={brace,amplitude=8pt},yshift=0pt]
(1.4,-1.4) --++ (-3,-3) node [black,midway,xshift=0.8cm] {};
\draw (1,-4) node[rotate=45]{\scriptsize $n$ leaves};
\draw (-3,-5)--++(-1,-1);
\draw (-3,-5)--++(1,-1);\fill (-2,-6) circle (\pt pt);
\draw (-3,-6.5) node{\scriptsize $c_{n+1}$};
\alphaLis{-5*\pente}{-5}
\draw (-2,-8) node{\scriptsize Tree $c_n=\sigma ^n(a),~n\geq 1$};
\end{tikzpicture}

\end{center}
\end{minipage}
\begin{minipage}{0.55\textwidth}
\begin{center}
\newcommand{\penteGPeignes}{-0.4}
\newcommand{\penteDPeignes}{0.6}
\newcommand{\penteGGrandPeigne}{-1.2}
\newcommand{\penteDGrandPeigne}{8}
\newcommand{\tp}{3}
\newcommand{\dd}{0.15}
\newcommand{\alis}[2]{\draw (#1,#2) [color=red,line width=1pt] circle (1.7*\tp pt);}
\newcommand{\queue}[3]{
\draw (#1,#2)--++(0.5*\penteGPeignes,-0.5);
\draw (#1,#2)--++(0.5*\penteDPeignes,-0.5);
\draw (#1+0.25*\penteGPeignes+0.25*\penteDPeignes,#2-0.7) node{#3};
}
\newcommand{\peigneDeux}[4]{
\draw (#1,#2)--++(#3*\penteGPeignes,-#3);
\fill (#1,#2) circle (#4 pt);
\fill (#1+\penteDPeignes,#2-1) circle (#4 pt);
\foreach \j in {1,...,#3} \draw (#1+\j*\penteGPeignes-\penteGPeignes,#2-\j+1)--++(\penteDPeignes,-1);
\foreach \j in {1,...,#3} \fill (#1+\j*\penteGPeignes,#2-\j) circle (#4 pt);
\foreach \j in {2,...,#3} \queue{#1+\j*\penteGPeignes-\penteGPeignes+\penteDPeignes}{#2-\j}{\scriptsize $c_{\j}$};
\queue{#1+#3*\penteGPeignes}{#2-#3}{\scriptsize $\ \dots$}
}
\newcommand{\peigneTrois}[4]{
\draw (#1,#2)--++(#3*\penteGPeignes,-#3);
\fill (#1,#2) circle (#4 pt);
\fill (#1+\penteDPeignes,#2-1) circle (#4 pt);
\fill (#1+\penteGPeignes+\penteDPeignes,#2-2) circle (#4 pt);
\foreach \j in {1,...,#3} \draw (#1+\j*\penteGPeignes-\penteGPeignes,#2-\j+1)--++(\penteDPeignes,-1);
\foreach \j in {1,...,#3} \fill (#1+\j*\penteGPeignes,#2-\j) circle (#4 pt);
\foreach \j in {3,...,#3} \queue{#1+\j*\penteGPeignes-\penteGPeignes+\penteDPeignes}{#2-\j}{\scriptsize $c_{\j}$};
\queue{#1+#3*\penteGPeignes}{#2-#3}{\scriptsize $\ \dots$}
}
\newcommand{\peigneQuatre}[4]{
\draw (#1,#2)--++(#3*\penteGPeignes,-#3);
\fill (#1,#2) circle (#4 pt);
\fill (#1+\penteDPeignes,#2-1) circle (#4 pt);
\fill (#1+\penteGPeignes+\penteDPeignes,#2-2) circle (#4 pt);
\fill (#1+2*\penteGPeignes+\penteDPeignes,#2-3) circle (#4 pt);
\foreach \j in {1,...,#3} \draw (#1+\j*\penteGPeignes-\penteGPeignes,#2-\j+1)--++(\penteDPeignes,-1);
\foreach \j in {1,...,#3} \fill (#1+\j*\penteGPeignes,#2-\j) circle (#4 pt);
\foreach \j in {4,...,#3} \queue{#1+\j*\penteGPeignes-\penteGPeignes+\penteDPeignes}{#2-\j}{\scriptsize $c_{\j}$};
\queue{#1+#3*\penteGPeignes}{#2-#3}{\scriptsize $\ \dots$}
}
\newcommand{\peigneCinq}[4]{
\draw (#1,#2)--++(#3*\penteGPeignes,-#3);
\fill (#1,#2) circle (#4 pt);
\fill (#1+\penteDPeignes,#2-1) circle (#4 pt);
\fill (#1+\penteGPeignes+\penteDPeignes,#2-2) circle (#4 pt);
\fill (#1+2*\penteGPeignes+\penteDPeignes,#2-3) circle (#4 pt);
\fill (#1+3*\penteGPeignes+\penteDPeignes,#2-4) circle (#4 pt);
\foreach \j in {1,...,#3} \draw (#1+\j*\penteGPeignes-\penteGPeignes,#2-\j+1)--++(\penteDPeignes,-1);
\foreach \j in {1,...,#3} \fill (#1+\j*\penteGPeignes,#2-\j) circle (#4 pt);
\foreach \j in {5,...,#3} \queue{#1+\j*\penteGPeignes-\penteGPeignes+\penteDPeignes}{#2-\j}{\scriptsize $c_{\j}$};
\queue{#1+#3*\penteGPeignes}{#2-#3}{\scriptsize $\ \dots$}
}
\newcommand{\peigneSix}[4]{
\draw (#1,#2)--++(#3*\penteGPeignes,-#3);
\fill (#1,#2) circle (#4 pt);
\fill (#1+\penteDPeignes,#2-1) circle (#4 pt);
\fill (#1+\penteGPeignes+\penteDPeignes,#2-2) circle (#4 pt);
\fill (#1+2*\penteGPeignes+\penteDPeignes,#2-3) circle (#4 pt);
\fill (#1+3*\penteGPeignes+\penteDPeignes,#2-4) circle (#4 pt);
\fill (#1+4*\penteGPeignes+\penteDPeignes,#2-5) circle (#4 pt);
\foreach \j in {1,...,#3} \draw (#1+\j*\penteGPeignes-\penteGPeignes,#2-\j+1)--++(\penteDPeignes,-1);
\foreach \j in {1,...,#3} \fill (#1+\j*\penteGPeignes,#2-\j) circle (#4 pt);
\foreach \j in {6,...,#3} \queue{#1+\j*\penteGPeignes-\penteGPeignes+\penteDPeignes}{#2-\j}{\scriptsize $c_{\j}$};
\queue{#1+#3*\penteGPeignes}{#2-#3}{\scriptsize $\ \dots$}
}
\newcommand{\peigneSept}[4]{
\draw (#1,#2)--++(#3*\penteGPeignes,-#3);
\fill (#1,#2) circle (#4 pt);
\fill (#1+\penteDPeignes,#2-1) circle (#4 pt);
\fill (#1+\penteGPeignes+\penteDPeignes,#2-2) circle (#4 pt);
\fill (#1+2*\penteGPeignes+\penteDPeignes,#2-3) circle (#4 pt);
\fill (#1+3*\penteGPeignes+\penteDPeignes,#2-4) circle (#4 pt);
\fill (#1+4*\penteGPeignes+\penteDPeignes,#2-5) circle (#4 pt);
\fill (#1+5*\penteGPeignes+\penteDPeignes,#2-6) circle (#4 pt);
\foreach \j in {1,...,#3} \draw (#1+\j*\penteGPeignes-\penteGPeignes,#2-\j+1)--++(\penteDPeignes,-1);
\foreach \j in {1,...,#3} \fill (#1+\j*\penteGPeignes,#2-\j) circle (#4 pt);
\queue{#1+#3*\penteGPeignes}{#2-#3}{\scriptsize $\ \dots$}
}

\begin{tikzpicture}[scale=0.5]
\fill (0,0) circle (\tp pt);
\draw (0,0)--++(6*\penteGGrandPeigne,-6);
\draw (0,0)--++(\penteDGrandPeigne,-1);
\peigneDeux{\penteDGrandPeigne}{-1}{11}{\tp}
\fill (\penteGGrandPeigne,-1) circle (\tp pt);
\draw (\penteGGrandPeigne,-1)--(\penteGGrandPeigne+\penteDGrandPeigne-\dd*\penteDGrandPeigne,-2);
\peigneTrois{\penteGGrandPeigne+\penteDGrandPeigne-\dd*\penteDGrandPeigne}{-2}{10}{\tp}
\fill (2*\penteGGrandPeigne,-2) circle (\tp pt);
\draw (2*\penteGGrandPeigne,-2)--(2*\penteGGrandPeigne+\penteDGrandPeigne-2*\dd*\penteDGrandPeigne,-3);
\peigneQuatre{2*\penteGGrandPeigne+\penteDGrandPeigne-2*\dd*\penteDGrandPeigne}{-3}{9}{\tp}
\fill (3*\penteGGrandPeigne,-3) circle (\tp pt);
\draw (3*\penteGGrandPeigne,-3)--(3*\penteGGrandPeigne+\penteDGrandPeigne-3*\dd*\penteDGrandPeigne,-4);
\peigneCinq{3*\penteGGrandPeigne+\penteDGrandPeigne-3*\dd*\penteDGrandPeigne}{-4}{8}{\tp}
\fill (4*\penteGGrandPeigne,-4) circle (\tp pt);
\draw (4*\penteGGrandPeigne,-4)--(4*\penteGGrandPeigne+\penteDGrandPeigne-4*\dd*\penteDGrandPeigne,-5);
\peigneSix{4*\penteGGrandPeigne+\penteDGrandPeigne-4*\dd*\penteDGrandPeigne}{-5}{7}{\tp}
\fill (5*\penteGGrandPeigne,-5) circle (\tp pt);
\draw (5*\penteGGrandPeigne,-5)--(5*\penteGGrandPeigne+\penteDGrandPeigne-5*\dd*\penteDGrandPeigne,-6);
\peigneSept{5*\penteGGrandPeigne+\penteDGrandPeigne-5*\dd*\penteDGrandPeigne}{-6}{6}{\tp}
\fill (6*\penteGGrandPeigne,-6) circle (\tp pt);
\draw (6*\penteGGrandPeigne,-6)--++(0.5*\penteGGrandPeigne,-0.5);
\draw (6*\penteGGrandPeigne,-6)--++(0.5*\penteDGrandPeigne-6*\dd*0.5*\penteDGrandPeigne,-0.5);
\draw(6*\penteGGrandPeigne+0.5*0.5*\penteGGrandPeigne+0.56*\dd*0.5*\penteDGrandPeigne,-6-0.6) node{\scriptsize $\!\cdots$};
\alis{\penteDGrandPeigne+\penteDPeignes}{-2}
\alis{\penteGGrandPeigne+\penteDGrandPeigne-\dd*\penteDGrandPeigne+\penteGPeignes+\penteDPeignes}{-4}
\alis{2*\penteGGrandPeigne+\penteDGrandPeigne-2*\dd*\penteDGrandPeigne+2*\penteGPeignes+\penteDPeignes}{-6}
\alis{3*\penteGGrandPeigne+\penteDGrandPeigne-3*\dd*\penteDGrandPeigne+3*\penteGPeignes+\penteDPeignes}{-8}
\alis{4*\penteGGrandPeigne+\penteDGrandPeigne-4*\dd*\penteDGrandPeigne+4*\penteGPeignes+\penteDPeignes}{-10}
\alis{5*\penteGGrandPeigne+\penteDGrandPeigne-5*\dd*\penteDGrandPeigne+5*\penteGPeignes+\penteDPeignes}{-12}
\draw (0,-14) node{\scriptsize Stable context tree $\rond T_a$, generated by $a=1010^210^3\dots$};
\end{tikzpicture}

\end{center}
\end{minipage}
\end{center}

Computing the contexts leads to show that $\rond T_a$ admits two one-parameter families of context alpha-LIS, namely $0^q10^q1$, $q\geq 0$ and $10^N10^{N+2}$, $N\geq 1$.
They are coloured red in the pictures.
Thus the related matrix $Q$ is infinite -- see~\eqref{defQ} for a definition of $Q$.
Moreover, the infinite contexts of $\rond T_a$ are the following ones:
on one side, the $0^\infty$ and the $0^n10^\infty$, $n\geq 0$ which are rational;
on the other side, the $\sigma ^n(a)$ which are irrational.

Given a non-null probabilising of $\rond T_a$, assume that an invariant $\sigma$-finite measure $\mu$ exists.
Reasoning like in the proof of Theorem~\ref{fQbij}, for every $\alpha s\in\rond S$, decompose the number $\mu\left(\alpha s\rond R\right)$ through the partition of $s\rond R$ induced by cylinders based on finite contexts that have $s$ as a prefix and by such infinite contexts as well.
These writings show that $Q$'s entries are necessarily finite (sums of summable families) as well as the sums $\sum _{c\in\rond C^i,~c=s\dots}q_c(\alpha )\mu (c)$.
Note further that $Q$ is row-stochastic because $\rond T_a$ is a stable tree (see Proposition~\ref{pro:stochasticity}).
Since all infinite contexts are shifted from $a$ and since $\mu$ is invariant, all the numbers $\mu (c)$, $c\in\rond C^i$ can be written as $\mu (c)=\frac{\mu (a)}{p_c}$ where $p_c$ is a finite product of $q_{\sigma ^k(a)}(\beta)$, $k\geq 0$, $\beta\in\rond A$.
Consequently, all the sums
\[
\ell _{\alpha s}:=\sum _{\substack{{c\in\rond C^i}\\[1pt]{c=s\dots}}}\frac{q_c(\alpha )}{p_c}
\]
are also finite
-- note that the sums $\ell _{0^q10^q1}$, $q\geq 1$ are reduced to a single term, all other $\ell _{\alpha s}$'s being true infinite sums.

Finally, like in the proof of Theorem~\ref{fQbij}, the above decompositions lead to the following statement:
a $\sigma$-finite measure $\mu$ is invariant if and only if it satisfies the (infinite) matricial equation
\begin{equation}
\label{mu_S}
\mu (a)\ell +\mu _{\rond S}Q=\mu _{\rond S}
\end{equation}
where $\ell$ and $\mu _{\rond S}$ denote the infinite row-vectors $\ell =\left( \ell _ {\alpha s}\right) _{\alpha s\in\rond S}$ and $\mu _{\rond S}=\left( \mu \left(\alpha s\rond R\right)\right) _{\alpha s\in\rond S}$.
Notice that after a straightforward continuation of the function $f$ to $\sigma$-finite measures, $\mu _{\rond S}=f(\mu)$ (see~\eqref{def:f}).
Using the vocabulary of ~\cite{kitchens/98}, one finally gets the following result.
\begin{pro}
\label{prop:arithmetique}
Let $U$ be some non-null VLMC defined from the context tree $\rond T_a$ and let $Q$ be its $Q$-matrix (see~\eqref{defQ}).

(i) If $Q$ is positive recurrent, then $U$ admits a unique half-line of invariant $\sigma$-finite measures.
All of them are finite ones.

(ii) If $Q$ is null recurrent, then $U$ admits a unique half-line of invariant $\sigma$-finite measures.
None of them are finite ones, but they turn all infinite words negligeable.

(iii) If $Q$ is transient, then $U$ does not admit any invariant probability measure.
\end{pro}

We build an example of VLMC on $\rond T_a$ that admits an invariant $\sigma$-finite measure that charges (all) irrational infinite contexts.
Proposition~\ref{prop:arithmetique} shows that the corresponding $Q$-matrix is necessarily transient.

\vskip 5pt
To exhibit such an example, one first have to compute the ``form'' of $Q$ (check which entries vanish, see below) and to make explicit the way how $\ell$'s and $Q$'s entries are expressed in terms of the~$q_c$.
This being done, one sees that any row-stochastic matrix $A$ having the form of $Q$ (same positive entries, same zero ones) is the $Q$-matrix of a probabilised context tree $\rond T_a$ (choice of the $q_c$, $c\in\rond C^f$).
Furthermore, if $X$ is any positive row-vector that satisfies $XA<X$ (strict inequality for every coordinate), the vector $X-XA$ can be chosen as the $\ell$-vector of such a probabilised $\rond T_a$ (choice of the $q_c$, $c\in\rond C^i$).

\vskip 5pt
Therefore, thanks to Equation~\eqref{mu_S}, an example of invariant $\sigma$-finite measure that charges (all) irrational infinite contexts is given by any row-stochastic matrix $A$ having the required form, together with a positive row-vector $X$ that satisfies $XA<X$.
Note that such an $A$ is necessarily transient and that a corresponding $X$ has necessarily non summable coordinates (see~\cite{kitchens/98}).
Below, we give such a matrix $A$.

\vskip 5pt
Order the context alpha-LIS by increasing length, placing $10^{q-1}10^{q+1}$ before $0^q10^q1$ (both alpha-LIS have the same length).
For this order, the form of $Q$ is written hereunder, a $*$ denoting a positive entry.
As heuristic hint, remark first that the (transient) matrix $V$ given hereunder and the positive vector $X=(1,3,2,5,4,7,6,9,8\dots)$ satisfy $XV<X$.

\[\scriptsize
Q=\left(
\begin{array}{*{20}c}
*&*&*&*&*&*&*&*&*&*&*&\dots\\
*&*&*&*&*&*&*&*&*&*&*&\dots\\
*&*&0&0&0&0&0&0&0&0&0&\dots\\
*&*&*&*&*&*&*&*&*&*&*&\dots\\
*&*&0&*&0&0&0&0&0&0&0&\dots\\
*&*&*&*&*&*&*&*&*&*&*&\dots\\
*&*&0&*&0&*&0&0&0&0&0&\dots\\
*&*&*&*&*&*&*&*&*&*&*&\dots\\
*&*&0&*&0&*&0&*&0&0&0&\dots\\
*&*&*&*&*&*&*&*&*&*&*&\dots\\
*&*&0&*&0&*&0&*&0&*&0&\dots\\
\vdots&\vdots&\vdots&\vdots&\vdots&\vdots&\vdots&\vdots&\vdots&\vdots&\vdots&\ddots
\end{array}
\right)
\hskip 10pt
V=\left(
\begin{array}{*{20}c}
0&0&\color{red}1&0&0&0&0&0&0&0&0&\cdots\\
0&0&0&0&\color{red}1&0&0&0&0&0&0&\cdots\\
0&\color{red}1&0&0&0&0&0&0&0&0&0&\cdots\\
0&0&0&0&0&0&\color{red}1&0&0&0&0&\cdots\\
0&0&0&\color{red}1&0&0&0&0&0&0&0&\cdots\\
0&0&0&0&0&0&0&0&\color{red}1&0&0&\cdots\\
0&0&0&0&0&\color{red}1&0&0&0&0&0&\cdots\\
0&0&0&0&0&0&0&0&0&0&\color{red}1&\cdots\\
0&0&0&0&0&0&0&\color{red}1&0&0&0&\cdots\\
0&0&0&0&0&0&0&0&0&0&0&\cdots\\
0&0&0&0&0&0&0&0&0&\color{red}1&0&\cdots\\
\vdots&\vdots&\vdots&\vdots&\vdots&\vdots&\vdots&\vdots&\vdots&\vdots&\vdots&\ddots
\end{array}
\right)
\]
The (transient) matrix $A$ we give is a deformation of $V$ that has the form of $Q$.
It satisfies $XA<X$ for the row-vector $X$ given above.
Let $r,s\in ]0,1[$ and let $A$ be the matrix
\[\scriptsize
A=\left(
\begin{array}{*{20}c}
\frac{r^2}{0!}&\frac{r^3}{1!}&R_1(r)&\frac{r^5}{3!}&\frac{r^6}{4!}&\frac{r^7}{5!}&\frac{r^8}{6!}&\frac{r^9}{7!}&\frac{r^{10}}{8!}&\frac{r^{11}}{9!}&\frac{r^{12}}{10!}&\cdots\\[5pt]
\frac{r^3}{0!}&\frac{r^4}{1!}&\frac{r^5}{2!}&\frac{r^6}{3!}&R_2(r)&\frac{r^8}{5!}&\frac{r^9}{6!}&\frac{r^{10}}{7!}&\frac{r^{11}}{8!}&\frac{r^{12}}{9!}&\frac{r^{13}}{10!}&\cdots\\[5pt]
s&1-s&0&0&0&0&0&0&0&0&0&\cdots\\[5pt]
\frac{r^5}{0!}&\frac{r^6}{1!}&\frac{r^7}{2!}&\frac{r^8}{3!}&\frac{r^9}{4!}&\frac{r^{10}}{5!}&R_4(r)&\frac{r^{12}}{7!}&\frac{r^{13}}{8!}&\frac{r^{14}}{9!}&\frac{r^{15}}{10!}&\cdots\\[5pt]
\frac{s^2}{2!}&\frac{s^2}{2!}&0&1-\frac{s^2}{1!}&0&0&0&0&0&0&0&\cdots\\[5pt]
\frac{r^7}{0!}&\frac{r^8}{1!}&\frac{r^9}{2!}&\frac{r^{10}}{3!}&\frac{r^{11}}{4!}&\frac{r^{12}}{5!}&\frac{r^{13}}{6!}&\frac{r^{14}}{7!}&R_6(r)&\frac{r^{16}}{9!}&\frac{r^{17}}{10!}&\cdots\\[5pt]
\frac{s^3}{3!}&\frac{s^3}{3!}&0&\frac{s^3}{3!}&0&1-\frac{s^3}{2!}&0&0&0&0&0&\cdots\\[5pt]
\frac{r^9}{0!}&\frac{r^{10}}{1!}&\frac{r^{11}}{2!}&\frac{r^{12}}{3!}&\frac{r^{13}}{4!}&\frac{r^{14}}{5!}&\frac{r^{15}}{6!}&\frac{r^{16}}{7!}&\frac{r^{17}}{8!}&\frac{r^{18}}{9!}&R_8(r)&\cdots\\[5pt]
\frac{s^4}{4!}&\frac{s^4}{4!}&0&\frac{s^4}{4!}&0&\frac{s^4}{4!}&0&1-\frac{s^4}{3!}&0&0&0&\cdots\\[5pt]
\frac{r^{11}}{0!}&\frac{r^{12}}{1!}&\frac{r^{13}}{2!}&\frac{r^{14}}{3!}&\frac{r^{15}}{4!}&\frac{r^{16}}{5!}&\frac{r^{17}}{6!}&\frac{r^{18}}{7!}&\frac{r^{19}}{8!}&\frac{r^{20}}{9!}&\frac{r^{21}}{10!}&\cdots\\[5pt]
\frac{s^5}{5!}&\frac{s^5}{5!}&0&\frac{s^5}{5!}&0&\frac{s^5}{5!}&0&\frac{s^5}{5!}&0&1-\frac{s^5}{5!}&0&\cdots\\[5pt]
\vdots&\vdots&\vdots&\vdots&\vdots&\vdots&\vdots&\vdots&\vdots&\vdots&\vdots&\ddots
\end{array}
\right)
\]
where $R_1(r)=1+\frac{r^{4}}{2}-r^{2}e^r$ and $R_n(r)=1+\frac{r^{2n+3}}{(n+2)!}-r^{n+1}e^r$ when $n\geq 2$ so that $A$ is row-stochastic.
It turns out that $A$ satisfies $XA<X$ as soon as $r$ and $s$ are small enough.
More precisely, $0<r<1/3$ and $0<s<1/10$ is a sufficient condition.
All this can be checked by patient but simple calculations.

\bibliographystyle{plainnat} 
\bibliography{paccaut}

\end{document}